\newcommand\R{\mathbb{R}}
\newcommand\bbS{\mathbb{S}}
\newcommand\bbM{\mathbb{M}}
\newcommand\cC{\mathcal{C}}
\newcommand\cD{\mathcal{D}}
\newcommand\cA{\mathcal{A}}
\newcommand\cT{\mathcal{T}}
\newcommand\cP{\mathcal{P}}
\newcommand\cV{\mathcal{V}}
\newcommand{\cF}{\mathcal{F}}
\newcommand\mt{\mathtt{t}}
\newcommand\bn{\boldsymbol{n}}
\newcommand\bF{\boldsymbol{f}}
\newcommand\bg{\boldsymbol{g}}
\newcommand\bu{\boldsymbol{u}}
\newcommand\bv{\boldsymbol{v}}
\newcommand\be{\boldsymbol{e}}
\newcommand\beps{\boldsymbol{\varepsilon}}
\newcommand\bsig{\boldsymbol{\sigma}}
\newcommand\bnabla{\boldsymbol{\nabla}}
\newcommand\btau{\boldsymbol{\tau}}
\newcommand\bze{\boldsymbol{\zeta}}
\newcommand\bet{\boldsymbol{\eta}}
\newcommand\bgam{\boldsymbol{\gamma}}

\newcommand\bdiv{\mathop{\mathbf{div}}\nolimits}
\newcommand{\jump}[1]{{\llbracket{#1}\rrbracket}}
\newcommand{\mean}[1]{\left\{\kern-1.ex\left\{ #1 \right\}\kern-1.ex\right\}}		

\documentclass[11pt]{article}
\usepackage[english]{babel}

\usepackage{textcomp}
\usepackage{lmodern}

\usepackage[a4paper,tmargin=1.35cm,bmargin=2.cm, rmargin=1.3cm,lmargin=1.3cm]{geometry}

\usepackage{amssymb, amsmath, mathtools, amsthm, amsfonts}

\usepackage{graphicx}
\usepackage[table]{xcolor}
\usepackage{booktabs, multirow}
\usepackage{stmaryrd}
\usepackage{enumitem}

\definecolor{lightblue}{rgb}{0.22,0.45,0.70}
\definecolor{lightgreen}{rgb}{0.22,0.50,0.25}
\usepackage[colorlinks=true,breaklinks=true,linkcolor=lightblue,citecolor=lightgreen,urlcolor=lightblue]{hyperref}
\usepackage[numbers]{natbib}

\allowdisplaybreaks

\DeclareMathOperator*{\esssup}{ess\,sup}

\DeclarePairedDelimiter\norm{\lVert}{\rVert}
\DeclarePairedDelimiter{\inner}{(}{)}
\DeclarePairedDelimiter{\set}{\{}{\}}
\DeclarePairedDelimiter{\dual}{\langle}{\rangle}
\newtheorem{remark}{Remark}[section]
\newtheorem{lemma}{Lemma}[section]
\newtheorem{theorem}{Theorem}[section]
\newtheorem{prop}{Proposition}[section]
\newtheorem{corollary}{Corollary}[section]
\newtheorem{assumption}{Assumption}[section]
\numberwithin{equation}{section}
\numberwithin{figure}{section}
\numberwithin{table}{section}

\xdefinecolor{mygrey}{rgb}{0.65,0.65,0.6375}
\newcolumntype{g}{ >{\columncolor{mygrey}} c }
\date{}
\title{Symmetric mixed discontinuous Galerkin methods for linear viscoelasticity\thanks{This research was  supported by Spain's Ministry of Economy Project PID2020-116287GB-I00, by the Monash Mathematics Research Fund S05802-3951284, and by the Australian Research Council through the Discovery Project grant DP220103160. }}

\author{ 
{\sc Salim Meddahi}\thanks{Facultad de Ciencias, Universidad de Oviedo, Federico Garc\'ia Lorca,  18, 33007-Oviedo, Spain, e-mail: {\tt salim@uniovi.es}.} 
\quad 
and 
\quad 
{\sc Ricardo Ruiz-Baier}\thanks{School of Mathematics, Monash University, 9 Rainforest Walk, Melbourne, Victoria 3800, Australia;  and   Universidad Adventista de Chile, Casilla 7-D, Chill\'an, Chile, e-mail: {\tt ricardo.ruizbaier@monash.edu}.}
}

\begin{document}

\maketitle

\begin{abstract}
\noindent We propose and rigorously analyse semi- and fully discrete discontinuous Galerkin methods for an initial and boundary value problem  describing inertial viscoelasticity in terms of elastic and viscoelastic stress components, and with mixed boundary conditions. The arbitrary-order spatial discretisation imposes strongly  the symmetry of the stress tensor, and it is combined with a Newmark trapezoidal rule as time-advancing scheme. We establish stability and convergence properties, and the theoretical findings are further confirmed via illustrative numerical simulations in 2D and 3D. 
  
\end{abstract}
\medskip

\noindent
\textbf{Mathematics Subject Classification:} 65M30, 65M12, 65M15, 74H15. 
\medskip

\noindent
\textbf{Keywords:} 
Mixed finite elements, linear viscoelasticity, stress-based formulation, error estimates.

\section{Introduction}  
\textbf{Scope and related work.}
Viscoelastic models can be used to describe a large class of conventional and unconventional materials with time-dependent mechanical behaviour, including polymers and elastomers, metals at high temperature, and, notably, some types of biological tissues (comprising extracellular matrix, cells, cell clusters, and so on) \cite{klatt}. The key constituents of collagen and elastin inherently exhibit viscoelastic and elastic characteristics (a viscoelastic material will eventually return to its original shape upon the removal of any deforming force). The combination of these two properties is observed in many other materials under mechanical loads  \cite{salencon}, and they are utilised in a wide range of applications to damp mechanical shocks, to attenuate resonant vibrations, and  to control noise propagation. Viscoelastic material laws are defined fitting tests of important phenomena such as creep compliance, rate-dependency of stress, hysteresis, and stress-relaxation. One of the simplest models  involving strain history in the constitutive equations is Zener's (or \emph{standard linear}) model in  viscoelasticity \cite{zener}, which is able to replicate creep-recovery and stress-relaxation \cite{salencon}. It consists of a spring and a Maxwell component in parallel, where, in turn, the Maxwell component is an assemblage of one spring and one viscous element (dashpot) in serial.  As in, e.g., \cite{Morro, gurtin}, based on the Boltzmann superposition principle, it is possible to use the Volterra integral form of the typical constitutive law for Zener's model  to eliminate the stress and formulate the viscoelastic system as an integro-differential problem written only in terms of displacement. Pure displacement formulations have been also studied from the numerical analysis viewpoint, and a number of contributions are available regarding continuous and discontinuous Galerkin methods including, for instance, \cite{fernandez11,jang21,riv1, riv2, Shaw1}. Even if the acceleration term endows the displacement with additional time regularity, the analysis is still far from trivial. 

The analysis  of viscoelasticity based only on a differential representation is also feasible. It suffices to consider a dual-mixed framework where the dual variable, the stress, enters the system together with the primal unknown (displacement).  To the authors' knowledge, this has been first proposed in \cite{Becache}, introducing a stress splitting  into elastic and viscoelastic contributions and focusing on first-order approximation of stress and using special grids. The approach was later extended in \cite{lee,rognes2,rognes,yuan} to high-order and to more general finite element discretisations. The analysis of these formulations is based on a dynamical system approach since the systems are of first-order type involving stress and velocity. In contrast, here we follow  the methods advanced in \cite{gmm-2020, MarquezMeddahi}, where one re-formulates the problem as a second-order hyperbolic PDE. This is achieved by using  the momentum balance to remove the acceleration, leading to a second-order in time of grad-div type written solely in terms of the Cauchy stress, which is separated into elastic and viscoelastic parts. 
 
 An important issue in the construction of stable mixed methods for elasticity is the preservation of symmetry for the Cauchy stress. Starting from the foundational work  \cite{ArnoldFalkWinther}, there has been an abundant body of developments in the design and analysis of conforming mixed finite elements on simplicial and rectangular meshes for both 2D and 3D; see, e.g., \cite{scotCockburn, ArnoldAwanouWinther, hu}. However, even if the simultaneous imposition of H(div)-conformity and strong symmetry of stress is possible, it typically entails a very large number of degrees of freedom and schemes that are not trivial to implement using standard finite element libraries. The usual ways to overcome this difficulty consist in either, (a) maintaining H(div)-conformity and relaxing the symmetry constraint, which comes at the expense of adding the rotation tensor as additional field variable playing  the role of a Lagrange multiplier enforcing the angular momentum conservation constraint, see for example \cite{ArnoldFalkWinther, CGG, GG} and the references therein;  or (b) to renounce H(div)-conformity and use non-conforming or DG approximations as in, e.g., \cite{ArnoldAwanouWintherNC, GGNC, Wu}.  
  
 Motivated  by the ability of DG methods to handle efficiently $hp$-adaptive strategies and to facilitate the  implementation of high order methods, we opt herein for an H(div)-based interior penalty method to solve the standard linear solid model of viscoelasticity. The space discretisation strategy amounts to approximate each stress component by symmetric tensors with piecewise polynomial entries of arbitrary degree $k\geq 1$, in 2D and 3D. We point out that our continuous formulation only asks the Cauchy stress (the sum of the elastic and viscoelastic stress components) to be H(div)-conforming. Consequently, the present approach only penalises the jumps of the normal total stress on the internal facets, therefore requiring fewer coupling conditions between the local degrees of freedom than penalising each stress component separately. We show that the resulting mixed DG semi-discrete scheme is robust and accurate for general domains and boundary conditions, and for heterogeneous media. Additionally, we prove that the fully discrete scheme relying on the classical second-order implicit Newmark method is  stable and  convergent. Finally, under piecewise regularity assumptions on the exact solution of the problem, we derive optimal asymptotic error estimates in a suitable H(div)-DG norm. 
   
\medskip 
\noindent\textbf{Outline.}
The contents of this paper have been organized in the following manner. The remainder of this section contains preliminary notational conventions and definition of useful functional spaces. Section~\ref{sec:model} presents the precise definition of Zener's model problem along with the derivation of its weak formulation in mixed form and recalling its unique solvability. Preliminary definitions and auxiliary tools needed for the  analysis in discontinuous finite dimensional spaces are collected in Section~\ref{sec:FE}. The precise definition and the analysis of convergence for a semi-discrete mixed method are detailed in Section~\ref{sec:semi-discrete}, and the fully discrete case is treated in Section~\ref{sec:fully-discrete}. Several numerical results are presented in Section~\ref{sec:results}, confirming the expected rates of convergence for different parameter sets including the nearly incompressible regime, and illustrating the use of the method in relatively simple problems of applicative relevance. 

\medskip 
\noindent\textbf{Recurrent notation and Sobolev spaces.} 
We denote the space of real matrices of order $d\times d$ by $\bbM$, and let $\bbS:= \set{\btau\in \bbM;\ \btau = \btau^{\mt} } $  be the subspace of symmetric matrices, where $\btau^{\mt}:=(\tau_{ji})$ stands for the transpose of $\btau = (\tau_{ij})$. The component-wise inner product of two matrices $\bsig, \,\btau \in\bbM$ is defined by $\bsig:\btau:= \sum_{i,j}\sigma_{ij}\tau_{ij}$. 

Let $D$ be a polyhedral Lipschitz bounded domain of $\R^d$ $(d=2,3)$, with boundary $\partial D$. Along this paper we convene to apply all differential operators row-wise.  Hence, given a tensorial function $\bsig:D\to \bbM$ and a vector field $\bu:D\to \R^d$, we set the divergence $\bdiv \bsig:D \to \R^d$, the   gradient $\bnabla \bu:D \to \bbM$, and the linearised strain tensor $\beps(\bu) : \Omega \to \bbS$ as
\[
(\bdiv \bsig)_i := \sum_j   \partial_j \sigma_{ij} \,, \quad (\bnabla \bu)_{ij} := \partial_j u_i\,,
\quad\hbox{and}\quad \beps(\bu) := \frac{1}{2}\left[\bnabla\bu+(\bnabla\bu)^{\mt}\right].
\]
 For $s\in \R$, $H^s(D,E)$ stands for the usual Hilbertian Sobolev space of functions with domain $D$ and values in E, where $E$ is either $\R$, $\R^d$ or $\bbM$. In the case $E=\R$ we simply write $H^s(D)$. The norm of $H^s(D,E)$ is denoted $\norm{\cdot}_{s,D}$ and the corresponding semi-norm $|\cdot|_{s,D}$, indistinctly for $E=\R,\R^d,\bbM$. We use the convention  $H^0(D, E):=L^2(D,E)$ and let $(\cdot, \cdot)_D$ be the inner product in $L^2(D, E)$, for $E=\R,\R^d,\bbM$, namely,
\begin{equation}\label{L2}
	(\bu, \bv)_D:=\int_D \bu\cdot\bv,\ \forall \bu,\bv\in L^2(D,\R^d),\quad  (\bsig, \btau)_D:=\int_D \bsig:\btau,\ \forall \bsig, \btau\in L^2(D,\bbM).
\end{equation}
The space of tensors in $L^2(D, \bbS)$ with divergence in $L^2(D, \R^d)$ is denoted $H(\bdiv, D, \bbS)$. We denote the corresponding norm $\norm{\btau}^2_{H(\bdiv,D)}:=\norm{\btau}_{0,D}^2+\norm{\bdiv\btau}^2_{0,D}$. Let $\bn$ be the outward unit normal vector to $\partial D$. The Green formula
\[
(\btau, \beps(\bv))_D + (\bdiv \btau, \bv)_D = \int_{\partial D} \btau\bn\cdot \bv\qquad  \forall \bv \in H^1(D,\R^d),
\] 
can be used to extend the normal trace operator $ \cC^\infty(\overline D, \bbS)\ni \btau \to (\btau|_{\partial \Omega})\bn$ to a linear continuous mapping $(\cdot|_{\partial \Omega})\bn:\, H(\bdiv, D, \bbS) \to H^{-\frac{1}{2}}(\partial D, \R^d)$, where $H^{-\frac{1}{2}}(\partial D, \R^d)$ is the dual of $H^{\frac{1}{2}}(\partial D, \R^d)$.

\medskip 
\noindent\textbf{Sobolev spaces for time dependent problems.} Since we will deal with a space-time domain problem, besides the Sobolev spaces defined above, we need to introduce spaces of functions acting on a bounded time interval $(0,T)$ and with values in a separable Hilbert space $V$, whose norm is denoted here by $\norm{\cdot}_{V}$. In particular, for $1 \leq p\leq \infty$, $L^p(V)$ is the space of classes of functions $f:\ (0,T)\to V$ that are B\"ochner-measurable and such that $\norm{f}_{L^p(V)}<\infty$, with 
\[
\norm{f}^p_{L^p(V)}:= \int_0^T\norm{f(t)}_{V}^p\, \text{d}t\quad \hbox{for $1\leq p < \infty$},
\quad\hbox{and} \quad \norm{f}_{L^\infty(V)}:= \esssup_{[0, T]} \norm{f(t)}_V.
\]
We use the notation $\mathcal{C}^0(V)$ for the Banach space consisting of all continuous functions $f:\ [0,T]\to V$. More generally, for any $k\in \mathbb{N}$, $\mathcal{C}^k(V)$ denotes the subspace of $\mathcal{C}^0(V)$ of all functions $f$ with (strong) derivatives $\frac{\text{d}^j f}{\text{d}t^j}$ in $\mathcal{C}^0(V)$ for all $1\leq j\leq k$. In what follows, we will use indistinctly the notations $\dot{f}:= \frac{\text{d} f}{\text{d}t}$ and $\ddot{f} := \frac{\text{d}^2 f}{\text{d}t^2} $ to express the first and second derivatives with respect to   $t$. Furthermore,   we consider the Sobolev space
\[
W^{1, \infty}(V):= \left\{f: \ \exists g\in L^\infty(V)
\ \text{and}\ \exists f_0\in V\ \text{such that}\
 f(t) = f_0 + \int_0^t g(s)\, \text{d}s\quad \forall t\in [0,T]\right\},
\]
and define the space $W^{k, \infty}(V)$ recursively  for all $k\in\mathbb{N}$.

Throughout this paper, we shall use the letter $C$ to denote a generic positive constant independent of the mesh size  $h$ and the time discretisation parameter $\Delta t$, that may stand for different values at its different occurrences. Moreover, given any positive expressions $X$ and $Y$ depending on $h$ and $\Delta t$, the notation $X \,\lesssim\, Y$  means that $X \,\le\, C\, Y$.

\section{A mixed variational formulation of the Zener model}\label{sec:model}
We aim to study  the dynamical equation of motion 
\[
\rho\ddot{\bu}-\bdiv\bsig =\bF \quad \text{in $\Omega\times (0, T]$},
\]
for a viscoelastic body represented by a polyhedral Lipschitz domain  $\Omega\subset \mathbb R^d$ ($d=2,3$). Here,  $\bu:\Omega\times[0, T] \to \R^d$ is the displacement field, $\bsig:\Omega\times[0,T]\to \bbS$ is the Cauchy stress tensor and $\bF:\Omega\times[0, T] \to \R^d$ represents the body force. We assume that the linearised strain tensor $\beps(\bu)$ is related to the stress tensor through Zener's constitutive law for viscoelasticity (see \cite{salencon}):
\begin{equation}\label{oldConstitutive}
	\bsig +  \omega \dot{\bsig} = \cC \beps(\bu) + \omega \cD \beps(\dot{\bu}) \quad \text{in $\Omega\times (0, T]$}.
\end{equation}

The symmetric and positive definite fourth-order tensors  $\cC$ and $\cD$ are such that $\cD - \cC$ is also positive definite in order to guarantee that the system is dissipative. We assume that the relaxation time $\omega\in L^\infty(\Omega)$ is positive and bounded away from zero:  $\omega(x) \geq \omega_0>0$ \textit{a.e.} in $\Omega$. Moreover, we assume that there exists a polygonal/polyhedral disjoint partition $\big\{\Omega_j,\ j= 1,\ldots,J\big\}$ of  $\bar \Omega$  such that $\rho|_{\Omega_j}:= \rho_j>0$ for all $j=1,\ldots,J$ and let $\rho^+:= \max_j \rho_j$ and $\rho^-:= \min_j \rho_j$. 

We assume mixed loading boundary conditions: the structure is clamped  ($\bu = \mathbf 0$) on $\Gamma_D  \times (0, T]$ where the boundary subset $\Gamma_D\subset\Gamma := \partial \Omega$ is of positive surface measure, and  free of stress ($\boldsymbol{\bsig}\bn = \mathbf 0$) on $ \Gamma_N  \times (0, T]$, where $\Gamma_N:= \Gamma \setminus \Gamma_D$. By $\bn$ we denote the exterior unit normal vector on $\Gamma$. Finally, we assume the initial conditions:
\begin{equation}\label{init1}
	\bu(0) = \bu_0 \quad\text{in $\Omega$}, \quad 
 \dot{\bu}(0) = \bu_1 \quad\text{in $\Omega$},   \quad \text{and} \quad 
 \bsig(0) = \bsig_0 \quad\text{in $\Omega$}.
\end{equation}

With the purpose of having the (total) stress tensor $\bsig$ as a primary unknown, we additively decompose this variable into a purely elastic component $\bgam:= \cC \beps(\bu)$ and a viscoelastic component $\bze:= \bsig - \bgam$, which allows us to deduce from the constitutive law \eqref{oldConstitutive} that  
\[
 \dot \bze + \tfrac{1}{\omega} \bze =  (\cD - \cC) \beps(\dot \bu).
\]
Hence, adopting the notations $\cA := \cC^{-1}$ and $\cV:= (\cD - \cC)^{-1}$, the model problem can be recast in terms of $\bu$, $\bgam$, and $\bze$, as follows
 \begin{align}\label{split123}
 \begin{split}
 \rho\ddot{\bu}-\bdiv(\bgam +  \bze) &=\bF  \quad\text{in $\Omega\times (0, T]$}, 
 \\[0.25ex]
 (\bgam +  \bze) &= (\bgam +  \bze)^{\mt} \quad \text{in $\Omega\times (0, T]$},
 \\[0.25ex]
 \cA \ddot{\bgam} &= \beps(\ddot{\bu})  \quad \text{in $\Omega\times (0, T]$}, 
 \\[0.25ex]
  \cV \ddot{\bze} +  \tfrac{1}{\omega}\cV \dot{\bze} &= \beps(\ddot{\bu}) \quad \text{in $\Omega\times (0, T]$},
 \\[0.25ex]
 \bu &= \mathbf 0 \quad  \text{on $\Gamma_D\times (0, T]$},
 \\[0.25ex]
 (\bgam + \bze)\bn &= \mathbf 0 \quad  \text{on $\Gamma_N\times (0, T]$}.
\end{split}
\end{align}
One readily notes that the material law splits now into two parts, one for each component of the total stress. The main  unknown consists in a pair of tensors $(\bgam, \bze)\in L^2(\Omega,\bbM\times \bbM)$ such that $\bgam +  \bze\in H(\bdiv,\Omega, \bbS)$. The  traction boundary condition on $\Gamma_N$ has to be included in an essential manner, for which we require the following closed subspace of $H(\bdiv, \Omega, \bbS)$  
\[
H_N(\bdiv, \Omega, \bbS) := \set*{\btau\in H(\bdiv, \Omega, \bbS); \quad 
	\left\langle\btau\bn,\bv\right\rangle_{\Gamma}= 0 
	\quad \text{$\forall\bv\in H^{1/2}(\partial\Omega,\R^d)$,\, $\bv|_{\Gamma_D} = \mathbf{0}$}},
\]
where $\dual{\cdot, \cdot}_\Gamma$ holds for the duality pairing between $H^{1/2}(\Gamma,\R^d)$ and $H^{-1/2}(\Gamma,\R^d)$. We then consider the energy space  
\[
 \mathcal{H}^+_{\text{sym}}:= 
\Big\{ (\bet, \btau) \in L^2(\Omega,\bbM\times \bbM):\ \bet + \btau \in H_N(\bdiv,\Omega, \bbS)  \Big\},
\]
endowed with the Hilbertian norm  
\[
\norm*{(\bet, \btau)}^2_{\mathcal{H}^+_{\text{sym}}}:=  \norm*{\bet}^2_{0,\Omega} + \norm*{\btau}^2_{0,\Omega}  + \norm*{\bdiv (\bet +  \btau) }^2_{0,\Omega}. 
\] 

In what follows, when $D = \Omega$ in \eqref{L2}, we simply denote the $L^2$-inner product by $\inner*{\cdot,\cdot}$. We  consider an arbitrary $(\bet, \btau)\in \mathcal{H}^+_{\text{sym}}$, test the third and fourth rows of \eqref{split123} with $\bet$ and $\btau$ and add the resulting equations  to get
\begin{equation}\label{const+}
(\cA \ddot{\bgam},\bet) + \Big( \cV (\ddot{\bze} + \tfrac{1}{\omega}\dot{\bze}) , \btau\Big)  = \Big(\beps(\ddot{\bu}), \bet + \btau\Big) = \Big(\bnabla \ddot\bu , \bet + \btau\Big),
\end{equation} 
where the last identity follows from the symmetry of $\bet + \btau$. 
Next, we integrate by parts in the right-hand side of \eqref{const+} and take into account the boundary conditions on $\Gamma_N\times (0, T]$ to obtain
\begin{equation}\label{var0}
(\cA \ddot{\bgam},\bet) + \Big( \cV (\ddot{\bze} + \tfrac{1}{\omega}\dot{\bze}) , \btau\Big)  =   
- \Big(\ddot{\bu}, \bdiv(\bet + \btau)\Big).
\end{equation} 
Substituting back $\ddot{\bu} =  \rho^{-1}\big(\bF + \bdiv(\bgam + \bze)\big)$ into \eqref{var0} yields  
\begin{equation*}
A\Big( (\ddot\bgam ,  \ddot\bze), (\bet ,   \btau) \Big)  + (\tfrac{1}{\omega}\cV\dot\bze,\btau)
+ \Big(\tfrac{1}{\rho}\bdiv(\bgam +  \bze) , \bdiv(\bet +  \btau) \Big)  = - \Big(\tfrac{1}{\rho}\bF , \bdiv(\bet +  \btau) \Big).
\end{equation*}
for all $(\bet, \btau )\in \mathcal{H}^+_{\text{sym}}$, 
where    
\[
A\Big( (\bgam, \bze) ,(\bet, \btau)\Big) := (\cA\bgam, \bet) + (\cV \bze, \btau), \quad  (\bgam, \bze),\, (\bet, \btau) \in L^2(\Omega,\bbM\times \bbM).
\]
It is important to notice that, as a consequence of our hypotheses on $\cC$ and $\cD$, the bilinear form $A$ is symmetric, bounded and coercive, i.e., there exist positive constants $M$ and $\alpha$, depending only on $\cC$ and $\cD$, such that
\begin{align}\label{contA}
\Big| A\Big( (\bgam, \bze) ,(\bet, \btau)\Big) \Big| & \leq M \| (\bgam, \bze) \|_{0,\Omega} 
\|(\bet, \btau) \|_{0,\Omega} \qquad \forall \, (\bgam, \bze) ,(\bet, \btau) \in L^2(\Omega,\bbM\times \bbM), \\
\label{ellipA}
 A\Big( (\bet, \btau),(\bet, \btau) \Big) & \geq \alpha  
\|(\bet, \btau) \|^2_{0,\Omega} \qquad \forall  \,(\bet, \btau) \in L^2(\Omega,\bbM\times \bbM).
\end{align}

We let
\[ 
	\mathfrak{L}_{\text{sym}}^2 := \set*{ (\bet, \btau) \in L^2(\Omega,\bbM\times \bbM);\ \bet +  \btau\in L^2(\Omega, \bbS)},
	\]
and consider the following variational formulation of \eqref{split123}:  Given $\bF\in L^{2}(L^2(\Omega,\R^d))$, we look for $(\bgam, \bze)\in \cC^0(\mathcal{H}^+_{\text{sym} })\cap\cC^1( \mathfrak{L}_{\text{sym}}^2 )$ satisfying,  for all $(\bet, \btau) \in \mathcal{H}^+_{\text{sym} }$, 
\begin{align}\label{varFormR1-varFormR2}
\frac{\text{d}}{\text{d}t} \Big\{ A\Big( (\dot\bgam ,  \dot\bze), (\bet ,   \btau) \Big)  + (\tfrac{1}{\omega}\cV\bze, \btau)\Big\} 
+ \Big(\tfrac{1}{\rho}\bdiv(\bgam + \bze) , \bdiv(\bet +  \btau) \Big)  &= 
 - \Big(\tfrac{1}{\rho}\bF , \bdiv(\bet +  \btau) \Big) ,  
\nonumber\\
(\bgam(0), \bze(0)) = (\bgam_0, \bze_0), \qquad (\dot\bgam(0), \dot\bze(0)) &= (\bgam_1, \bze_1),
\end{align} 
 where $(\bgam_0, \bze_0)\in \mathcal{H}^+_{\text{sym}}$ and $(\bgam_1, \bze_1) \in \mathfrak{L}_{\text{sym}}^2$ are given by 
\begin{align*}
		\bgam_0 &:= \cC \beps(\bu_0),\quad   \bze_0 = \bsig_0 - \bgam_0,
\quad 
\bgam_1 := \cC \beps(\bu_1), \quad\bze_1 :=  \cD \beps(\bu_1) - \bgam_1 - \tfrac{1}{\omega}\bze_0.
\end{align*}  

Classical techniques for  second order evolution problems  with energy methods \cite{Evans, Renardy} have been successfully applied in \cite[Theorem 5.2, Lemma 5.1]{gmm-2020} to prove  the well-posedness of   \eqref{varFormR1-varFormR2}. 

\begin{theorem}\label{theorem-R1-R2}
Assume that $\bF \in H^1(L^2(\Omega,\R^d))$. Then,  problem \eqref{varFormR1-varFormR2} admits a unique solution. Moreover, there exists a constant $C>0$  such that,
\[
\max_{t\in [0,T]}\norm*{(\bgam, \bze)(t)}_{\mathcal{H}^+_{\text{sym}}} + \max_{t\in [0,T]}\norm*{(\dot\bgam, \dot\bze)(t)}_{0,\Omega} \leq C \Big(  \norm*{\bF}_{H^1(L^2(\Omega,\R^d))} 
+ \norm*{(\bgam_0, \bze_0)}_{\mathcal{H}^+_{\text{sym}}} + \norm*{(\bgam_1, \bze_1)}_{0,\Omega}\Big). 
\]
\end{theorem}
\section{Finite element spaces and auxiliary results}\label{sec:FE}
We consider  a sequence $\{\mathcal{T}_h\}_h$ of shape regular meshes that subdivide the domain $\bar \Omega$ into  triangles/tetrahedra $K$ of diameter $h_K$. The parameter $h:= \max_{K\in \cT_h} \{h_K\}$ represents the mesh size of $\cT_h$.  We assume that $\mathcal{T}_h$ is aligned with the partition $\bar\Omega = \cup_{j= 1}^J \bar{\Omega}_j$ and that $\cT_h(\Omega_j) := \set*{K\in \cT_h;\ K\subset \Omega_j }$ is a shape regular mesh of $\bar\Omega_j$ for all $j=1,\cdots, J$ and  all $h$.
For all $s\geq 0$, we consider the broken Sobolev space    
\[
  H^s(\cup_j\Omega_j) := \set*{ v\in L^2(\Omega);\ v|_{\Omega_j}\in H^s(\Omega_j),\ \forall j =1,\ldots,J }
\]
corresponding to the partition $\bar\Omega = \cup_{j= 1}^J \bar{\Omega}_j$. 
Its vectorial and tensorial versions are denoted $H^s(\cup_j\Omega_j,\R^d)$ and $H^s(\cup_j\Omega_j,\bbM)$, respectively. Similarly, the broken Sobolev space with respect to the subdivision of $\bar \Omega$ into $\cT_h$ is  
\[
 H^s(\cT_h,E):=
 \set{\bv \in L^2(\Omega, E): \quad \bv|_K\in H^s(K, E)\quad \forall K\in \cT_h},\quad \text{for $E \in \set{ \R, \R^d, \bbM}$}. 
\]
For each $\bv:=\set{\bv_K}\in H^s(\cT_h,\R^d)$ and $\btau:= \set{\btau_K}\in H^s(\cT_h,\bbM)$ the components $\bv_K$ and $\btau_K$  represent the restrictions $\bv|_K$ and $\btau|_K$. When no confusion arises, the  subscripts will be dropped.

Hereafter, given an integer $m\geq 0$ and a domain $D\subset \mathbb{R}^d$, $\cP_m(D)$ denotes the space of polynomials of degree at most $m$ on $D$. We introduce the space   
\[
 \cP_m(\cT_h) := 
 \set{ v\in L^2(\Omega): \ v|_K \in \cP_m(K),\ \forall K\in \cT_h }
 \]
 of piecewise polynomial functions relatively to $\cT_h$. We also consider the space $\cP_m(\cT_h,E)$ of functions with values in $E$ and entries in $\cP_m(\cT_h)$, where $E$ is either $\R^d$, $\bbM$ or $\bbS$. 
 
Let us introduce now notations related to DG approximations of $H(\text{div})$-type spaces. We say that a closed subset $F\subset \overline{\Omega}$ is an interior edge/face if $F$ has a positive $(d-1)$-dimensional measure and if there are distinct elements $K$ and $K'$ such that $F =\bar K\cap \bar K'$. A closed subset $F\subset \overline{\Omega}$ is a boundary edge/face if there exists $K\in \cT_h$ such that $F$ is an edge/face of $K$ and $F =  \bar K\cap \partial \Omega$. We consider the set $\cF_h^0$ of interior edges/faces, the set $\cF_h^\partial$ of boundary edges/faces. 
We assume that  $\cF_h^\partial$ is compatible with the partition $\partial \Omega = \Gamma_D \cup \Gamma_N$ in the sense that, if 
$
\cF_h^D = \set*{F\in \cF_h^\partial:\, F\subset \Gamma_D}$  and $\cF_h^N = \set*{F\in \cF_h^\partial:\, F\subset \Gamma_N},
$
then $\Gamma_D = \cup_{F\in \cF_h^D} F$ and $\Gamma_N = \cup_{F\in \cF_h^N} F$. We denote   
\[
  \cF_h := \cF_h^0\cup \cF_h^\partial\qquad \text{and} \qquad \cF^*_h:= \cF_h^{0} \cup \cF_h^{N},
\]
and  we introduce the set $\cF(K):= \set{F\in \cF_h;\, F\subset \partial K}$ of edges/faces composing the boundary of $K\in \cT_h$.
 We will need the space (given on the skeletons of the triangulations $\cT_h$)  $L^2(\cF^*_h):= \bigoplus_{F\in \mathcal{F}^*_h} L^2(F)$. Its vector valued version is denoted $L^2(\cF^*_h,\R^d)$. Here again, the components $\bv_F$ of $\bv := \set{\bv_F}\in L^2(\cF^*_h,\R^d)$  coincide with the restrictions $\bv|_F$.  We endow $L^2(\cF^*_h,\R^d)$ with the inner product 
\[
(\bu, \bv)_{\cF^*_h} := \sum_{F\in \cF^*_h} \int_F \bu_F\cdot \bv_F\quad \forall \bu,\bv\in L^2(\cF^*_h,\R^d)
\]
and denote the corresponding norm $\norm*{\bv}^2_{0,\cF^*_h}:= (\bv,\bv)_{\cF^*_h}$. From now on, $h_\cF\in L^2(\cF_h)$ is the piecewise constant function defined by $h_\cF|_F := h_F$ for all $F \in \cF^*_h$ with $h_F$ denoting the diameter of edge/face $F$. 

Given  $\bv\in H^s(\cT_h,\R^d)$ and $\btau\in H^s(\cT_h,\bbM)$, with $s>1/2$, we define averages $\mean{\bv}\in L^2(\cF^*_h,\R^d)$ and jumps $\jump{\btau}\in L^2(\cF^*_h,\R^d)$ by
\[
 \mean{\bv}_F := (\bv_K + \bv_{K'})/2 \quad \text{and} \quad \jump{\btau}_F := 
 \btau_K \bn_K + \btau_{K'}\bn_{K'} 
 \quad \forall F \in \cF(K)\cap \cF(K'),
\]
with the conventions 
\[
 \mean{\bv}_F := \bv_K  \quad \text{and} \quad \jump{\btau}_F := 
 \btau_K \bn_K  
 \quad \forall F \in \cF(K),\,\, F\in \cF_h^N,
\]
where $\bn_K$ is the outward unit normal vector to $\partial K$.

For any $k\geq 1$, we consider $\mathcal{H}^{DG}_{\text{sym},h} := \cP_{k}(\cT_h,\bbS) \times \cP_{k}(\cT_h,\bbS)$ and let $\mathcal{H}^+_{\text{sym}}(h) :=\mathcal{H}^+_{\text{sym}} + \mathcal{H}^{DG}_{\text{sym},h}$.  Given $(\bet, \btau) \in \mathcal{H}^+_{\text{sym}}(h)$, we define $\bdiv_h (\bet + \btau) \in  L^2(\Omega,\R^d)$ by $\bdiv_h (\bet +  \btau)|_{K} := \bdiv (\bet_K + \btau_K)$ for all $K\in \cT_h$ and endow $\mathcal{H}^+_{\text{sym}}(h)$ with the norm
\begin{equation}\label{norm:sym}
 \norm*{(\bet, \btau)}^2_{\mathcal{H}^+_{\text{sym}}(h)} := \norm*{(\bet, \btau)}^2_{0,\Omega} + \norm*{\bdiv_h (\bet + \btau)}^2_{0,\Omega} + \norm*{h_{\cF}^{-1/2} \jump{\bet + \btau}}^2_{0,\cF^*_h}.
\end{equation}

We end this section by recalling  technical results that will be needed in what follows. We begin with the well-known trace inequality, see for example \cite[Proposition 4.1]{conThan}.
\begin{lemma}
There exists a constant $C>0$ independent of $h$ such that 
\begin{equation}\label{multiplicativetrace}
		h_K^{1/2}\norm{v}_{0,\partial K} \leq C \big( \norm{v}_{0,K} + h_K \norm*{\nabla v}_{0,K} \big),  
	\end{equation}
	for all $v\in H^1(K)$ and all $K\in \cT_h$. 
\end{lemma}
It is easy to deduce from \eqref{multiplicativetrace} the following discrete trace inequality (see also \cite[Proposition 4.1]{conThan})
\begin{equation}\label{discTrace}
  \norm*{h^{1/2}_{\cF}\mean{v}}_{0,\cF^*_h}\leq C_{\textup{tr}} \norm*{v}_{0,\Omega}\quad \forall  v\in \cP_k(\cT_h). 
 \end{equation}
The Scott--Zhang like quasi-interpolation operator $\Pi_h:\, L^2(\Omega) \to \cP_{k}(\cT_h)\cap H^1(\Omega)$, obtained in \cite{ern} by applying an $L^2$-orthogonal projection onto $\cP_{k}(\cT_h)$ followed by an averaging procedure with continuous and piecewise $\cP_{k}$ range,   will be especially useful in the forthcoming analysis. We recall in the next lemma the local approximation properties provided in \cite[Theorem 5.2]{ern}. Let us first introduce some notations. For any $K\in \cT_h$, we introduce the subset of $\cT_h$ defined by  $\cT_h^K := \set*{K'\in\cT_h:\, K\cap K' \neq \emptyset}$ and let $D_K = \text{interior}\left(\cup_{K'\in \cT_h^K} K' \right)$.   

\begin{lemma}
	The quasi-interpolation operator $\Pi_h$ is invariant in the space $\cP_{k}(\cT_h)\cap H^1(\Omega)$ and there exists a constant $C>0$ independent of $h$ such that 
	\begin{equation}\label{scott0}
		|v - \Pi_h v |_{m,K} \leq C h_K^{r - m} |v|_{r,D_K},
	\end{equation}
for all real numbers $0\leq r \leq k+1$, all natural numbers $0\leq m \leq [r]$, all $v\in H^r(D_K)$, and all $K\in \cT_h$. Here $[r]$ stands for the the largest integer less than or equal to $r$. 
\end{lemma}
We point out that, as a consequence of \eqref{scott0} and the triangle inequality, it holds
\begin{equation}\label{scott00}
		|\Pi_h v|_{m,K} \lesssim |v|_{m,D_K},
	\end{equation}
	for all natural numbers $0\leq m \leq k+1$,  all $v\in H^m(D_K)$, and all $K\in \cT_h$. 	 Moreover, it is straightforward to deduce from \eqref{scott0} and the multiplicative trace inequality \eqref{multiplicativetrace} that 
	\begin{equation}\label{scott0b}
		h_K^{1/2}\norm{v - \Pi_h v}_{0,\partial K} + h_K^{3/2}\norm{\nabla ( v - \Pi_h v )}_{0,\partial K} \lesssim h_K^r |v|_{r,D_K},
	\end{equation}
	for all $2 \leq r \leq k+1$ ($k\geq 1$), all $v\in H^r(D_K)$ and all $K\in \cT_h$. 

We infer from \eqref{scott00} a global stability property for $\Pi_h$ on $H^m(\cup_j\Omega_j)$,  $0\leq m \leq k+1$, by taking advantage of the fact that the cardinal $\#(\cT_h^K)$ of $\cT_h^K$ is uniformly bounded for all $K\in \cT_h$ and all $h$, as a consequence of the shape-regularity of the mesh sequence $\{\cT_h\}$. Indeed, given $v\in H^m(\cup_j\Omega_j)$, we let $ \cT_h^K(\Omega_j) := \set*{K'\in\cT_h(\Omega_j):\, K\cap K' \neq \emptyset}$ be the subset of elements in $\cT_h^K$ that are contained in $\bar\Omega_j$ and denote $D_K^j := \text{interior}\left(\cup_{K'\in \cT_h^K(\Omega_j)} K' \right)$. It follows from \eqref{scott00} that 
\[
		\norm{\Pi_h v}_{m,K} \lesssim \norm{v}_{m,D_K^j},  \quad \forall v\in H^m(D^j_K), \,\, 0\leq m \leq k+1,\quad  \forall K\in \cT_h^K(\Omega_j).
\]
Summing over $K\in \cT_h^K(\Omega_j)$  and using that $\#(\cT_h^{K}(\Omega_j)) \leq \#(\cT_h^{K}) \leq c$ for all $ j \leq J$ and all $h$, we deduce that 
\begin{equation}\label{scott0a}
		\norm{\Pi_h v}_{m,\Omega_j} \lesssim  \norm{v}_{m,\Omega_j},\quad  \forall j=1,\ldots, J.
\end{equation}
Finally, it follows from a successive application of the discrete trace inequality \eqref{discTrace} and the stability estimate \eqref{scott0a} for $m=1$ that 
\begin{equation}\label{scott0c}
	\norm*{h^{1/2}_{\cF}\mean{\nabla \Pi_h v}}_{0,\cF^*_h}\leq C_{\textup{tr}} \sum_{j=1}^J |\Pi_h v|_{1,\Omega_j} \lesssim  \sum_{j=1}^J  \norm*{v}_{1,\Omega_j}  \quad \forall  v\in H^1(\cup_j\Omega_j).
\end{equation}
 
 In the sequel, we use the same notation for the tensorial version $\Pi_h:\, L^{2}(\Omega,\bbS)\to \cP_{k}(\cT_h,\bbS)\cap H^1(\Omega,\bbS)$ of the quasi-interpolation operator, which is obtained by applying the scalar operator componentwise. It is important to notice that such an operator preserves tensor symmetry. As a consequence of \eqref{scott0} and \eqref{scott0b} we have the following result.

\begin{lemma}\label{maintool}
There exists a constant $C>0$ independent of $h$ such that 
\begin{multline}\label{tool}
	\norm*{(\bet - \Pi_h \bet, \btau - \Pi_h \btau)}_{0,\Omega} + \norm*{\bdiv\big( (\bet + \btau) - \Pi_h(\bet + \btau) \big)}_{0,\Omega} 
	\\
	+ \norm*{h_F^{\frac{1}{2}} \mean{\bdiv\big((\bet + \btau) - \Pi_h(\bet + \btau)\big)}}_{0,\cF_h^*} \leq 
	C h^{\min\{r,k\}} \sum_{j=1}^J\Big( \norm*{(\bet, \btau)}_{r,\Omega_j} + \norm*{\bet + \btau}_{r+1,\Omega_j} \Big),
\end{multline}
	for all $(\bet, \btau )\in \mathcal{H}^+_{\text{sym}}$ such that $(\bet, \btau )\in  H^r(\cup_j\Omega_j,\bbM\times \bbM)$ and $\bet +\btau \in   H^{r+1}(\cup_j\Omega_j,\mathbb R^d)$, $r\geq 1$.
\end{lemma}
\begin{proof}
Given $K\in \cT_h(\Omega_j)$, $1\leq j \leq J$, we obtain from  \eqref{scott0} that  
\begin{equation}\label{esti1}
\norm*{(\bet - \Pi_h \bet, \btau - \Pi_h \btau)}^2_{0,K}\lesssim h_K^{2\min\{r,k+1\}} ( \norm*{\bet}^2_{r,D^j_K} + \norm*{\btau}^2_{r,D^j_K}),
\end{equation}
 and 
 \begin{align}\label{esti2}
 \norm*{\bdiv\big( (\bet + \btau) - \Pi_h(\bet + \btau) \big)}^2_{0,K} &\lesssim 
\norm*{\bnabla\big( (\bet + \btau) - \Pi_h(\bet + \btau) \big)}^2_{0,K}
\lesssim h_K^{2\min\{r,k\}}  \norm*{\bet + \btau}^2_{r+1,D^j_K}.
\end{align} 
For the last term in the left-hand side of \eqref{tool}, we notice that 
\[
  \norm*{h_F^{\frac{1}{2}} \mean{\bdiv\big((\bet + \btau) - \Pi_h(\bet + \btau)\big)}}^2_{0,\cF_h^*}   
  \lesssim    
  \sum_{j=1}^J \sum_{K\in \cT_h(\Omega_j)}  h_K\norm*{ \bnabla \Big( (\bgam + \bze) -    \Pi_h(\bgam + \bze)\Big) }^2_{0,\partial K},
\]
and \eqref{scott0b} yields  
 \begin{equation}\label{esti3}
 h_K\norm*{ \bnabla \Big( (\bgam + \bze) -    \Pi_h(\bgam + \bze)\Big) }^2_{0,\partial K}    
 \lesssim h_K^{2\min\{r,k\}} \norm*{\bgam + \bze}^2_{r+1,D_K^j}, \quad \forall K\in \cT_h(\Omega_j).
\end{equation}
Summing \eqref{esti1}, \eqref{esti2} and \eqref{esti3} over $K\in \cT_h(\Omega_j)$  and then over $j=1,\ldots, J$ and invoking the shape-regularity of the mesh sequence give the result. 
\end{proof}

\section{The semi-discrete DG problem and its convergence analysis}\label{sec:semi-discrete}

Our DG scheme requires the external force $\bF$ to have traces at the interelement boundaries of the mesh $\cT_h$. In order to make hypotheses on the data that are realistic  from the practical point of view, we will only assume that $\bF$ is piecewise smooth relatively to the partition $\{\Omega_j,\, j = 1, \ldots , J\}$ of $\bar\Omega$. 

\begin{assumption}\label{assumption1}
	The body force satisfies  $\bF\in W^{1,\infty}(H^1(\cup_j \Omega_j, \R^d))$.	  
\end{assumption}

We are now in a position to introduce the following discontinuous Galerkin semi-discretisation of  \eqref{varFormR1-varFormR2}: Find $(\bgam_h, \bze_h)\in \cC^2(\mathcal{H}^{DG}_{\text{sym},h})$  solution of 
\begin{align}\label{varFormR1-varFormR2-h}
\nonumber
 &A\Big( (\ddot\bgam_h ,  \ddot\bze_h), (\bet ,   \btau) \Big)  + (\tfrac{1}{\omega}\cV\dot\bze_h,\btau) 
+ \Big(\tfrac{1}{\rho}\bdiv_h(\bgam_h +  \bze_h) , \bdiv_h(\bet +  \btau) \Big)  
\\
&\quad - \Big(\mean{\tfrac{1}{\rho} \bdiv_h (\bgam_h + \bze_h)}, \jump{\bet_h +  \btau_h}\Big)_{\cF^*_h}
- \Big(\mean{\tfrac{1}{\rho} \bdiv_h (\bet_h +  \btau_h)}, \jump{\bgam_h + \bze_h}\Big)_{\cF^*_h}
\nonumber\\
&\qquad \qquad + \Big(\texttt{a} h_\cF^{-1}\jump{\bgam_h + \bze_h}, \jump{\bet_h +  \btau_h} \Big)_{\cF^*_h}
 = -\Big(\tfrac{1}{\rho}\bF,  \bdiv_h (\bet_h +  \btau_h) \Big)
 + \Big(\mean{\tfrac{1}{\rho} \bF}, \jump{\bet_h +  \btau_h}\Big)_{\cF^*_h}, \quad
 \end{align}
for all $(\bet_h, \btau_h)\in \mathcal{H}^{DG}_{\text{sym},h}$, where $\texttt{a}>0$ is a given, sufficiently large parameter. 
 
We assume that  the solution $(\bgam_h(t),\bze_h(t))$ of problem \eqref{varFormR1-varFormR2-h} is started up with the initial conditions 
\begin{equation}\label{initial-R1-R2-h*c}
(\bgam_h(0),\bze_h(0))= (\Pi_h\bgam_0, \Pi_h\bze_0),
\quad  (\dot\bgam_h(0),\dot\bze_h(0))= (\Pi_h\bgam_1, \Pi_h\bze_1).
\end{equation}
In this way, the projected error $(\be^h_{\bgam}(t),  \be^h_{\bze}(t)) := (\Pi_h\bgam - \bgam_h, \Pi_h\bze - \bze_h)(t)$  satisfies, by construction,  the vanishing initial conditions: 
\begin{equation*}
(\be^h_{\bgam},  \be^h_{\bze})(0)=(\mathbf 0,\mathbf 0)\quad \text{and} \quad  (\dot\be^h_{\bgam},  \dot\be^h_{\bze})(0)= (\mathbf 0,\mathbf 0).
\end{equation*}
 
The convergence analysis of the semi-discrete problem \eqref{varFormR1-varFormR2-h} requires the following time regularity assumptions on the solution $(\bgam, \bze)$ of \eqref{varFormR1-varFormR2} that are not guaranteed by Theorem~\ref{theorem-R1-R2}.

\begin{assumption}\label{assumption2}
	The solution $(\bgam, \bze)$ of \eqref{varFormR1-varFormR2} satisfies 
	\begin{enumerate}[label=\roman*)]
	\item $(\bgam,\bze)\in W^{3,\infty}(L^2(\Omega),\bbS\times \bbS)$, 
	\item and $\bdiv(\bgam+\bze)\in W^{1,\infty}(L^2(\Omega),\R^d)$.
	\end{enumerate}
\end{assumption}

We begin by verifying that the DG scheme \eqref{varFormR1-varFormR2-h} is consistent with problem \eqref{varFormR1-varFormR2}.

\begin{prop}\label{consistency}
Under Assumptions~\ref{assumption1} and \ref{assumption2}(i), the solution $(\bgam, \bze)$ of \eqref{varFormR1-varFormR2} satisfies the identity  
 \begin{align}\label{consistent}
\nonumber
 &A\Big( (\ddot\bgam ,  \ddot\bze), (\bet_h ,   \btau_h) \Big)  + (\tfrac{1}{\omega}\cV\dot\bze,\btau_h)  
+ \Big(\tfrac{1}{\rho}\bdiv(\bgam +  \bze) , \bdiv_h(\bet_h +  \btau_h) \Big)  
\\
&\quad - \Big(\mean{\tfrac{1}{\rho} \bdiv (\bgam + \bze)}, \jump{\bet_h +  \btau_h}\Big)_{\cF^*_h}
- \Big(\mean{\tfrac{1}{\rho} \bdiv_h (\bet_h + \btau_h)}, \jump{\bgam + \bze}\Big)_{\cF^*_h}
 \nonumber\\
&\qquad + \Big(\texttt{a} h_\cF^{-1}\jump{\bgam + \bze}, \jump{\bet_h + \btau_h} \Big)_{\cF^*_h}
 = -\Big(\tfrac{1}{\rho}\bF,  \bdiv_h (\bet_h +  \btau_h) \Big) 
 + \Big(\mean{\tfrac{1}{\rho} \bF}, \jump{\bet_h + \btau_h}\Big)_{\cF^*_h},
 \end{align}
for all $(\bet_h, \btau_h)\in \mathcal{H}^{DG}_{\text{sym},h}$. 
\end{prop}

\begin{proof}
Let us first notice that the acceleration field $\ddot\bu =  \tfrac{1}{\rho} \big(  \bdiv (\bgam + \bze) + \bF \big)\in L^\infty(L^2(\Omega,\R^d))$ satisfies, by virtue of Assumption~\ref{assumption2}(i), $\beps(\ddot \bu) = \cA \ddot\bgam\in L^\infty(L^2(\Omega,\bbS))$. It follows from the boundary condition on $\Gamma_D$ and Korn's inequality that $\ddot \bu \in L^\infty(H^1(\Omega,\R^d))$. Hence, using that $\jump{\bgam + \bze} = \mathbf 0$ yields 
\begin{align}\label{consistent0}
\nonumber
 &A\Big( (\ddot\bgam ,  \ddot\bze), (\bet_h ,   \btau_h) \Big)  + (\tfrac{1}{\omega}\cV\dot\bze,\btau_h)  
+ \Big(\tfrac{1}{\rho}\bdiv(\bgam +  \bze) , \bdiv_h(\bet_h +  \btau_h) \Big)  
\\
\nonumber&\quad - \Big(\mean{\tfrac{1}{\rho} \bdiv (\bgam + \bze)}, \jump{\bet_h +  \btau_h}\Big)_{\cF^*_h}
- \Big(\mean{\tfrac{1}{\rho} \bdiv_h (\bet_h + \btau_h)}, \jump{\bgam + \bze}\Big)_{\cF^*_h}
\\
\nonumber&\qquad + \Big(\texttt{a} h_\cF^{-1}\jump{\bgam + \bze}, \jump{\bet_h +  \btau_h} \Big)_{\cF^*_h} =   A\Big( (\ddot\bgam ,  \ddot\bze), (\bet_h ,   \btau_h) \Big)  + (\tfrac{1}{\omega}\cV\dot\bze,\btau_h) 
+ \Big( \ddot\bu,  \bdiv_h (\bet_h +  \btau_h) \Big)
\\
 &\qquad \quad  \qquad \quad  
 - \Big(\ddot\bu , \jump{\bet_h +  \btau_h}\Big)_{\cF^*_h} 
 -\Big( \tfrac{1}{\rho} \bF,  \bdiv_h (\bet_h +  \btau_h) \Big) + \Big(\mean{ \tfrac{1}{\rho} \bF}, \jump{\bet_h +  \btau_h}\Big)_{\cF^*_h},
\end{align}
 for all $(\bet_h, \btau_h)\in \mathcal{H}^{DG}_{\text{sym},h}$. Now, taking into account that ($\bu|_{\Gamma_D} = \mathbf 0$),
 \begin{align*}
 	\Big(\ddot\bu , \jump{\bet_h +  \btau_h}\Big)_{\cF^*_h} &= \Big(\ddot\bu , \jump{\bet_h +  \btau_h}\Big)_{\cF_h}  
 	= \sum_{K\in \cT_h}\int_{\partial K} \ddot{\bu}\cdot (\bet_h +  \btau_h)\bn_K 
 	\\
 	&= \sum_{K\in \cT_h}  \Big( \inner*{\beps(\ddot \bu), \bet_h +  \btau_h}_K + \inner*{\ddot{\bu},  \bdiv (\bet_h +  \btau_h)}_K \Big) ,
 \end{align*}
 and keeping in mind \eqref{const+}, we deduce that  
 \begin{align*}
 \Big( \ddot\bu,  \bdiv_h (\bet_h +  \btau_h) \Big) - \Big(\ddot\bu , \jump{\bet_h +  \btau_h}\Big)_{\cF^*_h} &= - \Big(\beps(\ddot \bu), \bet_h +  \btau_h\Big) 
 = -A\Big( (\ddot\bgam ,  \ddot\bze), (\bet_h ,   \btau_h) \Big)  - (\tfrac{1}{\omega}\cV\dot\bze, \btau_h).
 \end{align*}
 Substituting back the last identity in \eqref{consistent0} gives the sought consistency result. 
\end{proof}

Let us  consider the splitting $\big(\bgam -\bgam_h, \bze - \bze_h\big) = \big(\mathcal I_{\bgam}^h, \mathcal I_{\bze}^h\big) + \big(\be^h_{\bgam}, \be^h_{\bze}\big)$ defined by $\big(\be^h_{\bgam}, \be^h_{\bze}\big)(t) = \big(\Pi_h\bgam -\bgam_h, \Pi_h\bze - \bze_h\big)\in \mathcal{H}^{DG}_{\text{sym},h}$ and $\big(\mathcal I_{\bgam}^h, \mathcal I_{\bze}^h\big)(t) := (\bgam - \Pi_h \bgam, \bze - \Pi_h\bze)  \in \mathcal{H}^+_{\text{sym}}$.  
Next,  
we begin the convergence analysis by providing a stability estimate in terms of a DG  energy functional $\mathcal{E}_h$ given, for all $(\bet , \btau)\in \cC^1(\mathcal{H}^+_{\text{sym}}(h))$, by 
\begin{equation*}
\mathcal{E}_h\Big((\bet, \btau)\Big)(t):= \frac{1}{2} A\Big( (\dot\bet ,  \dot\btau), (\dot\bet ,  \dot\btau) \Big) + \frac{1}{2}  
\Big(\tfrac{1}{\rho} \bdiv_h (\bet +  \btau), \bdiv_h (\bet +  \btau) \Big) + \frac{1}{2}\Big(\texttt{a} h_\cF^{-1}\jump{\bet + \btau}, \jump{\bet + \btau} \Big)_{\cF^*_h}.
\end{equation*}
It is straightforward to deduce from \eqref{contA} and \eqref{ellipA}  that  $\mathcal{E}_h$ satisfies  
\begin{multline}\label{eq-extra-1DG} 
C^-\Big( \norm*{ (\dot\bet ,   \dot\btau) }_{0,\Omega}^2 + \norm*{ \bdiv_h (\bet +  \btau) }_{0,\Omega}^2 + \norm*{h_{\cF}^{-1/2} \jump{\bet + \btau}}^2_{0,\cF^*_h} \Big)\leq 	\mathcal{E}\Big((\bet, \btau)\Big)(t)
\\
  \leq  
  C^+ \Big( \norm*{ (\dot\bet ,   \dot\btau) }_{0,\Omega}^2 + \norm*{ \bdiv_h (\bet +  \btau) }_{0,\Omega}^2 + \norm*{h_{\cF}^{-1/2} \jump{\bet + \btau}}^2_{0,\cF^*_h}\Big),
\end{multline}
for all $(\bet , \btau)\in \cC^1(\mathcal{H}^+_{\text{sym}}(h))$, with $C^-:= \min\set*{\frac{\alpha}{2}, \frac{1}{2\rho^+}, \tfrac{\texttt a}{2}} $ and $C^+ := \max\set*{\frac{M}{2}, \frac{1}{2\rho^-}, \tfrac{\texttt a}{2}}$.

\begin{lemma} \label{stability}
Under Assumptions~\ref{assumption1} and \ref{assumption2}, there exists a positive parameter $\texttt{a}_0$ such that for all $\emph{\texttt{a}} \geq \texttt{a}_0$, the estimate 
\begin{align}\label{stab0}
\begin{split}
\displaystyle
\max_{[0, T]}\mathcal{E}_h\Big(\big( \be^h_{\bgam}, \be^h_{\bze} \big)\Big)      
\leq C \Big( \max_{[0, T]}\norm*{(\mathcal I_{\bgam}^h, \mathcal I_{\bze}^h)}^2_{W^{2,\infty}(L^2(\Omega, \mathbb M\times \mathbb M))}  + \max_{[0, T]}\norm*{\bdiv ({\mathcal I}^h_{\bgam} +  {\mathcal I}^h_{\bze})}^2_{W^{1,\infty}(L^2(\Omega, \R^d))} 
\\[0.25ex]
\qquad +  \max_{[0, T]}\norm*{ h_\cF^{1/2} \mean{\tfrac{1}{\rho} \bdiv ({\mathcal I}^h_{\bgam} +  {\mathcal I}^h_{\bze})}}^2_{W^{1,\infty}(L^2(\cF^*_h, \R^d))} \Big),
\end{split}
\end{align}
holds true with a constant $C>0$ independent of $h$.
\end{lemma}
\begin{proof} 
We have seen in the proof of Proposition~\ref{consistency} that, under Assumption~\ref{assumption2}(i), $\ddot\bu =  \tfrac{1}{\rho} \big(  \bdiv (\bgam + \bze) + \bF \big) \in L^{\infty}(H^1(\Omega,\R^d))$. Using Assumption~\ref{assumption1} and Assumption~\ref{assumption2}(ii) we can also assert that 
\[
  \frac{\text{d}^3 \bu}{\text{d}t^3} =  \tfrac{1}{\rho} \big(  \bdiv (\dot\bgam + \dot\bze) + \dot\bF \big) \in L^{\infty}(L^2(\Omega,\R^d)) 
  \quad \text{and}\quad 
  \beps(\frac{\text{d}^3\bu}{\text{d}t^3})= \frac{\text{d}^3 \bgam}{\text{d}t^3}\in L^\infty(L^2(\Omega, \bbS)).
\]
Consequently, $\ddot\bu   \in W^{1,\infty}(H^1(\Omega,\R^d))$ and it readily follows that 
\[
  \bdiv( \bgam + \bze ) = \rho \ddot \bu - \bF \in W^{1,\infty}(H^1(\cup_j\Omega_j, \R^d)),
\]
which implies that the interlement traces of $\tfrac{1}{\rho} \bdiv ({\mathcal I}^h_{\bgam} +  {\mathcal I}^h_{\bze})$ involved in \eqref{stab0} are meaningful. This also allows us to introduce  the linear form  $(\bet_h , \btau_h) \mapsto G\Big(({\mathcal I}^h_{\bgam},  {\mathcal I}^h_{\bze}), (\bet_h ,\btau_h)\Big)$ defined on $\mathcal{H}_{\text{sym},h}^{DG}$ by 
\begin{align*}
G\Big(({\mathcal I}^h_{\bgam},  {\mathcal I}^h_{\bze}), &(\bet_h ,\btau_h)\Big) :=  -\Big(\tfrac{1}{\rho} \bdiv ({\mathcal I}^h_{\bgam} +  {\mathcal I}^h_{\bze}), \bdiv_h(\bet_h + \btau_h) \Big)   +  \Big(\mean{\tfrac{1}{\rho} \bdiv ({\mathcal I}^h_{\bgam} +  {\mathcal I}^h_{\bze})}, \jump{\bet_h +  \btau_h}\Big)_{\cF^*_h},
\end{align*}
and we deduce from the Cauchy Schwarz inequality   that, if  $\emph{\texttt{a}} \geq 1$, 
\begin{align}\label{cotaG}
\begin{split}
	&\big|G\Big(({\mathcal I}^h_{\bgam},  {\mathcal I}^h_{\bze}), (\bet_h ,\btau_h)\Big)\Big| 
	\\[0.25ex]
	&\quad \leq \sqrt 2\left( \frac{1}{ \rho^-}\norm*{\bdiv ({\mathcal I}^h_{\bgam} + {\mathcal I}^h_{\bze})}^2_{0,\Omega}  + \norm*{ h_\cF^{1/2} \mean{\tfrac{1}{\rho} \bdiv ({\mathcal I}^h_{\bgam} +  {\mathcal I}^h_{\bze})}}^2_{0,\cF^*_h} \right)^{1/2} \mathcal{E}_h\Big((\bet_h, \btau_h)\Big)^{1/2}.
\end{split}
\end{align}
Next, using the consistency property \eqref{consistent},   it is straightforward to deduce that    
\begin{align}\label{errorIdentity}
\nonumber
&A\Big( (\ddot\be^h_{\bgam} ,  \ddot\be^h_{\bze}), (\bet_h ,   \btau_h) \Big)  + (\tfrac{1}{\omega}\cV\dot\be^h_{\bze}, \btau_h)  + \Big(\tfrac{1}{\rho}\bdiv_h (\be^h_{\bgam} +  \be^h_{\bze}),  \bdiv_h(\bet_h + \btau_h) \Big) 
    \\[0.25ex]
&\quad
  - \Big(\mean{\tfrac{1}{\rho} \bdiv_h (\be^h_{\bgam} +  \be^h_{\bze})}, \jump{\bet_h + \btau_h}\Big)_{\cF^*_h}
  - \Big(\mean{\tfrac{1}{\rho} \bdiv_h (\bet_h + \btau_h)}, \jump{\be^h_{\bgam} +  \be^h_{\bze}}\Big)_{\cF^*_h} 
  \\[0.25ex]
&\quad
+ \Big(\texttt{a} h_\cF^{-1}\jump{\be^h_{\bgam} +  \be^h_{\bze}}, \jump{\bet_h + \btau_h} \Big)_{\cF^*_h} 
= 
- A\Big( (\ddot{\mathcal I}^h_{\bgam} ,  \ddot{\mathcal I}^h_{\bze}), (\bet_h ,   \btau_h) \Big)  - (\tfrac{1}{\omega}\cV\dot{\mathcal I}^h_{\bze},\btau_h) +
G\Big(({\mathcal I}^h_{\bgam} ,  {\mathcal I}^h_{\bze}),(\bet_h ,\btau_h)\Big), 
\nonumber
\end{align}
for all $(\bet_h , \btau_h)\in \mathcal{H}_{\text{sym},h}^{DG}$. The choice $(\bet_h ,\btau_h) = \big( \dot{\be}^h_{\bgam}, \dot{\be}^h_{\bze} \big)(t)$ in  \eqref{errorIdentity} yields  
\begin{align*}
\begin{split}
\dot{\mathcal{E}}_h\Big( \big( \be^h_{\bgam}, \be^h_{\bze} \big) \Big)  
 \leq   \dfrac{\text{d}}{\text{d}t} \Big(\mean{\tfrac{1}{\rho} \bdiv_h (\be^h_{\bgam}+ \be^h_{\bze})}, \jump{\be^h_{\bgam}+ \be^h_{\bze}}\Big)_{\cF^*_h}  
 - A\Big( (\ddot{\mathcal I}^h_{\bgam} ,  \ddot{\mathcal I}^h_{\bze}), \big( \dot{\be}^h_{\bgam}, \dot{\be}^h_{\bze} \big) \Big)  
 \\[0.25ex]
- (\tfrac{1}{\omega}\cV\dot{\mathcal I}^h_{\bze},\dot{\be}^h_{\bze}) +  G\Big(({\mathcal I}^h_{\bgam} ,  {\mathcal I}^h_{\bze}),\big( \dot{\be}^h_{\bgam}, \dot{\be}^h_{\bze} \big)\Big),
\end{split}
\end{align*}
where we took into account that the term $\big( \tfrac{1}{\omega}\cV\dot\be^h_{\bze},\dot\be^h_{\bze}  \big)$ is non-negative. Integrating the last estimate with respect to time we get   
\begin{align}\label{in1}
\begin{split}
\mathcal{E}_h\Big(\big( \be^h_{\bgam}, \be^h_{\bze} \big)\Big) 
&\leq 
  \Big(\mean{\tfrac{1}{\rho} \bdiv_h (\be^h_{\bgam}+ \be^h_{\bze})}, \jump{\be^h_{\bgam}+ \be^h_{\bze}}\Big)_{\cF^*_h}  - \int_0^tA\Big( (\ddot{\mathcal I}^h_{\bgam} ,  \ddot{\mathcal I}^h_{\bze}), \big( \dot{\be}^h_{\bgam}, \dot{\be}^h_{\bze} \big) \Big)\, \text{d}s
  \\
 &
    - \int_0^t(\tfrac{1}{\omega}\cV\dot{\mathcal I}^h_{\bze},\dot{\be}^h_{\bze})\, \text{d}s+  \int_0^t G\Big(({\mathcal I}^h_{\bgam} ,  {\mathcal I}^h_{\bze}),\big( \dot{\be}^h_{\bgam}, \dot{\be}^h_{\bze} \big)\Big)\, \text{d}s.
\end{split}
\end{align}
We will now estimate the different terms of the right-hand side of \eqref{in1} by using repeatedly  the Cauchy-Schwarz inequality, \eqref{contA} followed by Young's inequality $a b \leq \frac{a^2}{4} +  b^2$. Thanks to \eqref{discTrace},  the first term can be bounded as follows: 
\begin{align}\label{babor0}
\begin{split}
	\Big(\mean{\tfrac{1}{\rho} \bdiv_h (\be^h_{\bgam}+ \be^h_{\bze})}, \jump{\be^h_{\bgam}+ \be^h_{\bze}}\Big)_{\cF^*_h} &\leq \norm*{h_\cF^{1/2}\mean{\tfrac{1}{\rho} \bdiv_h (\be^h_{\bgam}+ \be^h_{\bze})} }_{0,\cF^*_h} \norm*{h_\cF^{-1/2}\jump{\be^h_{\bgam}+ \be^h_{\bze}} }_{0,\cF^*_h}
	\\[0.25ex]
	&\leq \frac{C_{\text{tr}}}{\sqrt{\rho^-}} \norm*{\frac{1}{\sqrt \rho} \bdiv_h (\be^h_{\bgam}+ \be^h_{\bze}) }_{0,\Omega}\norm*{h_\cF^{-1/2}\jump{\be^h_{\bgam}+ \be^h_{\bze}} }_{0,\cF^*_h} 
	\\[0.25ex]
	&\leq \frac{1}{4}\max_{[0, T]}\mathcal{E}_h\Big(\big( \be^h_{\bgam}, \be^h_{\bze} \big)\Big) +  \frac{2C_{\text{tr}}^2}{\rho^-}  \max_{[0, T]}\norm*{h_\cF^{-1/2}\jump{\be^h_{\bgam}+ \be^h_{\bze}} }_{0,\cF^*_h}^2.
	\end{split}
\end{align}
For the second term  we have
\begin{align}\label{babor1}
\begin{split}
	&-\int_0^t A\Big( (\ddot{\mathcal I}^h_{\bgam} ,  \ddot{\mathcal I}^h_{\bze}), (\dot{\be}^h_{\bgam} ,  \dot{\be}^h_{\bze}) \Big)\, \text{d}s  - \int_0^t(\tfrac{1}{\omega}\cV\dot{\mathcal I}^h_{\bze},\dot{\be}^h_{\bze}) \, \text{d}s 
	\\[0.25ex]&
	\leq  \max_{[0, T]} A\Big( (\dot{\be}^h_{\bgam} , \dot{\be}^h_{\bze}), (\dot{\be}^h_{\bgam} , \dot{\be}^h_{\bze})\Big)^{1/2}  \int_0^T \Big(\sqrt{M}\norm*{(\ddot{\mathcal I}^h_{\bgam} , \ddot{\mathcal I}^h_{\bze})}_{0,\Omega} + \frac{\sqrt{M}}{\omega_0}\norm*{(\dot{\mathcal I}^h_{\bgam} , \dot{\mathcal I}^h_{\bze})}_{0,\Omega}\Big)\text{d}t
	\\[0.25ex]
	& \leq \frac{1}{4}\max_{[0, T]}\mathcal{E}_h\Big(({\be}^h_{\bgam} , {\be}^h_{\bze})\Big) + 2M T^2 (1 + \frac{1}{\omega_0})^2 \norm*{(\mathcal I_{\bgam}^h, \mathcal I_{\bze}^h)}^2_{W^{2,\infty}(L^2(\Omega, \mathbb M\times \mathbb M))}.
	\end{split}	
\end{align}
Finally, an integration by parts gives   
\begin{equation*}
	\int_0^t G\Big(({\mathcal I}^h_{\bgam} ,  {\mathcal I}^h_{\bze}),\big( \dot{\be}^h_{\bgam}, \dot{\be}^h_{\bze} \big)\Big)\, \text{d}s = -\int_0^t G\Big((\dot{\mathcal I}^h_{\bgam} ,  \dot{\mathcal I}^h_{\bze}),\big( {\be}^h_{\bgam}, {\be}^h_{\bze} \big)\Big)\, \text{d}s + G\Big(({\mathcal I}^h_{\bgam} ,  {\mathcal I}^h_{\bze}),\big( {\be}^h_{\bgam}, {\be}^h_{\bze} \big)\Big),
\end{equation*}
 and applying \eqref{cotaG} we deduce that 
 
 \begin{align}\label{babor3}
\begin{split}
	\big| \int_0^t G\Big(({\mathcal I}^h_{\bgam} &,  {\mathcal I}^h_{\bze}),\big( \dot{\be}^h_{\bgam}, \dot{\be}^h_{\bze} \big)\Big)\, \text{d}s \big| \leq   2 (T+1)^2(1 + \frac{1}{\rho^-})^2\Big( \norm*{\bdiv ({\mathcal I}^h_{\bgam} +  {\mathcal I}^h_{\bze})}^2_{W^{1,\infty}(L^2(\Omega, \R^d))} 
	 \\[0.25ex]
	 &\quad +  \norm*{ h_\cF^{1/2} \mean{\tfrac{1}{\rho} \bdiv ({\mathcal I}^h_{\bgam} +  {\mathcal I}^h_{\bze})}}^2_{W^{1,\infty}(L^2(\cF^*_h, \R^d))} \Big) + \frac{1}{4}\max_{[0, T]}\mathcal{E}\Big(\big( \be^h_{\bgam}, \be^h_{\bze} \big)\Big).
\end{split}
\end{align}

Plugging \eqref{babor0},  \eqref{babor1},  and \eqref{babor3}  in \eqref{in1} and rearranging terms we deduce that \eqref{stab0} is satisfied for $\texttt{a}\geq \texttt{a}_0:=\max\{1, \frac{32C_{\text{tr}}^2}{\rho^-}\}$.   
\end{proof}


\begin{theorem} 
Let $(\bgam, \bze)$ and $(\bgam_h, \bze_h)$ be the solutions of problems \eqref{varFormR1-varFormR2} and  \eqref{varFormR1-varFormR2-h}, respectively. Under Assumptions~\ref{assumption1} and \ref{assumption2}, the  error estimate 
\begin{align}\label{errorE}
\begin{split}
\displaystyle
\max_{t\in [0, T]}&\norm*{(\bgam -\bgam_h, \bze - \bze_h)(t)}_{\mathcal{H}^+_{\text{sym}}(h)} + \max_{t\in [0, T]}\norm*{\big(\dot\bgam -\dot\bgam_h, \dot\bze - \dot\bze_h\big)(t)}_{0,\Omega} 
\\[0.25ex]
&
\leq C \Big( \norm*{(\mathcal I_{\bgam}^h, \mathcal I_{\bze}^h)}_{W^{2,\infty}(L^2(\Omega, \mathbb M\times \mathbb M))} +\norm*{  \bdiv (\mathcal I_{\bgam}^h +  \mathcal I_{\bze}^h)}_{W^{1,\infty}(L^2(\Omega,\R^d))}
\\[0.25ex]
& \qquad \quad  +\norm*{ h_\cF^{1/2} \mean{\tfrac{1}{\rho} \bdiv (\mathcal I_{\bgam}^h +  \mathcal I_{\bze}^h)}}_{W^{1,\infty}(L^2(\cF^*_h,\R^d))} \Big),
\end{split}
\end{align}
holds true for all $\emph{\texttt{a}}\geq \texttt{a}_0$, with $C>0$ independent of $h$. 
\end{theorem}
\begin{proof}
We deduce from \eqref{stab0}, the lower bound in \eqref{eq-extra-1DG},  and 
\[
\norm*{\big( \be^h_{\bgam}, \be^h_{\bze} \big)(t)}_{0,\Omega} = \norm*{\int_0^t \big( \dot{\be}^h_{\bgam}, \dot{\be}^h_{\bze} \big)(s)\, \text{d}s}_{0,\Omega}\leq T \max_{[0,T]} \norm*{\big( \dot{\be}^h_{\bgam}, \dot{\be}^h_{\bze} \big)}_{0,\Omega}, 
\]
that 
\begin{align*}
\max_{[0, T]}&\norm*{ (\dot{\be}^h_{\bgam} , \dot{\be}^h_{\bze})(t)}_{0,\Omega} + \max_{[0, T]}\norm*{ ({\be}^h_{\bgam} , {\be}^h_{\bze})(t)}_{\mathcal{H}^+_{\text{sym}}(h)} 
\lesssim \max_{[0, T]}\norm*{(\mathcal I_{\bgam}^h, \mathcal I_{\bze}^h)}_{W^{2,\infty}(L^2(\Omega, \mathbb M\times \mathbb M))}   
\\[0.25ex]
&+ \max_{[0, T]}\norm*{\bdiv ({\mathcal I}^h_{\bgam} +  {\mathcal I}^h_{\bze})}_{W^{1,\infty}(L^2(\Omega, \R^d))}+  \max_{[0, T]}\norm*{ h_\cF^{1/2} \mean{\tfrac{1}{\rho} \bdiv ({\mathcal I}^h_{\bgam} +  {\mathcal I}^h_{\bze})}}_{W^{1,\infty}(L^2(\cF^*_h, \R^d))},
\end{align*}
and the result follows from the triangle inequality.
\end{proof}

\begin{corollary}
Let $(\bgam, \bze)$ and $(\bgam_h, \bze_h)$ be the solutions of problems \eqref{varFormR1-varFormR2} and  \eqref{varFormR1-varFormR2-h}, respectively. Under Assumptions~\ref{assumption1} and \ref{assumption2}, and if $\bgam, \bze\in W^{2,\infty}( H^{r}(\cup_j\Omega_j,\bbM))$ and  $(\bgam + \bze) \in W^{1,\infty}( H^{r+1}(\cup_j\Omega_j,\bbM))$, with $r\geq 1$, we have that 
\begin{align*}
&\max_{t\in [0, T]}\norm*{(\bgam -\bgam_h, \bze - \bze_h)}_{\mathcal{H}^+_{\text{sym}}(h)} + \max_{t\in [0, T]}\norm*{\big(\dot\bgam -\dot\bgam_h, \dot\bze - \dot\bze_h\big)}_{0,\Omega}   
\\
&\qquad \leq C h^{\min\{r,k\}} \sum_{j=1}^J   \Big( \norm*{(\bet, \btau)}_{W^{2,\infty}(H^r(\Omega_j, \mathbb M\times \mathbb M))} + \norm*{\bet + \btau}_{W^{1,\infty}(H^{r+1}(\Omega_j, \mathbb M))} \Big),
\end{align*}
for all $\emph{\texttt{a}}\geq \texttt{a}_0$, with $C>0$ independent of $h$.
\end{corollary}
\begin{proof}
The result is a direct consequence of \eqref{errorE} and Lemma~\ref{maintool}. 
\end{proof}

\begin{remark} Note that if the mixed-loading boundary conditions of \eqref{split123} are non-homogeneous 
\begin{equation}\label{eq:bc-nonhomo}
 \bu = \bg_D \quad  \text{on $\Gamma_D\times (0, T]$}, \qquad 
 (\bgam + \bze)\bn = \bg_N \quad  \text{on $\Gamma_N\times (0, T]$},\end{equation}
for sufficiently smooth displacement and traction data $\bg_D,\bg_N$, then the semi-discrete formulation \eqref{varFormR1-varFormR2-h} is modified as follows: 
Find $(\bgam_h, \bze_h)\in \cC^2(\mathcal{H}^{DG}_{\text{sym},h})$  solution of 
\begin{align*}
\nonumber
 &A\Big( (\ddot\bgam_h ,  \ddot\bze_h), (\bet ,   \btau) \Big)  + (\tfrac{1}{\omega}\cV\dot\bze_h,\btau) 
+ \Big(\tfrac{1}{\rho}\bdiv_h(\bgam_h +  \bze_h) , \bdiv_h(\bet +  \btau) \Big)  
\\
&\quad - \Big(\mean{\tfrac{1}{\rho} \bdiv_h (\bgam_h + \bze_h)}, \jump{\bet_h +  \btau_h}\Big)_{\cF^*_h}
- \Big(\mean{\tfrac{1}{\rho} \bdiv_h (\bet_h +  \btau_h)}, \jump{\bgam_h + \bze_h}\Big)_{\cF^*_h}
\nonumber\\
&\qquad + \Big(\texttt{a} h_\cF^{-1}\jump{\bgam_h + \bze_h}, \jump{\bet_h +  \btau_h} \Big)_{\cF^*_h}
 = -\Big(\tfrac{1}{\rho}\bF,  \bdiv_h (\bet_h +  \btau_h) \Big) + \Big(\mean{\tfrac{1}{\rho} \bF}, \jump{\bet_h +  \btau_h}\Big)_{\cF^*_h}
 \nonumber\\
 &\qquad \qquad \qquad \quad 
 + \Big(\ddot{\bg}_D, (\bet_h +  \btau_h)\bn\Big)_{\cF^D_h} 
 + \Big(\texttt{a} h_\cF^{-1}{\bg}_N, (\bet_h +  \btau_h)\bn\Big)_{\cF^N_h}  
 , \quad \forall (\bet_h, \btau_h)\in \mathcal{H}^{DG}_{\text{sym},h}.
 \end{align*}

\end{remark}

\section{The fully discrete scheme and its convergence analysis}\label{sec:fully-discrete}

Given $L\in \mathbb{N}$, we consider a uniform partition of the time interval $[0, T]$ with step size 
$\Delta t := T/L$. Then, for any continuous function $\phi:[0, T]\to \R$ and for each $k\in\{0,1,\ldots,L\}$, 
we denote $\phi^k := \phi(t_k)$, where $t_k := k\,\Delta t$. In addition,  we let $t_{k+\frac{1}{2}}:= \frac{t_{k+1} + t_k}{2}$, 
$\phi^{k+\frac{1}{2}}:= \frac{\phi^{k+1} + \phi^k}{2}$, $\phi^{k-\frac{1}{2}}:= \frac{\phi^{k} + \phi^{k-1}}{2}$, and $\widehat {\phi^k} := \frac{\phi^{k+\frac{1}{2}} + \phi^{k-\frac{1}{2}}}{2}= \frac{\phi^{k-1} + 2\phi^k + \phi^{k+1}}{4}$. We adopt the same notation for 
vector/tensor valued functions. We also introduce the discrete time derivatives
\[
\partial_t \phi^k := \frac{\phi^{k+1} - \phi^k}{\Delta t}, \quad \bar \partial_t \phi^k :=  \frac{\phi^{k} - \phi^{k-1}}{\Delta t}\quad\text{and} \quad \partial^0_t \phi^k := \frac{\phi^{k+1} - \phi^{k-1}}{2\Delta t}.
\]

In what follows we utilise the Newmark trapezoidal rule for the time discretisation of 
\eqref{varFormR1-varFormR2-h}-\eqref{initial-R1-R2-h*c}. Namely, for  $k=1,\ldots,L-1$, 
we seek  $(\bgam_h^{k+1}, \bze_h^{k+1})\in \mathcal{H}^{DG}_{\text{sym},h}$  solution of   
\begin{align}\label{fullyDiscretePb1-Pb2}
\begin{split}
 &A\Big( \partial_t\bar \partial_t(\bgam^k_h ,  \bze^k_h), (\bet ,   \btau) \Big)  + (\tfrac{1}{\omega}\cV\partial^0_t\bze^k_h,\btau_h) 
+ \Big(\tfrac{1}{\rho}\bdiv_h(\widehat{\bgam^k_h} +  \widehat{\bze^k_h}) , \bdiv_h(\bet +  \btau) \Big)  
\\[0.25ex]
&\quad - \Big(\mean{\tfrac{1}{\rho} \bdiv_h (\widehat{\bgam^k_h} +  \widehat{\bze^k_h})}, \jump{\bet_h +  \btau_h}\Big)_{\cF^*_h}
- \Big(\mean{\tfrac{1}{\rho} \bdiv_h (\bet_h +  \btau_h)}, \jump{\widehat{\bgam^k_h} +  \widehat{\bze^k_h}}\Big)_{\cF^*_h}
\\[0.25ex]
&\qquad + \Big(\texttt{a} h_\cF^{-1}\jump{\widehat{\bgam^k_h} +  \widehat{\bze^k_h}}, \jump{\bet_h +  \btau_h} \Big)_{\cF^*_h}
 = -\Big(\tfrac{1}{\rho}\bF,  \bdiv_h (\bet_h +  \btau_h) \Big)
 \\[0.25ex] 
 &\qquad \quad + \Big(\mean{\tfrac{1}{\rho} \bF}, \jump{\bet_h +  \btau_h}\Big)_{\cF^*_h}, \quad \forall (\bet_h, \btau_h)\in \mathcal{H}^{DG}_{\text{sym},h}.
  \end{split}
 \end{align}
 For the sake of simplicity, we assume that the scheme \eqref{fullyDiscretePb1-Pb2} is started up with 
 \begin{equation}\label{initial-R1-R2-h*c*}
 	(\bgam^0_h, \bze^0_h) = \Pi_h( \bgam^0, \bze^0 )\quad \text{and}\quad (\bgam^1_h, \bze^1_h) = \Pi_h( \bgam(t_1), \bze(t_1) ). 
 \end{equation}
 Then, we introduce the projected error $\big(\be^k_{h,\bgam}, \be^k_{h,\bze}\big) = \big(\Pi_h\bgam(t_k) -\bgam^k_h, \Pi_h\bze(t_k) - \bze^k_h\big)\in \mathcal{H}^{DG}_{\text{sym},h}$ and notice that, because of \eqref{initial-R1-R2-h*c*}, we can ignore the error at the first two initial steps since $\big(\be^0_{h,\bgam}, \be^0_{h,\bze}\big) = (\mathbf 0, \mathbf 0)$ and $\big(\be^1_{h,\bgam}, \be^1_{h,\bze}\big) = (\mathbf 0, \mathbf 0)$. 
 
 We begin our convergence analysis by providing a stability estimate in terms of the fully discrete energy functional $\mathcal{E}^k_h$ given by 
\begin{align*}
\mathcal{E}^k_h:= \frac{1}{2} A\Big( \partial_t (\be^k_{\bgam,h} , \be^k_{\bze,h}), 
 \partial_t (\be^k_{\bgam,h} , \be^k_{\bze,h}) \Big) + \frac{1}{2}  
\norm*{\tfrac{1}{\sqrt{\rho}} \bdiv_h \big(\be^{k+\frac{1}{2}}_{\bgam,h} + \be^{k+\frac{1}{2}}_{\bze,h}\big)}_{0,\Omega}^2 
+ \frac{1}{2}\norm*{\texttt{a}^{\frac{1}{2}} h_\cF^{-\frac{1}{2}}\jump{\be^{k+\frac{1}{2}}_{\bgam,h} + \be^{k+\frac{1}{2}}_{\bze,h} }}_{0,\cF^*_h}^2.
\end{align*}
We notice that  
\begin{multline}\label{eq-extra-1DG-fully} 
C^-\Big( \norm*{ \partial_t (\be^k_{\bgam,h} , \be^k_{\bze,h}) }_{0,\Omega}^2 + \norm*{ \bdiv_h (\be^{k+\frac{1}{2}}_{\bgam,h} + \be^{k+\frac{1}{2}}_{\bze,h}) }_{0,\Omega}^2 + \norm*{h_{\cF}^{-1/2} \jump{\be^{k+\frac{1}{2}}_{\bgam,h} + \be^{k+\frac{1}{2}}_{\bze,h}}}^2_{0,\cF^*_h} \Big)\leq 	\mathcal{E}^k_h
\\[0.25ex]
  \leq  
  C^+ \Big( \norm*{ \partial_t (\be^k_{\bgam,h} , \be^k_{\bze,h}) }_{0,\Omega}^2 + \norm*{ \bdiv_h (\be^{k+\frac{1}{2}}_{\bgam,h} + \be^{k+\frac{1}{2}}_{\bze,h}) }_{0,\Omega}^2 + \norm*{h_{\cF}^{-1/2} \jump{\be^{k+\frac{1}{2}}_{\bgam,h} + \be^{k+\frac{1}{2}}_{\bze,h}}}^2_{0,\cF^*_h}\Big),
\end{multline}
with the same constants $C^-$ and $C^+$ appearing in  \eqref{eq-extra-1DG}.

\begin{lemma}
Under Assumptions~\ref{assumption1} and \ref{assumption2},  the estimate 
\begin{align}\label{full4}
\begin{split}
\max_n\mathcal{E}^n_h&\leq C \Big(   \max_n\norm*{(\mathfrak{X}_1^n, \mathfrak{X}_2^n)}^2_{0,\Omega} + \max_n\norm*{\mathfrak{X}_3^n}^2_{0,\Omega}  +  \max_n\norm{\bdiv  \mathfrak{X}_4^n }_{0,\Omega} 
\\[0.25ex]
& 
 + \max_n\norm*{h_\cF^{\frac{1}{2}}\mean{\tfrac{1}{\rho}\bdiv \mathfrak{X}_4^n}}^2_{0,\cF^*_h} + \max_n\norm*{\bdiv  (\partial_t  \mathfrak{X}_4^n)}^2_{0,\Omega} 
 + \max_n\norm*{h_\cF^{\frac{1}{2}}\mean{\tfrac{1}{\rho}\bdiv  (\partial_t  \mathfrak{X}_4^n)}}^2_{0,\cF^*_h}\Big),  
 \end{split}
\end{align}
holds true for all $\texttt{a}\geq \texttt{a}_0$, with $C>0$ independent of $h$ and $\Delta t$.
\end{lemma}

\begin{proof}
It follows from \eqref{consistent} that the projected error $\big(\be^k_{h,\bgam}, \be^k_{h,\bze}\big) \in \mathcal{H}^{DG}_{\text{sym},h}$ solves 
\begin{align}\label{projectedErrorEqk}
\begin{split}
&A\Big( \partial_t\bar \partial_t(\be^k_{h,\bgam}, \be^k_{h,\bze}), (\bet_h ,   \btau_h) \Big)  + (\tfrac{1}{\omega}\cV\partial^0_t\be^k_{h,\bze},\btau_h) 
+ \Big(\tfrac{1}{\rho}\bdiv_h(\widehat{\be^k_{h,\bgam}} +  \widehat{\be^k_{h,\bze}}) , \bdiv_h(\bet_h +  \btau_h) \Big)  
\\[.5ex]
&\quad - \Big(\mean{\tfrac{1}{\rho} \bdiv_h (\widehat{\be^k_{h,\bgam}} +  \widehat{\be^k_{h,\bze}})}, \jump{\bet_h +  \btau_h}\Big)_{\cF^*_h}
- \Big(\mean{\tfrac{1}{\rho} \bdiv_h (\bet_h +  \btau_h)}, \jump{\widehat{\be^k_{h,\bgam}} +  \widehat{\be^k_{h,\bze}}}\Big)_{\cF^*_h}
\\[.5ex]
&\qquad + \Big(\texttt{a} h_\cF^{-1}\jump{\widehat{\be^k_{h,\bgam}} +  \widehat{\be^k_{h,\bze}}}, \jump{\bet_h +  \btau_h} \Big)_{\cF^*_h}
 = G^k((\bet_h, \btau_h)), 
\end{split}
\end{align}
for all $(\bet_h, \btau_h)\in \mathcal{H}^{DG}_{\text{sym},h}$, where  
\begin{align*}
	G^k((\bet_h, \btau_h)) &:=A\Big( (\mathfrak{X}_1^k, \mathfrak{X}_2^k)  , (\bet_h, \btau_h) \Big)  + (\tfrac{1}{\omega}\cV\mathfrak{X}_3^k,\btau_h) + \big(\tfrac{1}{\rho} \bdiv \mathfrak{X}_4^k,  \bdiv_h (\bet_h + \btau_h) \big)
	\\[0.25ex]
	&\quad - \big(\mean{\tfrac{1}{\rho} \bdiv  \mathfrak{X}_4^k }, \jump{\bet_h + \btau_h}\big)_{\cF_h},
\end{align*}
with consistency functions 
\begin{align*}
	\mathfrak{X}_1^k&:=  \Pi_h \partial_t\bar \partial_t \bgam(t_k) - \ddot\bgam(t_k), 
 & \mathfrak{X}_2^k &:=   \Pi_h \partial_t\bar \partial_t\bze(t_k) - \ddot\bze(t_k), 
	\\ 
	\mathfrak{X}_3^k &:=   \Pi_h \partial^0_t \bze(t_k) - \dot\bze(t_k), & \mathfrak{X}_4^k &:= \Pi_h (\widehat{\bgam(t_k)} + \widehat{\bze(t_k)}) - (\bgam + \bze)(t_k).
\end{align*}  

	Taking  $(\bet_h, \btau_h) = \partial^0_t \big(\be^k_{h,\bgam}, \be^k_{h,\bze}\big)= \partial_t (\be_{h,\bgam}^{k-\frac{1}{2}}, \be_{h,\bze}^{k-\frac{1}{2}})$  in \eqref{projectedErrorEqk} 
	 yields the identity 
\begin{align*}
	& A\Big(  \partial_t\bar\partial_t( \be^k_{\bgam,h},  \be^k_{\bze,h}), 
  \partial_t (\be_{h,\bgam}^{k-\frac{1}{2}}, \be_{h,\bze}^{k-\frac{1}{2}}) \Big) + 
 (\tfrac{1}{\omega}\cV\partial^0_t\be^k_{h,\bze},\partial^0_t\be^k_{h,\bze}) \\[2ex]
\displaystyle
&+  \Big(\tfrac{1}{\rho}\bdiv_h  \big(\widehat{\be^k_{h,\bgam}} +  \widehat{\be^k_{h,\bze}}\big)  , \bdiv_h  \partial_t (\be_{h,\bgam}^{k-\frac{1}{2}} + \be_{h,\bze}^{k-\frac{1}{2}}) \Big) 
+  \Big(\texttt{a} h_\cF^{-1}\jump{ \big(\widehat{\be^k_{h,\bgam}} +  \widehat{\be^k_{h,\bze}}\big)}, \jump{ \partial_t (\be_{h,\bgam}^{k-\frac{1}{2}} + \be_{h,\bze}^{k-\frac{1}{2}})} \Big)_{\cF^*_h}
\\[0.25ex]
& -\Big(\mean{\tfrac{1}{\rho} \bdiv_h (\widehat{\be^k_{h,\bgam}} +  \widehat{\be^k_{h,\bze}})}, \jump{\partial_t (\be_{h,\bgam}^{k-\frac{1}{2}} + \be_{h,\bze}^{k-\frac{1}{2}})}\Big)_{\cF^*_h}
- \Big(\mean{\tfrac{1}{\rho} \bdiv_h \partial_t (\be_{h,\bgam}^{k-\frac{1}{2}} + \be_{h,\bze}^{k-\frac{1}{2}})}, \jump{\widehat{\be^k_{h,\bgam}} +  \widehat{\be^k_{h,\bze}}}\Big)_{\cF^*_h} 
\\[0.25ex]
&= G^k\Big(\partial^0_t \big(\be^k_{h,\bgam}, \be^k_{h,\bze}\big)\Big).
\end{align*}
Unfolding the first bilinear form  of the previous identity  according to the decompositions 
\[
  \bar\partial_t \be^k_{\star,h} = \frac{\be^k_{\star,h} - \be^{k-1}_{\star,h}}{\Delta t} \quad  \text{and} \quad \be^{k-\frac{1}{2}}_{\star,h} = \frac{\be^k_{\star,h} + \be^{k-1}_{\star,h}}{2} \quad (\star \in \set{\bgam, \bze}),
\]
and the last four bilinear forms  according to the decompositions 
\[
  \widehat{\be^k_{\star,h}} = \frac{\be^{k+\frac{1}{2}}_{\star,h} + \be^{k-\frac{1}{2}}_{\star,h}}{2}\quad  \text{and} \quad  \partial_t \be_{h,\star}^{k-\frac{1}{2}} = \frac{\be^{k+\frac{1}{2}}_{\star,h} - \be^{k-\frac{1}{2}}_{\star,h}}{\Delta t} \quad  (\star \in \set{\bgam, \bze}),
\]
we readily deduce that 
\begin{align*}
& \mathcal{E}^k_h - \mathcal{E}^{k-1}_h + \Delta t
 (\tfrac{1}{\omega}\cV\partial^0_t\be^k_{h,\bze},\partial^0_t\be^k_{h,\bze})
 =   \Delta t G^k\Big(\partial^0_t \big(\be^k_{h,\bgam}, \be^k_{h,\bze}\big)\Big)
 \\[0.25ex]
 &+\Big( \mean{\tfrac{1}{\rho} \bdiv_h \big(\be^{k+\frac{1}{2}}_{\bgam,h} + \be^{k+\frac{1}{2}}_{\bze,h}\big)}, \jump{\be^{k+\frac{1}{2}}_{\bgam,h} + \be^{k+\frac{1}{2}}_{\bze,h}} \Big)_{\cF^*_h} -  \Big( \mean{\tfrac{1}{\rho} \bdiv_h \big(\be^{k-\frac{1}{2}}_{\bgam,h} + \be^{k-\frac{1}{2}}_{\bze,h}\big)}, \jump{\be^{k-\frac{1}{2}}_{\bgam,h} + \be^{k-\frac{1}{2}}_{\bze,h}} \Big)_{\cF^*_h}.
\end{align*}
Summing the foregoing identity over $k=1,\ldots, n$,  we obtain the estimate 
\begin{align}\label{full2}
\begin{split}
	\mathcal{E}^n_h
\leq \Delta t \sum_{k= 1}^n  G^k\Big(\partial^0_t \big(\be^k_{h,\bgam}, \be^k_{h,\bze}\big)\Big) 
+   \Big( \mean{\tfrac{1}{\rho} \bdiv_h \big(\be^{n+\frac{1}{2}}_{\bgam,h} + \be^{n+\frac{1}{2}}_{\bze,h}\big)}, \jump{\be^{n+\frac{1}{2}}_{\bgam,h} + \be^{n+\frac{1}{2}}_{\bze,h}} \Big)_{\cF^*_h} .
\end{split}
\end{align}
Proceeding as in \eqref{babor0} we obtain 
\begin{equation}\label{may}
	\left| \Big( \mean{\tfrac{1}{\rho} \bdiv_h \big(\be^{n+\frac{1}{2}}_{\bgam,h} + \be^{n+\frac{1}{2}}_{\bze,h}\big)}, \jump{\be^{n+\frac{1}{2}}_{\bgam,h} + \be^{n+\frac{1}{2}}_{\bze,h}} \Big)_{\cF^*_h} \right| \leq \frac{2C_{\text{tr}}^2}{\rho^-} \norm*{ h_\cF^{-\frac{1}{2}}\jump{\be^{n+\frac{1}{2}}_{\bgam,h} + \be^{n+\frac{1}{2}}_{\bze,h} }}_{0,\cF^*_h}^2 + \frac{1}{4} \mathcal{E}^n_h.
\end{equation}
Next, we turn to estimate the first term on the right-hand side of \eqref{full2}. We begin by  performing a discrete integration by part in the summations containing the term $\mathfrak{X}_4^k$  to obtain  
\begin{align*}
	\Delta t\sum_{k= 1}^n &G^k\Big(\partial^0_t \big(\be^k_{h,\bgam}, \be^k_{h,\bze}\big)\Big) = \Delta t\sum_{k= 1}^n A\Big( (\mathfrak{X}_1^k, \mathfrak{X}_2^k)  , \partial_t (\be_{h,\bgam}^{k-\frac{1}{2}}, \be_{h,\bze}^{k-\frac{1}{2}}) \Big) 
	+ \Delta t\sum_{k= 1}^n\big(\tfrac{1}{\omega}\cV \mathfrak{X}_3^k,  \partial_t  \be_{h,\bze}^{k-\frac{1}{2}} \big)
	\\[0.25ex]
	 &-  \Delta t\sum_{k= 1}^{n-1}\Big(\tfrac{1}{\rho}\bdiv \partial_t \mathfrak{X}_4^k,  \bdiv_h  \big( \be_{\bgam, h}^{k+\frac{1}{2}} + \be_{\bze, h}^{k+\frac{1}{2}} \big)\Big) +  \Big( \tfrac{1}{\rho}\bdiv  \mathfrak{X}_4^n, \bdiv_h  \big( \be_{\bgam, h}^{n+\frac{1}{2}} + \be_{\bze, h}^{n+\frac{1}{2}} \big)\Big)
	 \\[0.25ex]
	& + \Delta t\sum_{k= 1}^{n-1} \Big(\mean{\tfrac{1}{\rho} \bdiv  \partial_t\mathfrak{X}_4^k }, \jump{\be_{\bgam, h}^{k+\frac{1}{2}} + \be_{\bze, h}^{k+\frac{1}{2}}}\Big)_{\cF^*_h} - \Big(\mean{\tfrac{1}{\rho} \bdiv \mathfrak{X}_4^n }, \jump{\be_{\bgam, h}^{n+\frac{1}{2}} + \be_{\bze, h}^{n+\frac{1}{2}}}\Big)_{\cF^*_h}.
\end{align*}  
It follows now from the Cauchy-Schwarz inequality that
\begin{align*}
	\Delta t &\left| \sum_{k= 1}^n  G^k\Big(\partial^0_t \big(\be^k_{h,\bgam}, \be^k_{h,\bze}\big)\Big) \right| \leq \max_n A\Big( (\mathfrak{X}_1^n, \mathfrak{X}_2^n)  , (\mathfrak{X}_1^n, \mathfrak{X}_2^n)\Big)^{\frac{1}{2}} \max_n A\Big( \partial_t (\be_{h,\bgam}^{n-\frac{1}{2}}, \be_{h,\bze}^{n-\frac{1}{2}})  , \partial_t (\be_{h,\bgam}^{n-\frac{1}{2}}, \be_{h,\bze}^{n-\frac{1}{2}})\Big)^{\frac{1}{2}}   
	\\[0.25ex]
	 & 
	+ \max_n\big(\tfrac{1}{\omega}\cV \mathfrak{X}_3^k,  \mathfrak{X}_3^k \big)^{\frac{1}{2}} \max_n\big(\tfrac{1}{\omega}\cV \partial_t  \be_{h,\bze}^{k-\frac{1}{2}} ,  \partial_t  \be_{h,\bze}^{k-\frac{1}{2}} \big)^{\frac{1}{2}}
	\\[0.25ex]
	 &+   \Big( \max_n \norm*{\tfrac{1}{\sqrt \rho} \bdiv \mathfrak{X}_4^k}_{0,\Omega} + \max_n \norm*{\tfrac{1}{\sqrt \rho} \bdiv \partial_t \mathfrak{X}_4^k}_{0,\Omega} \Big) \max_n \norm*{\tfrac{1}{\sqrt \rho} \bdiv_h \big( \be_{\bgam, h}^{n+\frac{1}{2}} + \be_{\bze, h}^{n+\frac{1}{2}} \big)}_{0,\Omega}
	 \\[0.25ex]
	 &+   \Big( \max_n \norm*{h^{1/2}_{\cF} \mean{\tfrac{1}{ \rho} \bdiv \mathfrak{X}_4^k}}_{0,\cF^*_h} + \max_n \norm*{h^{1/2}_{\cF} \mean{\tfrac{1}{ \rho} \bdiv \partial_t\mathfrak{X}_4^k}}_{0,\cF^*_h} \Big) \max_n \norm*{h^{-1/2}_{\cF} \jump{\be_{\bgam, h}^{n+\frac{1}{2}} + \be_{\bze, h}^{n+\frac{1}{2}}}}_{0,\cF^*_h},
\end{align*} 
for all $1\leq n\leq L$. Finally, using the discrete trace inequality \eqref{discTrace} together with a repeated use of Young's inequality $ab \leq \frac{a^2}{2\epsilon} + \frac{\epsilon b^2}{2}$ with  adequately selected parameters $\epsilon>0$, we deduce that 
\begin{align}\label{june}
\begin{split}
		\Delta t &\left| \sum_{k= 1}^n  G^k\Big(\partial^0_t \big(\be^n_{h,\bgam}, \be^n_{h,\bze}\big)\Big) \right|  \leq \frac{1}{4} \max_n\mathcal{E}^n_h+ C_0 \Bigg( \max_n\norm*{(\mathfrak{X}_1^n, \mathfrak{X}_2^n)}^2_{0,\Omega} + \max_n\norm*{\mathfrak{X}_3^n}^2_{0,\Omega}  
\\[0.25ex]
& 
 + \max_n\norm*{h_\cF^{\frac{1}{2}}\mean{\bdiv \mathfrak{X}_4^n}}^2_{0,\cF^*_h} + \max_n\norm*{\bdiv  (\partial_t  \mathfrak{X}_4^n)}^2_{0,\Omega} 
 + \max_n\norm*{h_\cF^{\frac{1}{2}}\mean{\bdiv  (\partial_t  \mathfrak{X}_4^n)}}^2_{0,\cF^*_h} \Bigg), 
\end{split}
\end{align}
with $C_0>0$ independent of $h$ and $\Delta t$. Combining \eqref{may} and \eqref{june} with \eqref{full2} gives the result for all $\texttt{a}\geq \texttt{a}_0$ with $\texttt{a}_0$ selected as in Lemma~\ref{stability}.
\end{proof}

We need further regularity hypotheses to estimate the different consistency terms appearing on the right-hand side of \eqref{full4}. 
\begin{assumption}\label{assumption3}
	The body force satisfies  $\bF\in W^{3,\infty}(H^1(\cup_j \Omega_j, \R^d))$.	  
\end{assumption}
\begin{assumption}\label{assumption4}
	The solution $(\bgam, \bze)$ of \eqref{varFormR1-varFormR2} satisfies 
	\begin{enumerate}[label=\roman*)]
	\item $(\bgam,\bze)\in W^{4,\infty}(L^2(\Omega),\bbS\times \bbS)$, 
	\item and $\bdiv(\bgam+\bze)\in W^{3,\infty}(L^2(\Omega),\R^d)$.
	\end{enumerate}
\end{assumption}

We point out that, with assumption Assumption~\ref{assumption3} and Assumption~\ref{assumption4} at hand, we can proceed as in the proof of Lemma~\ref{stability} to deduce that 
\begin{equation*}
\bdiv( \bgam + \bze ) = \rho \ddot \bu - \bF \in W^{3,\infty}(H^1(\cup_j\Omega_j, \R^d)).
\end{equation*}
Moreover, we can perform Taylor expansions centered at $t=t_n$ to obtain the expressions 
\begin{align*}
\mathfrak X_1^n &= \Pi_h\ddot{\bgam}(t_n) - \ddot{\bgam}(t_n) + 
\frac{\Delta t^2}{6 } \int_{-1}^{1} (1-|s|)^3\Pi_h\dfrac{\text{d}^4 \bgam}{\text{d}t^4}(t_n + \Delta t\, s)  \, \text{d}s, 
\\[0.25ex]
\mathfrak X_2^n &= \Pi_h\ddot{\bze}(t_n) - \ddot{\bze}(t_n) + 
\frac{\Delta t^2}{6 } \int_{-1}^{1} (1-|s|)^3\Pi_h\dfrac{\text{d}^4 \bze}{\text{d}t^4}(t_n + \Delta t\, s)  \, \text{d}s, 
\\[0.25ex]
 \mathfrak X_3^n &=  \Pi_h\dot{\bze}(t_n) - \dot{\bze}(t_n) + \frac{\Delta t^2}{2 } \int_{-1}^{1} (1 - |s|)^2 \Pi_h\dfrac{\text{d}^3\, \bze}{\text{d}t^3}(t_n + \Delta t\, s)\, \text{d}s,
 \\[0.25ex]
 \mathfrak X_4^n &= (\Pi_h - I)\big(\bgam(t_n) + \bze(t_n)\big)   +    \frac{\Delta t^2}{4} \int_{-1}^{1} (1 - |s|) \Pi_h(\ddot\bgam + \ddot\bze)(t_n + \Delta t\, s)  \, \text{d}s,
\end{align*}
and
\begin{align*}
\partial_t  \mathfrak X_4^n &= (\Pi_h - I)\partial_t\big(\bgam(t_n) + \bze(t_n)\big) + \Pi_h \frac{(\bgam + \bze)(t_{n+2}) - 3 (\bgam + \bze)(t_{n+1}) + 3(\bgam + \bze)(t_{n}) -(\bgam + \bze)(t_{n-1})}{4\Delta t}
\\[0.25ex]
&= \int_0^1 (\Pi_h - I)\big(\dot\bgam + \dot\bze\big)(t_n + \Delta t\, s) \, \text{d}s +
\Delta t^2 \int_0^1 (1-s)^2  \Pi_h \boldsymbol\alpha(s) \, \text{d}s,
\end{align*}
with
\[
  \boldsymbol\alpha(s):=  \dfrac{\text{d}^3 (\bgam + \bze)}{\text{d}t^3} (t_n + 2\Delta t\, s) -\dfrac{3}{8}\dfrac{\text{d}^3 (\bgam + \bze)}{\text{d}t^3} (t_n + \Delta t\, s) + \dfrac{1}{8}\dfrac{\text{d}^3 (\bgam + \bze)}{\text{d}t^3} (t_n - \Delta t\, s).
\]
Applying the stability property \eqref{scott0a} of $\Pi_h$ with $m=0$ we deduce that  
\begin{multline}\label{un}
  \max_n\norm*{(\mathfrak{X}_1^n, \mathfrak{X}_2^n)}_{0,\Omega} + \max_n\norm*{\mathfrak{X}_3^n}_{0,\Omega} \lesssim \norm*{(I - \Pi_h)\bgam}_{W^{2,\infty}(L^2(\Omega, \bbM))} + \norm*{(I - \Pi_h)\bze}_{W^{2,\infty}(L^2(\Omega, \bbM))}
  \\[0.25ex]
  + (\Delta t)^2 \Big( \norm*{\bgam}_{W^{4,\infty}(L^2(\Omega, \bbM))} + \norm*{\bze}_{W^{4,\infty}(L^2(\Omega, \bbM))} \Big),
\end{multline}
 while \eqref{scott0a} with $m=1$ yields
 \begin{multline*}
 	\max_n \norm*{\tfrac{1}{\sqrt \rho} \bdiv \mathfrak{X}_4^k}_{0,\Omega} + \max_n \norm*{\tfrac{1}{\sqrt \rho} \bdiv \partial_t \mathfrak{X}_4^k}_{0,\Omega}
  \\
 \lesssim \norm*{\bdiv (I - \Pi_h)(\bgam + \bze)}_{W^{1,\infty}(L^2(\Omega, \R^d))}   + (\Delta t)^2  \sum_{j=1}^J \norm*{\bgam + \bze}_{W^{3,\infty}(H^1(\Omega_j, \bbM))}. 
 \end{multline*}
Finally, by virtue of \eqref{scott0c}, it holds
\begin{multline}\label{trois}
	\max_n\norm*{h_\cF^{\frac{1}{2}}\mean{\bdiv \mathfrak{X}_4^n}}_{0,\cF^*_h} 
 + \max_n\norm*{h_\cF^{\frac{1}{2}}\mean{\bdiv  (\partial_t  \mathfrak{X}_4^n)}}_{0,\cF^*_h} 
  \\
\lesssim \norm*{h_\cF^{\frac{1}{2}}\mean{\bdiv (I - \Pi_h)(\bgam + \bze)}}_{W^{1,\infty}(L^2(\cF^*_h, \R^d))}    + (\Delta t)^2  \sum_{j=1}^J \norm*{\bgam + \bze}_{W^{3,\infty}(H^1(\Omega_j, \bbM))}. 
\end{multline}
As a consequence of \eqref{un}-\eqref{trois} and the stability estimate \eqref{full4}, we have that
\begin{align}\label{quatre}
\nonumber
\max_n\mathcal{E}^n_h & \leq  	\norm*{(I - \Pi_h)\bgam}_{W^{2,\infty}(L^2(\Omega, \bbM))} + \norm*{(I - \Pi_h)\bze}_{W^{2,\infty}(L^2(\Omega, \bbM))}
\\
  &\quad +
  \norm*{\bdiv (I - \Pi_h)(\bgam + \bze)}_{W^{1,\infty}(L^2(\Omega, \R^d))} 
  + \norm*{h_\cF^{\frac{1}{2}}\mean{\bdiv (I - \Pi_h)(\bgam + \bze)}}_{W^{1,\infty}(L^2(\cF^*_h, \R^d))}
\nonumber  \\
  &\quad + (\Delta t)^2 \Big( \norm*{\bgam}_{W^{4,\infty}(L^2(\Omega, \bbM))} + \norm*{\bze}_{W^{4,\infty}(L^2(\Omega, \bbM))} + \sum_{j=1}^J \norm*{\bgam + \bze}_{W^{3,\infty}(H^1(\Omega_j, \bbM))}\Big),
\end{align}
for all $\texttt{a}\geq \texttt{a}_0$.
We are now in a position to state the following error estimate.
\begin{theorem}
Let $(\bgam, \bze)$ and $\{ (\bgam^n_h, \bze^n_h),\ n= 0,\ldots, L\}$ be the solutions of \eqref{varFormR1-varFormR2} and \eqref{fullyDiscretePb1-Pb2}, respectively. Under Assumption~\ref{assumption3} and Assumption~\ref{assumption4}, we have that   
\begin{align}\label{fullConv1}
\nonumber
\max_n &\norm*{ (\dot\bgam, \dot\bze )(t_{n+\frac{1}{2}}) - \partial_t (\bgam_h^n, \bze_h^n)}_{0,\Omega}+
 \max_n
\norm*{ \bdiv  (\bgam + \bze)(t_{n+\frac{1}{2}}) - \bdiv_h (\bgam_h^{n+\frac{1}{2}} + \bze_h^{n+\frac{1}{2}})}_{0,\Omega} 
\\
&
+ \max_n \norm*{h_\cF^{-\frac{1}{2}} \jump{ (\bgam + \bze)(t_{n+\frac{1}{2}}) -  (\bgam_h^{n+\frac{1}{2}} + \bze_h^{n+\frac{1}{2}}) } }_{0,\cF_h}
\nonumber\\
&\leq C \bigg( 	\norm*{(I - \Pi_h)\bgam}_{W^{2,\infty}(L^2(\Omega, \bbM))} + \norm*{(I - \Pi_h)\bze}_{W^{2,\infty}(L^2(\Omega, \bbM))}
\nonumber\\
  &+
  \norm*{\bdiv (I - \Pi_h)(\bgam + \bze)}_{W^{1,\infty}(L^2(\Omega, \R^d))} 
  + \norm*{h_\cF^{\frac{1}{2}}\mean{\bdiv (I - \Pi_h)(\bgam + \bze)}}_{W^{1,\infty}(L^2(\cF^*_h, \R^d))}
\nonumber  \\
  &+ (\Delta t)^2 \Big( \norm*{\bgam}_{W^{4,\infty}(L^2(\Omega, \bbM))} + \norm*{\bze}_{W^{4,\infty}(L^2(\Omega, \bbM))} + \sum_{j=1}^J \norm*{\bgam + \bze}_{W^{3,\infty}(H^1(\Omega_j, \bbM))}\Big) \bigg),
\end{align}
for all $\texttt{a}\geq \texttt{a}_0$, with $C>0$ independent of $h$ and $\Delta t$.
\end{theorem}
\begin{proof}
It follows from \eqref{quatre}, the triangle inequality  and the lower bound of \eqref{eq-extra-1DG-fully}  that 
\begin{align}\label{uno}
\nonumber
& \max_n \norm*{ \partial_t(\bgam, \bze )(t_{n}) - \partial_t (\bgam_h^n, \bze_h^n)}_{0,\Omega}+
 \max_n
\norm*{ \bdiv  (\bgam + \bze)(t_{n+\frac{1}{2}}) - \bdiv_h (\bgam_h^{n+\frac{1}{2}} + \bze_h^{n+\frac{1}{2}})}_{0,\Omega} 
\\[0.25ex]
\nonumber &
\quad + \max_n \norm*{h_\cF^{-\frac{1}{2}} \jump{ (\bgam + \bze)(t_{n+\frac{1}{2}}) -  (\bgam_h^{n+\frac{1}{2}} + \bze_h^{n+\frac{1}{2}}) } }_{0,\cF^*_h}
\\[0.25ex]
\nonumber & \lesssim  	\norm*{(I - \Pi_h)\dot\bgam}_{L^{\infty}(L^2(\Omega, \bbM))}  + \norm*{(I - \Pi_h)\dot\bze}_{L^{\infty}(L^2(\Omega, \bbM))}+
  \norm*{\bdiv (I - \Pi_h)(\bgam + \bze)}_{L^{\infty}(L^2(\Omega, \R^d))}
\\[0.25ex]
  &  \quad 
  + \norm*{h_\cF^{\frac{1}{2}}\mean{\bdiv (I - \Pi_h)(\bgam + \bze)}}_{L^{\infty}(L^2(\cF^*_h, \R^d))}+ \max_n\mathcal{E}^n_h.
\end{align}
Moreover, using again the triangle inequality, 
\begin{align*}
\begin{split}
\max_n &\norm*{ (\dot\bgam, \dot\bze )(t_{n+\frac{1}{2}}) - \partial_t (\bgam_h^n, \bze_h^n)}_{0,\Omega} \leq \max_n \norm*{\partial_t(\bgam, \bze )(t_{n}) - \partial_t (\bgam_h^n, \bze_h^n)}_{0,\Omega} 
\\[0.25ex]
&+ \max_n \norm*{ (\dot\bgam, \dot\bze )(t_{n+\frac{1}{2}}) - \partial_t(\bgam, \bze )(t_{n})}_{0,\Omega}
\end{split}
\end{align*}
and the Taylor expansion  
\begin{equation*}
(\dot\bgam, \dot\bze )(t_{n+\frac{1}{2}}) - \partial_t(\bgam, \bze )(t_{n}) =  - \frac{\Delta t^2}{8} \int_{-1}^{1} (1 - |s|)^2 \dfrac{\text{d}^3 (\bgam, \bze )}{\text{d}t^3} (t_{n+\frac{1}{2}} + \frac{\Delta t}{2}s) \, \text{d}s,
\end{equation*}
we obtain the bound 
\begin{align}\label{unodos}
\nonumber
\max_n \norm*{(\dot\bgam, \dot\bze )(t_{n+\frac{1}{2}}) - \partial_t (\bgam_h^n, \bze_h^n)}_{0,\Omega} &\lesssim \max_n \norm*{ \partial_t(\bgam, \bze )(t_{n}) - \partial_t (\bgam_h^n, \bze_h^n) }_{0,\Omega} 
\\
&+ (\Delta t)^2 \max_n \norm*{ (\bgam, \bze )}_{W^{3,\infty}(L^2(\Omega, \bbM\times \bbM))},
\end{align}
and the desired result follows by combining \eqref{quatre}, \eqref{uno} and \eqref{unodos}.  
\end{proof}

\begin{corollary}\label{dot}
Let $(\bgam, \bze)$ and $\{ (\bgam^n_h, \bze^n_h),\ n= 0,\ldots, L\}$ be the solutions of \eqref{varFormR1-varFormR2} and \eqref{fullyDiscretePb1-Pb2}, respectively. If Assumptions~\ref{assumption3} and \ref{assumption4} hold true and if $\bgam$, $\bze$ $\in  W^{2,\infty}(H^r(\cup_j\Omega_j,\bbM))$ and $\bgam +\bze \in W^{1,\infty}(   H^{r+1}(\cup_j\Omega_j,\mathbb R^d))$, with $r\geq 1$, then we have    
\begin{align}\label{fine}
\nonumber
\max_n &\norm*{ (\dot\bgam, \dot\bze )(t_{n+\frac{1}{2}}) - \partial_t (\bgam_h^n, \bze_h^n)}_{0,\Omega}+
 \max_n
\norm*{ \bdiv  (\bgam + \bze)(t_{n+\frac{1}{2}}) - \bdiv_h (\bgam_h^{n+\frac{1}{2}} + \bze_h^{n+\frac{1}{2}})}_{0,\Omega} 
\\
&
+ \max_n \norm*{h_\cF^{-\frac{1}{2}} \jump{ (\bgam + \bze)(t_{n+\frac{1}{2}}) -  (\bgam_h^{n+\frac{1}{2}} + \bze_h^{n+\frac{1}{2}}) } }_{0,\cF_h}
\nonumber\\
&\leq C  	h^{\min\{r,k\}} \sum_{j=1}^J\Big( \norm*{(\bgam , \bze)}_{W^{2,\infty}(H^r(\Omega_j,\bbM\times \bbM))} + \norm*{\bgam + \bze}_{W^{1,\infty}(   H^{r+1}( \Omega_j,\mathbb R^d))} \Big)
 \nonumber \\
  &\quad + C (\Delta t)^2 \Big( \norm*{\bgam}_{W^{4,\infty}(L^2(\Omega, \bbM))} + \norm*{\bze}_{W^{4,\infty}(L^2(\Omega, \bbM))} + \sum_{j=1}^J \norm*{\bgam + \bze}_{W^{3,\infty}(H^1(\Omega_j, \bbM))}\Big),
\end{align}
for all $\texttt{a}\geq \texttt{a}_0$, with $C>0$ independent of $h$ and $\Delta t$.
\end{corollary}
\begin{proof}
The result follows by using the error estimates \eqref{tool} in \eqref{fullConv1}.
\end{proof}

\begin{remark}\label{rem:equiv-norm}
	We notice that 
\begin{align}\label{truco}
\nonumber 
\Big( (\bgam, \bze)(t_{k+\frac{1}{2}}) &- (\bgam_h^{k+\frac{1}{2}}, \bze_h^{k+\frac{1}{2}}) \Big) - \Big((\bgam, \bze)(t_{k-\frac{1}{2}}) - (\bgam_h^{k-\frac{1}{2}}, \bze_h^{k-\frac{1}{2}}) \Big) 
\\
&= (\bgam, \bze)(t_{k+\frac{1}{2}})
- (\bgam, \bze)(t_{k-\frac{1}{2}})
- \frac{\Delta t}{2}((\dot\bgam, \dot\bze)(t_{k+\frac{1}{2}}) + (\dot\bgam, \dot\bze)(t_{k-\frac{1}{2}}))  
\nonumber\\
&\qquad + \frac{\Delta t}{2} \Big( (\dot\bgam, \dot\bze)(t_{k+\frac{1}{2}}) - \partial_t (\bgam_h^{k}, \bze_h^{k})  +  (\dot\bgam, \dot\bze)(t_{k-\frac{1}{2}}) - \partial_t (\bgam_h^{k-1}, \bze_h^{k-1})  \Big).
\end{align}
Then, using a Taylor expansion centered at $t = t_{k}$, we find that
\begin{align}\label{Tay1}
\begin{split}
(\bgam, \bze)(t_{k+\frac{1}{2}})
- (\bgam, \bze)(t_{k-\frac{1}{2}}) &- \frac{\Delta t}{2}((\dot\bgam, \dot\bze)(t_{k+\frac{1}{2}}) + (\dot\bgam, \dot\bze)(t_{k-\frac{1}{2}})) =
 \\[0.25ex]
 &\qquad \frac{(\Delta t)^3}{16} \int_{-1}^{1} \frac{\textup{d}^3(\bgam, \bze)}{\textup{d}t^3}(t_k +\frac{\Delta t}{2}s ) (s^2 - 1)\, \textup{d}s.
\end{split}
\end{align}
In this way, substituting \eqref{Tay1} in \eqref{truco}, 
and summing  up the resulting  identity  over $k= 1,\ldots, n$, we deduce
that  
\begin{align*}
	\max_{n}\norm*{(\bgam, \bze)(t_{n+\frac{1}{2}}) - (\bgam_h^{n+\frac{1}{2}}, \bze_h^{n+\frac{1}{2}})}_{0,\Omega} 
\lesssim (\Delta t)^2 &\Big( \norm*{(\bgam, \bze)}_{W^{3,\infty}(\mathbb{L}^2(\Omega,\bbM\times \bbM))}
\\[0.25ex]
&
+ \max_n \norm*{ (\dot\bgam, \dot\bze )(t_{n+\frac{1}{2}}) - \partial_t (\bgam_h^n, \bze_h^n)}_{0,\Omega} \Big).
\end{align*}
Combining the last identity with \eqref{fine} we deduce that,  under the conditions of Corollary~\ref{dot}, we achieve the following asymptotic error estimate in the energy norm 
\begin{equation}\label{eq:energy-norm}
  \max_{n}\norm*{(\bgam, \bze)(t_{n+\frac{1}{2}}) - (\bgam_h^{n+\frac{1}{2}}, \bze_h^{n+\frac{1}{2}})}_{\mathcal{H}^+_{\text{sym}}(h)} \lesssim h^{\min\{r,k\}} + (\Delta t)^2.  
\end{equation}
 
\end{remark}

\section{Numerical results}\label{sec:results} 
We now present a number of computational tests in 2D and 3D, which have been implemented using the finite element library \texttt{FEniCS} \cite{alnaes15}, and we mention that, due to the use of symmetric spaces, special care is required when manipulating (interpolating, reshaping, projecting locally, etc.) UFL forms and expressions, functions over tensorial finite element spaces, and when updating of arrays in the time advancing algorithm. The stabilisation constant is taken as $\texttt{a} = \texttt{a}^* k^2$, where $k\geq 1$ is the polynomial degree and where $\texttt{a}^*$ is specified in each example. 

\medskip
\noindent\textbf{Example 1: Accuracy verification.}
 In order to investigate numerically the error decay predicted by Corollary \ref{dot} and Remark \ref{rem:equiv-norm}, we proceed to compare approximate and closed-form exact solutions for various levels of spatio-temporal refinement. Let us consider the unit square domain $\Omega = (0,1)^2$ and the following parameter-dependent smooth displacement 
\[
\bu(x,y,t) :=2\exp(-t) \begin{pmatrix}
\cos(\pi x)\sin(\pi y) + \frac{\displaystyle x^2}{\displaystyle \lambda_{\mathcal{C}} + \lambda_{\mathcal{D}}} \\
-\sin(\pi x)\cos(\pi y) + \frac{\displaystyle y^2}{\displaystyle \lambda_{\mathcal{C}} + \lambda_{\mathcal{D}}} \end{pmatrix},
\]
where the parameters come from the constitutive equations characterised by the elastic and viscous stress fourth-order elasticity tensors, here simply assumed as Hooke's law
\begin{equation}\label{eq:hooke} 
\mathcal{C} \btau = 2\mu_{\mathcal{C}}\btau + \lambda_{\mathcal{C}}\text{tr} \btau I, \quad 
  \mathcal{D} \btau = 2\mu_{\mathcal{D}}\btau + \lambda_{\mathcal{D}}\text{tr} \btau I.\end{equation}
The exact displacement is used to construct exact elastic and viscous stresses as well as appropriate initial conditions \eqref{init1} and non-homogeneous boundary conditions as in, e.g., \eqref{eq:bc-nonhomo} (for these first tests of convergence, we only consider them as of displacement type).  The remaining (adimensional) model and numerical parameters are chosen as $\omega = 0.01$, $\rho = 1$, $\texttt{a}^* = 5$.


\begin{table}[t]
	\centering
{\small	\begin{tabular}{||crc|gg|cccccccc||}
	\toprule 
$k$ &{\tt DoF}& $h$ & ${\tt E}^{n+\frac12}(\bsig)$ &  {\tt rate} & ${\tt e}^{n+\frac12}_0(\bsig)$ &  {\tt rate} 
& ${\tt e}^{n+\frac12}_{\bdiv_h}(\bsig)$ & {\tt rate} 
& ${\tt e}^{n+\frac12}_{\text{jump}}(\bsig)$ & {\tt rate} 
& ${\tt e}^{n+\frac12}_0(\bu)$ & {\tt rate}  \\
\midrule
\multirow{6}{*}{1}&
   144 & 0.707 & 3.59e+0 & * & 2.18e-01 & * & 2.48e+0 & * & 8.94e-01 & * & 3.58e-01 & * \\
 &  576 & 0.354 & 1.84e+0 & 0.964 & 5.81e-02 & 1.908 & 1.32e+0 & 0.915 & 4.68e-01 & 0.933 & 1.84e-01 & 0.963\\
 & 2'304 & 0.177 & 9.17e-01 & 1.006 & 1.47e-02 & 1.977 & 6.68e-01 & 0.978 & 2.34e-01 & 0.998 & 9.24e-02 & 0.991\\
 & 9'216 & 0.088 & 4.56e-01 & 1.006 & 3.70e-03 & 1.994 & 3.35e-01 & 0.995 & 1.18e-01 & 0.995 & 4.63e-02 & 0.998\\
 &36'864 & 0.044 & 2.28e-01 & 1.003 & 9.26e-04 & 1.999 & 1.68e-01 & 0.999 & 5.90e-02 & 0.995 & 2.32e-02 & 0.999\\
&147'456 & 0.022 & 1.14e-01 & 1.002 & 2.32e-04 & 2.000 & 8.39e-02 & 1.000 & 2.96e-02 & 0.996 & 1.16e-02 & 1.000\\
\midrule
	\multirow{6}{*}{2} &
   288 & 0.707 & 1.13e+0 & * & 4.85e-02 & * & 8.68e-01 & * & 2.10e-01 & * & 2.94e-01 & *\\
&  1'152 & 0.354 & 3.07e-01 & 1.875 & 6.45e-03 & 2.912 & 2.31e-01 & 1.910 & 6.95e-02 & 1.593 & 8.49e-02 & 1.794\\
&  4'608 & 0.177 & 7.87e-02 & 1.963 & 8.18e-04 & 2.978 & 5.87e-02 & 1.976 & 1.92e-02 & 1.856 & 2.20e-02 & 1.949\\
& 18'432 & 0.088 & 1.98e-02 & 1.990 & 1.03e-04 & 2.994 & 1.47e-02 & 1.994 & 4.98e-03 & 1.946 & 5.55e-03 & 1.987\\
& 73'728 & 0.044 & 4.96e-03 & 1.997 & 1.29e-05 & 2.989 & 3.69e-03 & 1.999 & 1.26e-03 & 1.978 & 1.39e-03 & 1.996\\
&294'912 & 0.022 & 1.24e-03 & 2.000 & 2.19e-06 & 2.560 & 9.21e-04 & 2.001 & 3.18e-04 & 1.990 & 3.48e-04 & 1.997\\
\midrule  
	\multirow{6}{*}{3} &
   480 & 0.707 & 2.69e-01 & * & 9.54e-03 & * & 2.20e-01 & * & 4.08e-02 & * & 4.47e-02 &* \\
&  1'920 & 0.354 & 3.60e-02 & 2.905 & 1.64e-03 & 3.918 & 2.93e-02 & 2.909 & 6.11e-03 & 2.740 & 6.06e-03 & 2.884\\
&  7'680 & 0.177 & 4.57e-03 & 2.977 & 3.58e-04 & 3.978 & 3.72e-03 & 2.977 & 8.11e-04 & 2.912 & 7.74e-04 & 2.971\\
& 30'720 & 0.088 & 5.74e-04 & 2.993 & 5.69e-05 & 3.731 & 4.67e-04 & 2.995 & 1.04e-04 & 2.960 & 9.72e-05 & 2.992\\
&122'880 & 0.044 & 7.32e-05 & 2.969 & 8.50e-06 & 2.945 & 5.85e-05 & 2.995 & 1.32e-05 & 2.981 & 1.22e-05 & 2.994\\
&491'520 & 0.022 & 1.01e-05 & 2.833 & 1.19e-06 & 2.821 & 8.65e-06 & 2.759 & 1.66e-06 & 2.990 & 1.59e-06 & 2.985\\
\bottomrule
	\end{tabular}}
	\caption{Example 1. Error history associated with the space discretisation for polynomial degrees $k=1,2,3$,  obtained for a fixed $\Delta t=10^{-6}$, going up to $T = 5\Delta t$,  and setting the parameters $E_{\mathcal{C} }=1,\nu_{\mathcal{C} }=0.25$, $E_{\mathcal{D} }=10,\nu_{\mathcal{D} }=0.4$, leading to the Lam\'e constants 
	$\mu_{\mathcal{C} } = \lambda_{\mathcal{C} } = 0.4$,  $\mu_{\mathcal{D} } = 3.5714$, $\lambda_{\mathcal{C} } =  14.2857$. The error decay in the energy norm, predicted by \eqref{eq:energy-norm}, and its associated convergence rates are highlighted. }\label{table:h}
\end{table} 

For tables and figures presenting accuracy verification, we use the following notation for the norms in Corollary \ref{dot}, as well as in Remark \ref{rem:equiv-norm}, and denoting separately the $L^2-$, discrete divergence, and jump contributions 
\begin{align*}
{\tt e}^{n+\frac12}_0(\bsig) &:= \max_{n}\norm*{(\bgam, \bze)(t_{n+\frac{1}{2}}) - (\bgam_h^{n+\frac{1}{2}}, \bze_h^{n+\frac{1}{2}})}_{0,\Omega}, \\
 {\tt e}^{n+\frac12}_{\bdiv_h}(\bsig) & : =  \max_n \norm*{ \bdiv  (\bgam + \bze)(t_{n+\frac{1}{2}}) - \bdiv_h (\bgam_h^{n+\frac{1}{2}} + \bze_h^{n+\frac{1}{2}})}_{0,\Omega},\\ 
{\tt e}^{n+\frac12}_{\text{jump}}(\bsig) &: =  \max_n \norm*{h_\cF^{-\frac{1}{2}} \jump{ (\bgam + \bze)(t_{n+\frac{1}{2}}) -  (\bgam_h^{n+\frac{1}{2}} + \bze_h^{n+\frac{1}{2}}) } }_{0,\cF_h},\\
{\tt E}^{n+\frac12}(\bsig) &: =   \max_{n}\norm*{(\bgam, \bze)(t_{n+\frac{1}{2}}) - (\bgam_h^{n+\frac{1}{2}}, \bze_h^{n+\frac{1}{2}})}_{\mathcal{H}^+_{\text{sym}}(h)} , \qquad {\tt e}^{n+\frac12}_0(\bu)  := \max_{n}\norm*{\bu(t_{n+\frac{1}{2}})  - \bu^{n+\frac12}_h}_{0,\Omega},
\end{align*}
where  the approximate displacements are postprocessed using the fully-discrete form of the momentum balance equation and applying a classical finite difference quadrature  
\[ \bu_h^{n+\frac12} = 2\bu_h^{n-\frac12}-\bu_h^{n-\frac32} +\frac{(\Delta t)^2}{\rho}\biggl[\bF+\bdiv_h(\bgam_h^{n+\frac12}+\bze_h^{n+\frac12})\biggr].\]

First we assess the convergence with respect to the space discretisation. As usual, we take a fixed $\Delta t$  (sufficiently small not to compromise the spatial accuracy), run the simulation over a short time horizon (here, of five time steps), and consider a sequence of six successively refined uniform meshes. The rates of convergence in space are computed as 
\[{\tt rate}  =\log(e_{(\cdot)}/\tilde{e}_{(\cdot)})[\log(h/\tilde{h})]^{-1}, \]
where $e,\tilde{e}$ denote errors generated on two consecutive  meshes of sizes $h$ and~$\tilde{h}$, respectively. We test with three different polynomial degrees.  Table~\ref{table:h} presents errors against the number of degrees of freedom (DoF). We can observe that the sum of elastic and viscous stresses and the postprocessed displacement converge to the corresponding exact fields approaching an optimal rate of $O(h^{k})$.   Note that, for the case $k=3$, the convergence of the first contribution to the stress error is slightly affected for the finest level as the error approaches the chosen value for the time step $\Delta t = 10^{-6}$.  
 The convergence has been assessed using mild parameters for Hooke's constitutive laws of elastic and viscous stresses. Varying these parameters does not seem to affect the convergence order of the method. We consider three other parameter sets  (Young's modulus and Poisson ratio) 
 \begin{gather*}
 E_{\mathcal{C}} \in \{10,100,1'000\}, \quad \nu_{\mathcal{C}} \in \{0.33, 0.45, 0.499\},\quad 
 E_{\mathcal{D}} \in \{100,1'000,10'000\}, \quad \nu_{\mathcal{D}} \in \{0.4, 0.475, 0.4999\},
 \end{gather*}
 which give Lam\'e constants 
 $\lambda_{(\cdot)} = \frac{E_{(\cdot)} \nu_{(\cdot)}}{(1+\nu_{(\cdot)})(1-2\nu_{(\cdot)})}$ and $\mu_{(\cdot)} = \frac{E_{(\cdot)}}{2(1+\nu_{(\cdot)})}$. In view of \eqref{eq:hooke}, we recall that the dissipativity condition requires that $\mu_{\mathcal{D} } >\mu_{\mathcal{C} }$ and $\lambda_{\mathcal{D} } >\lambda_{\mathcal{C} }$. 
  The error decay for these three cases is collected in the top panels of Figure~\ref{fig:convergence}, where we only display the energy error and that of the postprocessed displacement.  These results demonstrate the ability of the proposed family of numerical schemes to produce accurate approximations also in the regime where $\lambda_{\mathcal{D} } \approx$1.6e6 (nearly incompressible viscoelasticity). For sake of illustration we also present the obtained approximate stress components and the postprocessed displacement for one of these additional tests.

\begin{figure}[t!]
\begin{center}
\includegraphics[width=0.325\textwidth]{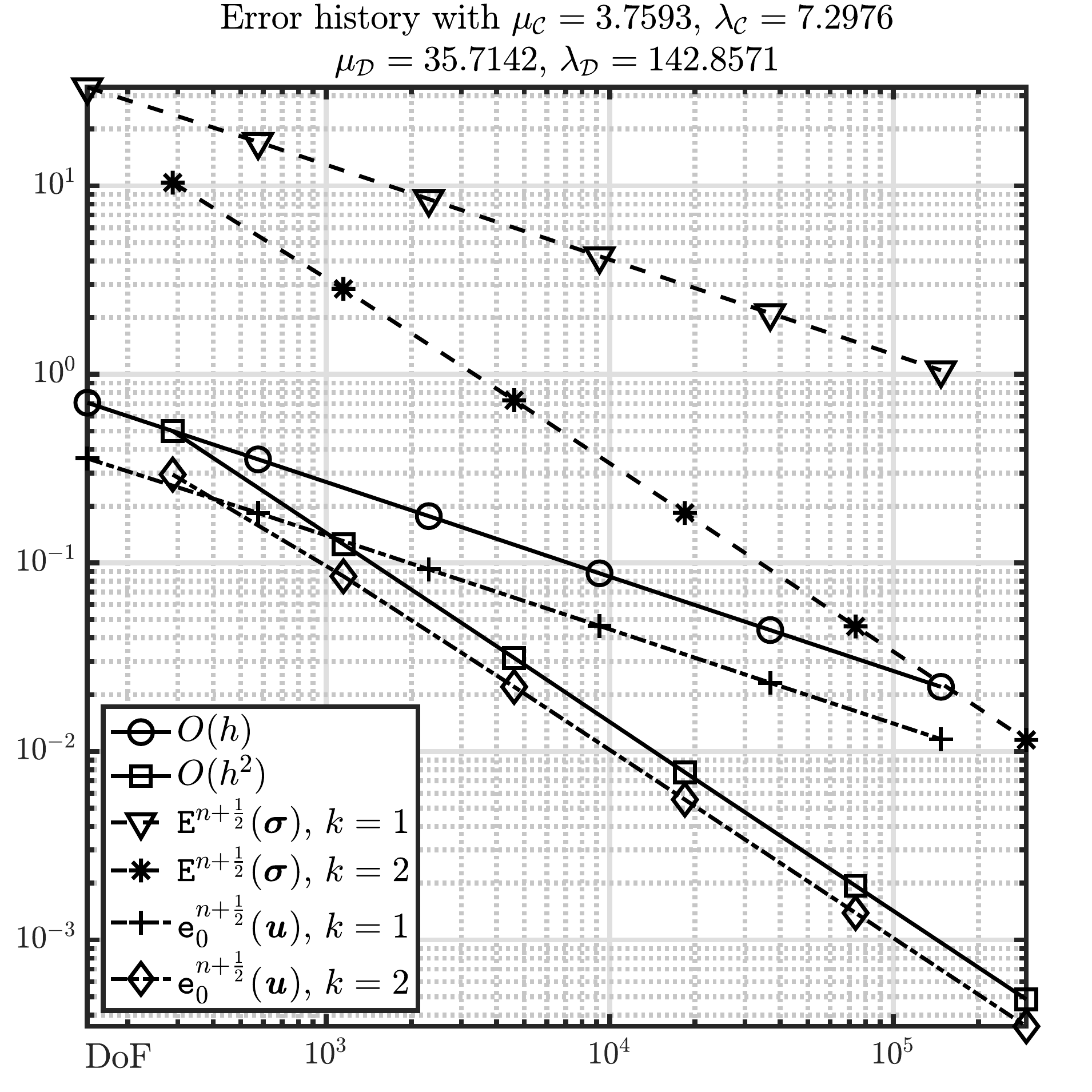}
\includegraphics[width=0.325\textwidth]{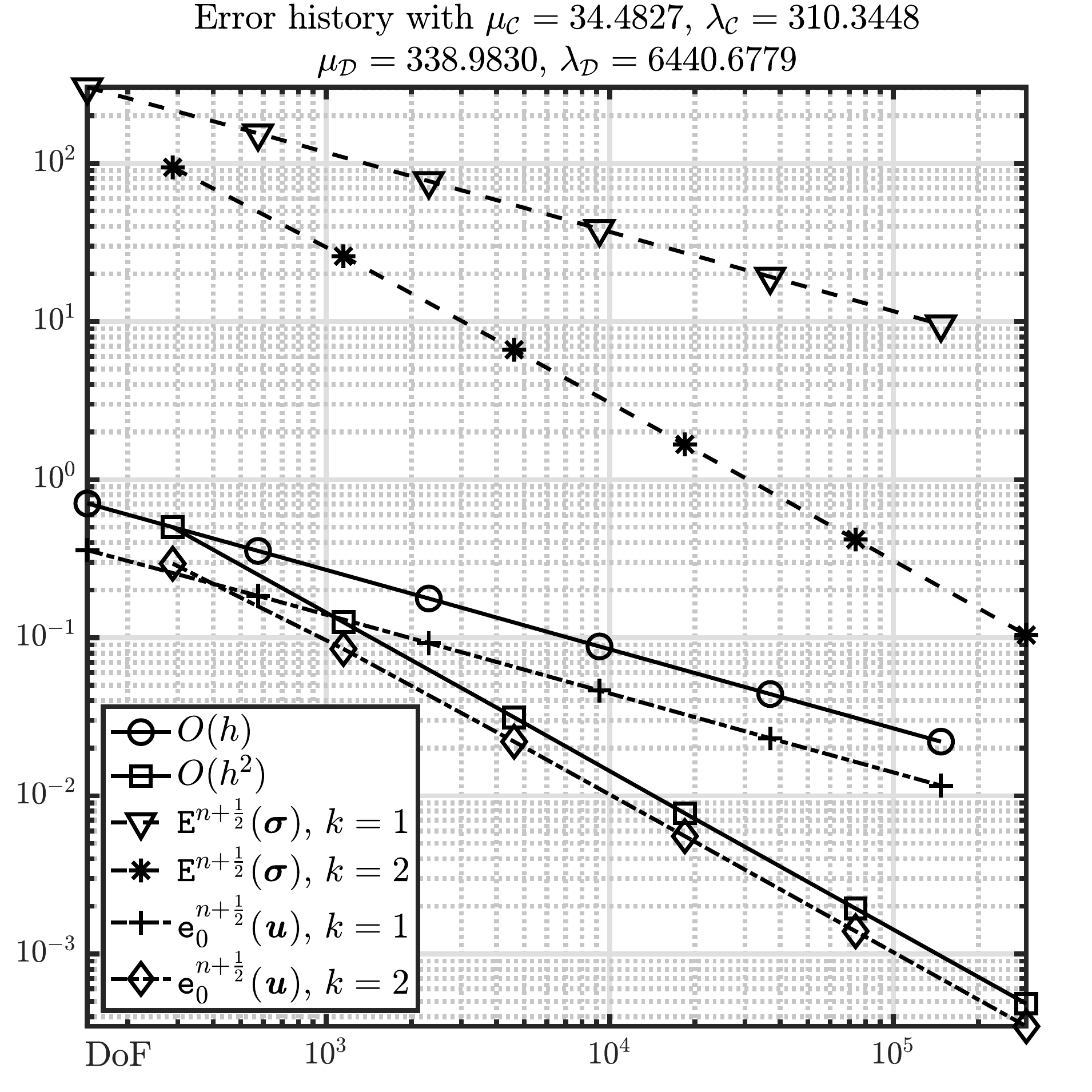}
\includegraphics[width=0.325\textwidth]{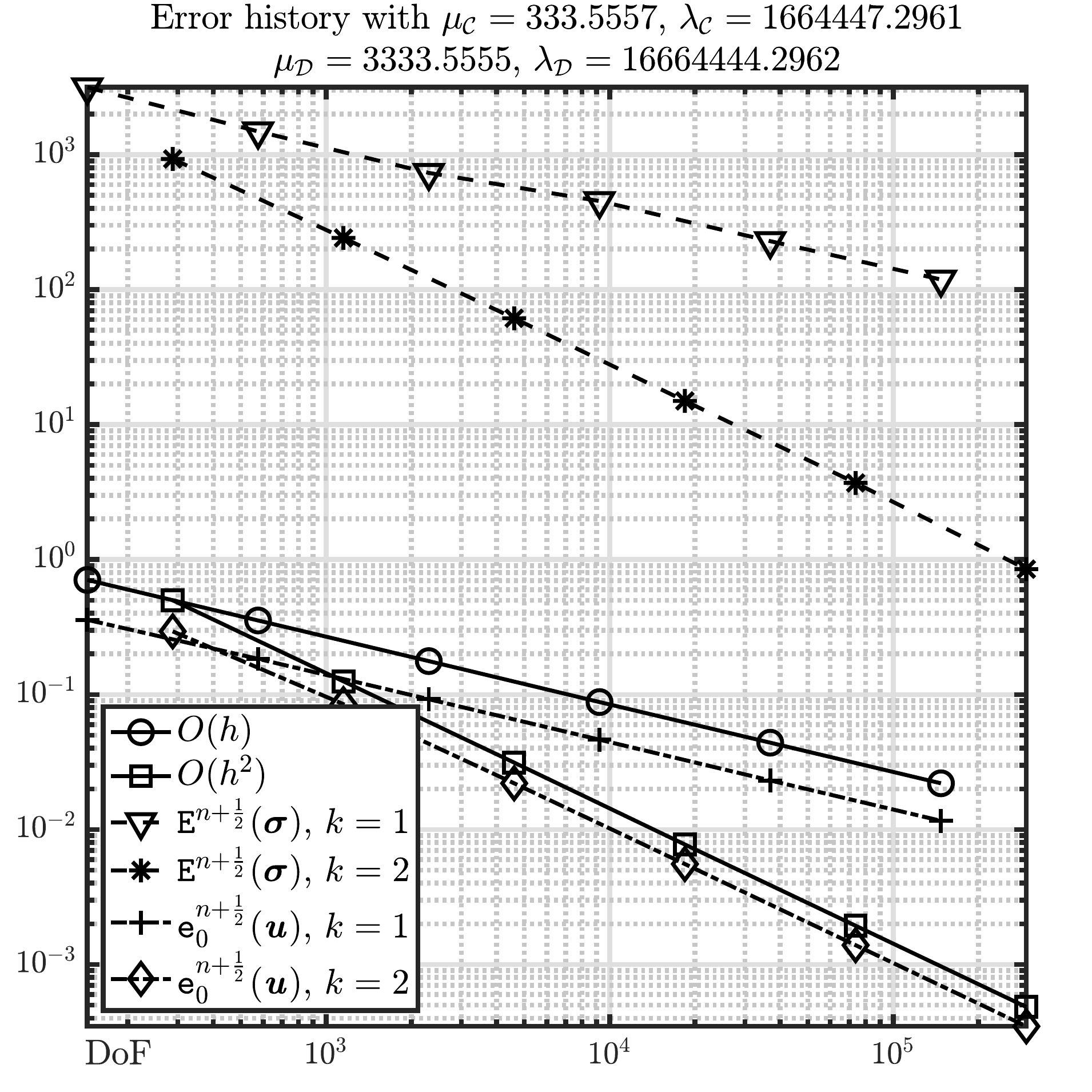}\\[1ex]
\includegraphics[width=0.245\textwidth]{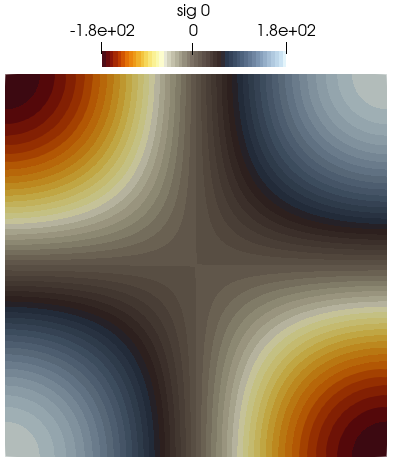}
\includegraphics[width=0.245\textwidth]{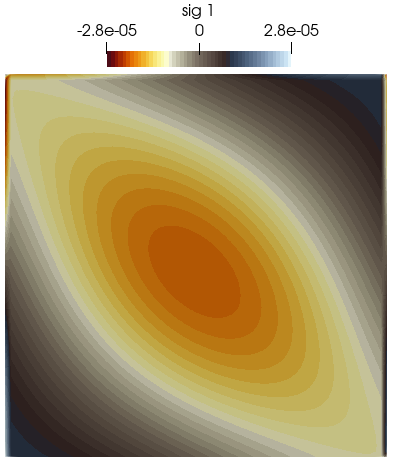}
\includegraphics[width=0.245\textwidth]{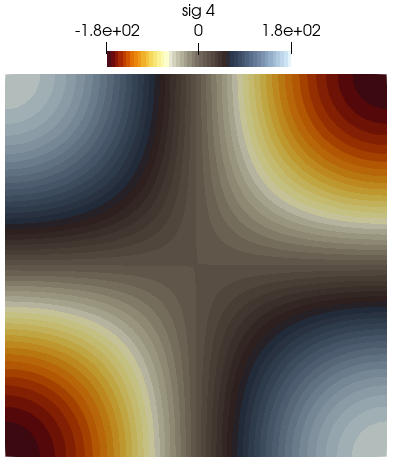}
\includegraphics[width=0.245\textwidth]{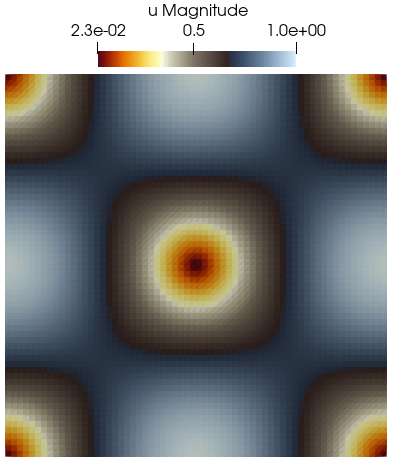}
\end{center}

\vspace{-4mm}
\caption{Example 1. Error history with respect to the space discretisation with polynomial degrees $k=1,2$ and varying the parameter space (top), and sample of approximate stress components ($xx$, $xy=yx$ and $yy$) and postprocessed displacement magnitude  at $t = 5\Delta t$ for the second parameter set  obtained with $k=1$ (bottom row).}\label{fig:convergence}
\end{figure}

\begin{table}[t]
	\centering
{\small	\begin{tabular}{||cc|gg|cccccccc||}
	\toprule 
$k$ &$\Delta t $ & ${\tt E}^{n+1/2}(\bsig)$ &  $\widehat{\tt rate}$  & ${\tt e}^{n+1/2}_0(\bsig)$ &  $\widehat{\tt rate}$ 
& ${\tt e}^{n+1/2}_{\bdiv_h}(\bsig)$ & $\widehat{\tt rate}$ 
& ${\tt e}^{n+1/2}_{\text{jump}}(\bsig)$ & $\widehat{\tt rate}$ 
& ${\tt e}^{n+1/2}_0(\bu)$ & $\widehat{\tt rate}$   \\
\midrule
\multirow{6}{*}{1}&
0.500000 & 6.19e+01 & *      & 1.51e+01 & *     & 4.69e+01 & *     & 5.44e-02 & *     & 2.72e+01 & * \\
&0.250000 & 1.41e+01 & 2.138 & 3.37e+00 & 2.158 & 1.07e+01 & 2.135 & 3.35e-02 & 1.216 & 5.39e+00 & 2.337\\
&0.125000 & 2.97e+00 & 2.246 & 6.73e-01 & 2.325 & 2.27e+00 & 2.230 & 1.95e-02 & 0.782 & 1.16e+00 & 2.217\\
&0.062500 & 7.84e-01 & 1.920 & 2.01e-01 & 1.811 & 5.67e-01 & 2.003 & 1.08e-02 & 0.855 & 2.69e-01 & 2.105\\
&0.031250 & 1.91e-01 & 2.039 & 4.92e-02 & 2.063 & 1.38e-01 & 2.041 & 3.26e-03 & 1.757 & 6.49e-02 & 2.053\\
&0.015625 & 4.84e-02 & 1.978 & 1.24e-02 & 1.993 & 3.49e-02 & 1.982 & 8.16e-04 & 1.958 & 1.60e-02 & 2.024\\
\midrule
	\multirow{6}{*}{2} & 
0.500000 & 6.19e+01 & *     & 1.51e+01 &    *  & 4.69e+01 & *     & 7.40e-03 &   *    & 2.72e+01 & * \\
&0.250000 & 1.40e+01 & 2.141 & 3.37e+00 & 2.158 & 1.07e+01 & 2.135 & 2.45e-03 & 1.724 & 5.39e+00 & 2.337\\
&0.125000 & 2.95e+00 & 2.252 & 6.73e-01 & 2.325 & 2.27e+00 & 2.230 & 1.62e-03 & 1.200 & 1.16e+00 & 2.217\\
&0.062500 & 7.74e-01 & 1.929 & 2.01e-01 & 1.811 & 5.67e-01 & 2.002 & 7.17e-04 & 0.953 & 2.69e-01 & 2.105\\
&0.031250 & 1.88e-01 & 2.044 & 4.92e-02 & 2.063 & 1.38e-01 & 2.040 & 2.33e-04 & 1.730 & 6.49e-02 & 2.053\\
&0.015625 & 4.75e-02 & 1.982 & 1.24e-02 & 1.993 & 3.50e-02 & 1.980 & 6.56e-05 & 1.972 & 1.60e-02 & 2.024\\
\bottomrule
	\end{tabular}}
	\caption{Example 1. Error history associated with the time discretisation, and obtained for a fixed mesh with $ h = 0.022$ and setting the parameters $E_{\mathcal{C} }=10,\nu_{\mathcal{C} }=0.4$, $E_{\mathcal{D} }=20,\nu_{\mathcal{D} }=0.45$, leading to the Lam\'e constants $\mu_{\mathcal{C} } = 3.5714$,  $\lambda_{\mathcal{C} } =14.2857$,  $\mu_{\mathcal{D} } = 6.8966$,  $\lambda_{\mathcal{D} } = 62.0689$.}\label{table:dt}
\end{table}

On the other hand, Table~\ref{table:dt} portrays the convergence results obtained after varying the time step discretising the time interval $[0,1]$. The rates of convergence in time, are computed as 
\[\widehat{\tt rate}  =\log(e_{(\cdot)}/\tilde{e}_{(\cdot)})[\log(\Delta t/\widetilde{\Delta t})]^{-1}, \]
where $e,\tilde{e}$ denote errors generated on two consecutive runs considering time steps $\Delta t$ and~$\widetilde{\Delta t}$, respectively. 
For this we choose a uniform mesh with  $ h = 0.022$ and consider the manufactured solutions
\[ \bu(x,y,t) = \exp(-t) \begin{pmatrix} 
xy  + \frac{\displaystyle x^2}{\displaystyle \lambda_{\mathcal{C}} + \lambda_{\mathcal{D}}} \\
xy  +  \frac{\displaystyle y^2}{\displaystyle \lambda_{\mathcal{C}} + \lambda_{\mathcal{D}}}\end{pmatrix},\]
together with the parameters $\texttt{a}^* =10$, $\omega =2$, $\rho = 1$, $\mu_{\mathcal{C} } = 3.5714$,  $\lambda_{\mathcal{C} } =14.2857$,  $\mu_{\mathcal{D} } = 6.8966$,  $\lambda_{\mathcal{D} } = 62.0689$. The expected convergence rate of $O([\Delta t]^2)$ is attained as the time step is refined. Similar parametric studies to those performed before (now shown here) have also confirmed  robustness with respect to other values in the parameter space.  

\begin{figure}[t!]
\begin{center}
\includegraphics[width=0.245\textwidth]{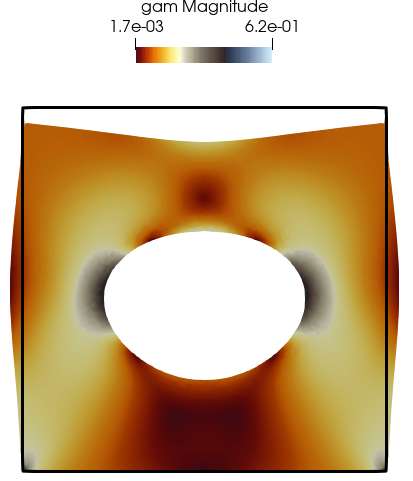}
\includegraphics[width=0.245\textwidth]{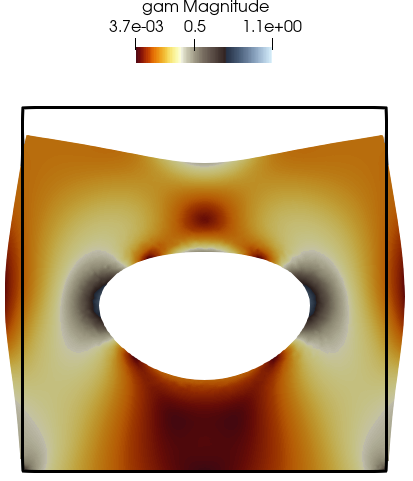}
\includegraphics[width=0.245\textwidth]{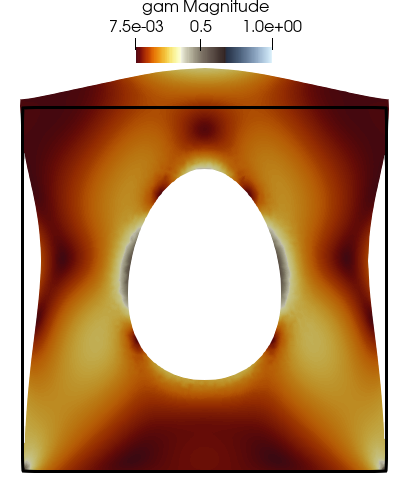}
\raisebox{-2mm}{\includegraphics[width=0.245\textwidth]{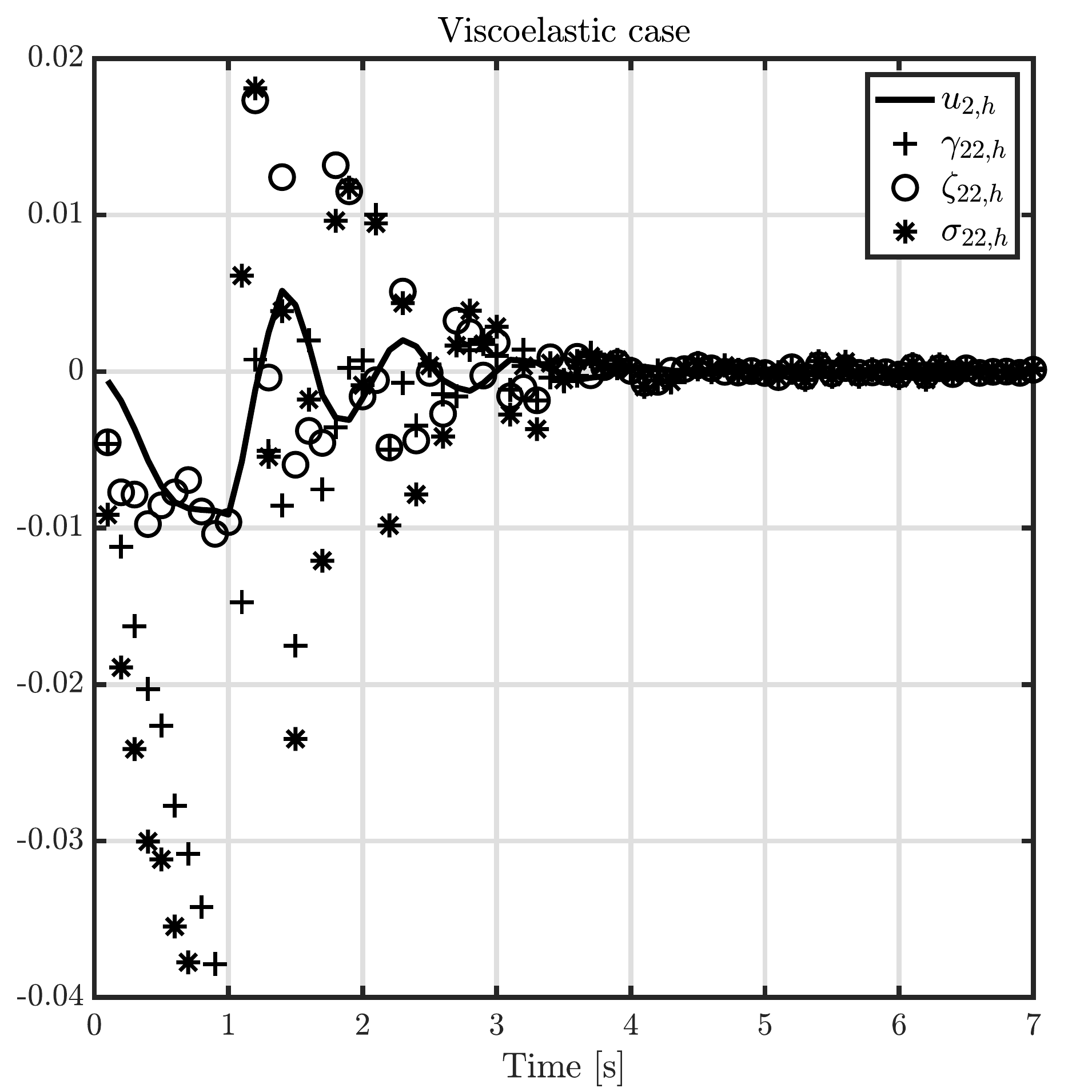}}\\
\includegraphics[width=0.245\textwidth]{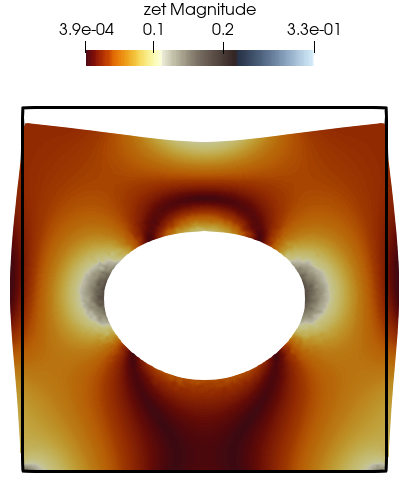}
\includegraphics[width=0.245\textwidth]{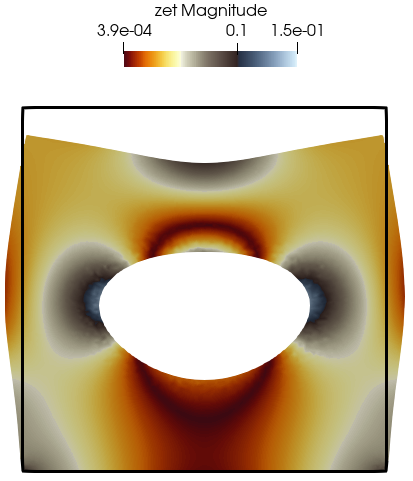}
\includegraphics[width=0.245\textwidth]{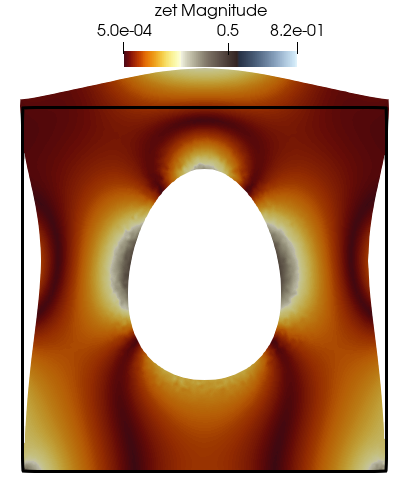}
\raisebox{-2mm}{\includegraphics[width=0.245\textwidth]{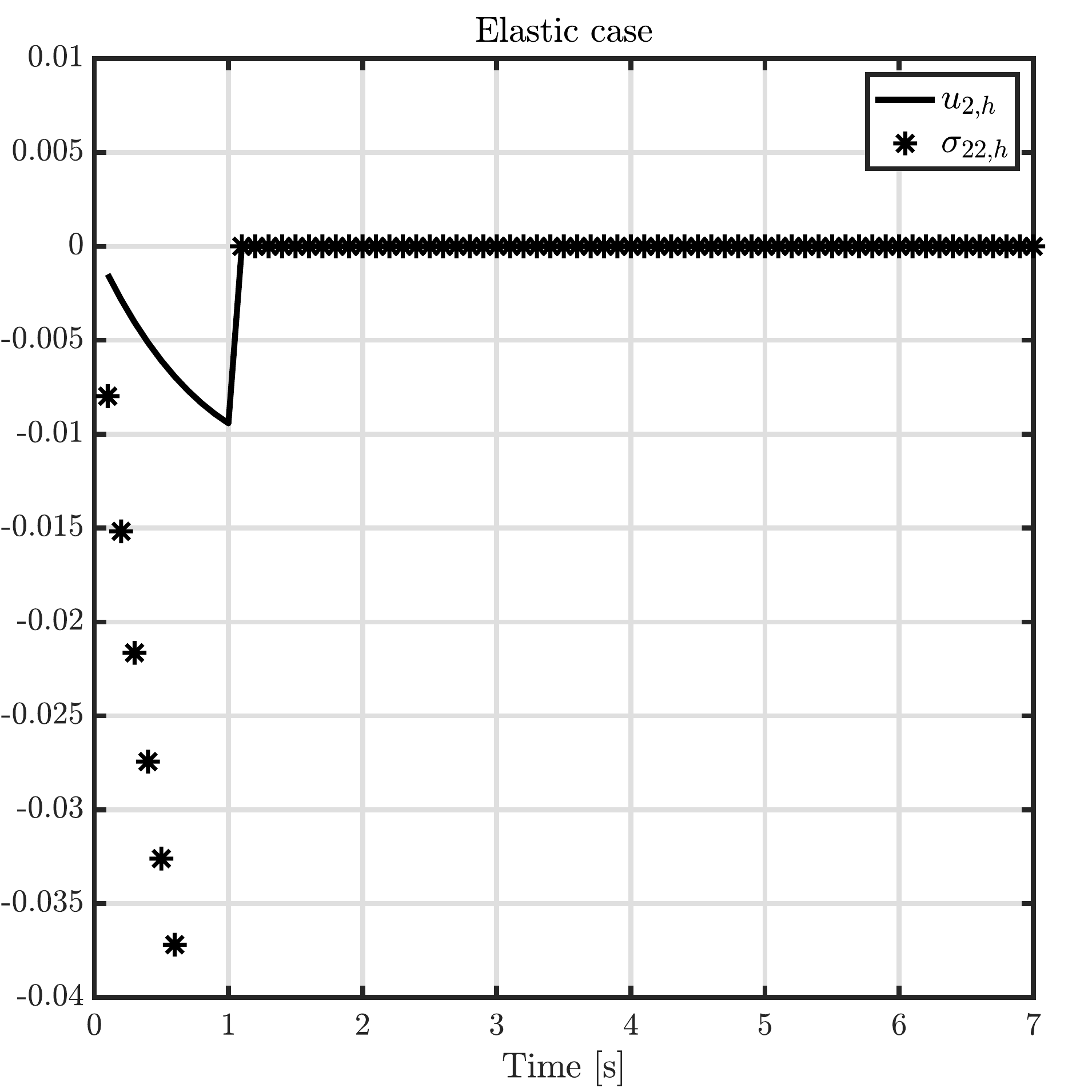}}\\
\includegraphics[width=0.245\textwidth]{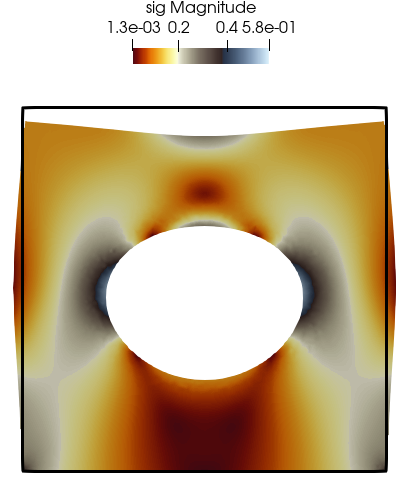}
\includegraphics[width=0.245\textwidth]{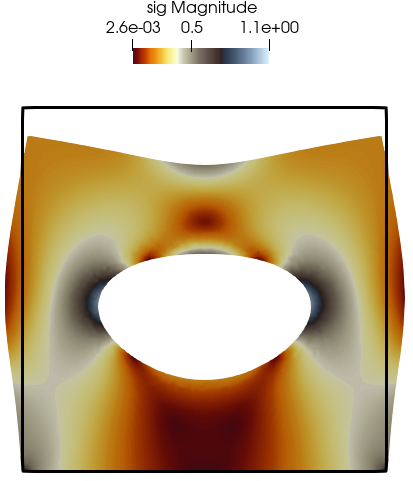}
\includegraphics[width=0.245\textwidth]{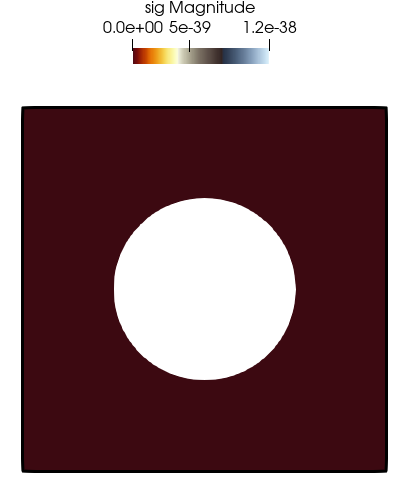}
\raisebox{-2mm}{\includegraphics[width=0.245\textwidth]{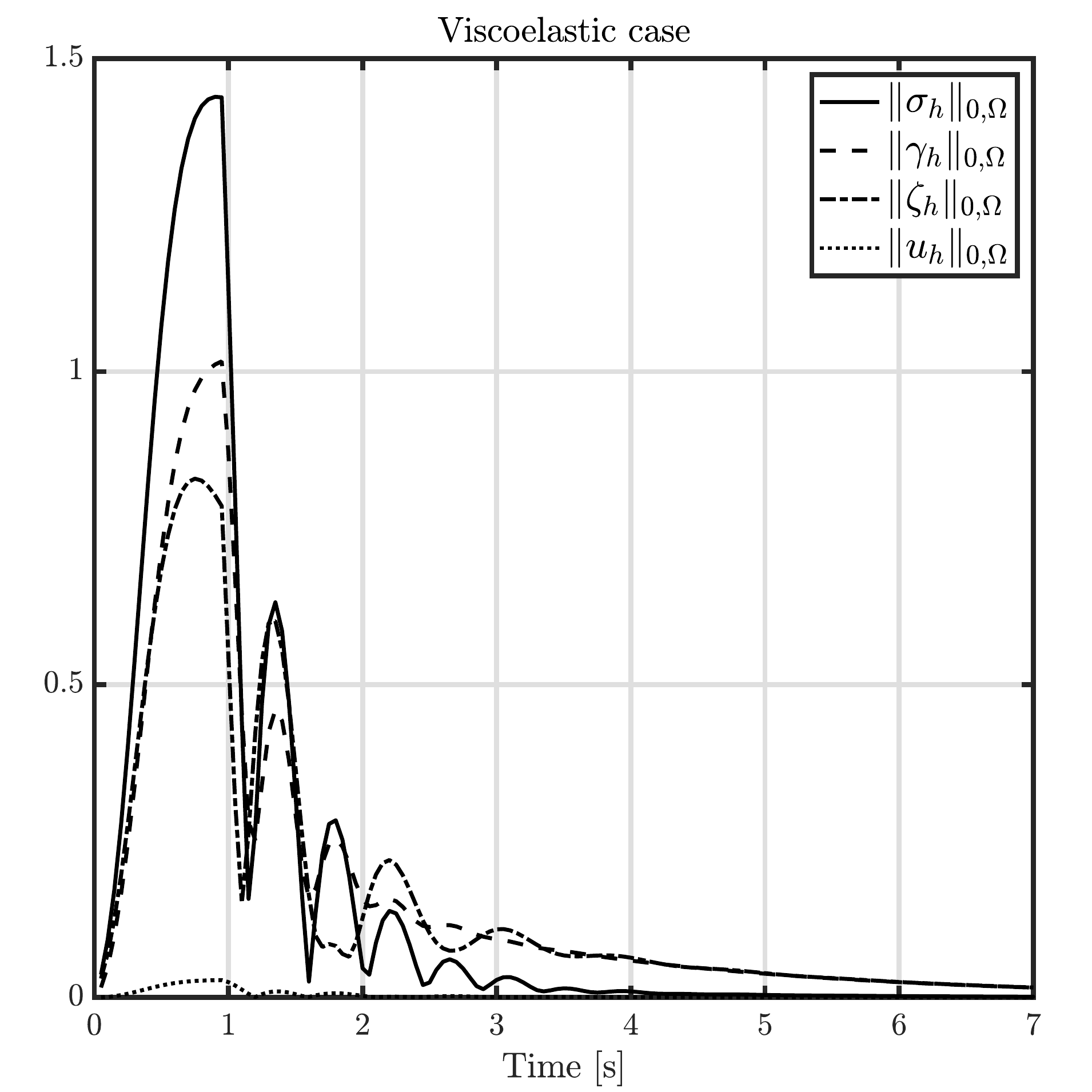}}
\end{center}

\vspace{-5mm}
\caption{Example 2. Snapshots at 10x displacement magnification taken at times  0.5\,s, 1\,s, 1.5\,s (left and two middle columns, respectively) of elastic and viscoelastic stresses for a viscoelastic material (top and middle row) and for an elastic plate (bottom row). Solutions obtained with a second order method. Left column plots: transients of vertical displacement and $yy-$components of elastic and viscous stresses at the point (0.5,0.82), as well as $L^2$-norms of all quantities.}\label{fig:plates}
\end{figure}

\medskip
\noindent\textbf{Example 2: Plane stress viscoelasticity in perforated plates.}
Our next example simulates the transient behaviour of perforated plates in viscoelastic vs purely elastic cases, focusing on plane stress conditions. We adapt the configuration proposed in \cite{giorla14} to Zener's rheological  model and take Hookean constitutive laws for the elastic spring and viscous dashpot stresses selecting the density $\rho =1$\,Kg/m$^3$,  characteristic time $\omega = 0.15$\,s, stabilisation parameter with $\texttt{a}^* = 10$, and   Young moduli and Poisson ratios:
 \[
 E_{\mathcal{C}} = 30\, \text{KPa}, \quad \nu_{\mathcal{C}} = 0.3,\quad 
 E_{\mathcal{D}}= 40\, \text{KPa}, \quad \nu_{\mathcal{D}} = 0.49.
 \]
The domain is a square plate with a circular hole of radius 0.25\,m: $\Omega = (0,1)^2\setminus B_{0.25}(0.5,0.5)\,\text{m}^2$ and the domain boundaries are split into the bottom segment $\Gamma_D $ 
on which we impose zero displacements, and the remainder of the boundary $\Gamma_N$ where we prescribe normal stresses. On the top segment we set a time-dependent traction 
$\boldsymbol{g}_N = \bigl(0,-\frac12H(t\leq 1)\sin(\pi t/5)\bigr)^{\tt t}$, where $H$ is the Heaviside function (meaning that a sinusoidal load is applied on the top edge until $t=1$\,s and then it is suddenly released), and on the vertical and circle sub-boundaries we set a traction free condition $(\bgam + \bze)\bn = \boldsymbol{0}$. 
 We use the numerical method in \eqref{fullyDiscretePb1-Pb2} with polynomial degree $k=2$ (yielding an overall quadratic order of convergence) and consider a final time of $T = 15$\,s, with a time step of $\Delta t = 0.1$\,s. The  unstructured mesh  contains 23'275 triangular elements. We plot in the left panels of Figure~\ref{fig:plates} three snapshots (at times 0.5\,s, 1\,s, 1.5\,s) of the elastic and viscous stress magnitudes portrayed on the deformed domain (where the displacement is magnified by a factor of 10 to assist a better visualisation). For comparison we also plot (in the third row) solution snapshots for  the purely elastic case (and where, after removing the load, the body immediately goes back to the undeformed configuration). Moreover, we  show in the rightmost  three panels, the evolution (only until $t =7$\,s) of axial stress in the vertical direction (in KPa) and vertical displacements (in m) at a point located between the circular hole and the top sub-boundary, as well as the dynamic behaviour of the $L^2-$norms of stresses and displacement. The results illustrate the expected dissipation property.

\begin{figure}[t!]
\begin{center}
\includegraphics[width=0.25\textwidth]{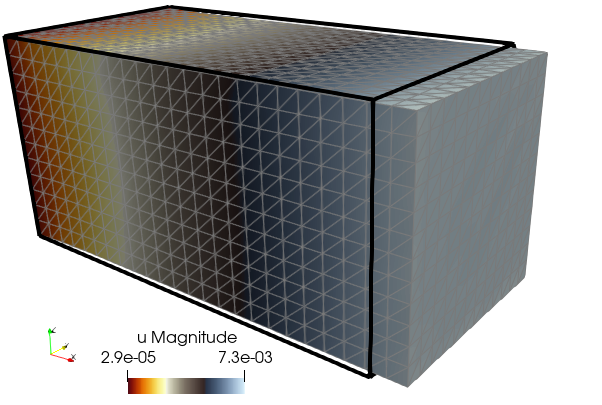}
\includegraphics[width=0.25\textwidth]{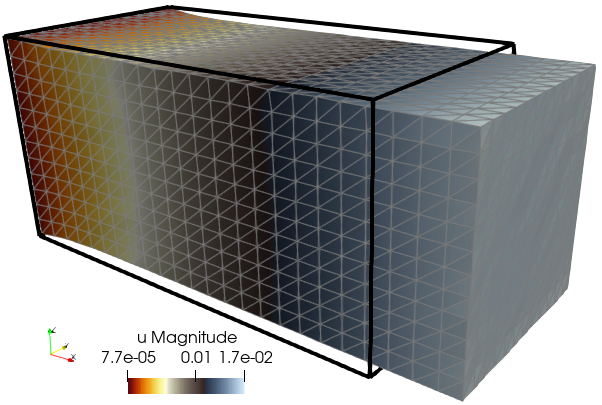}
\includegraphics[width=0.25\textwidth]{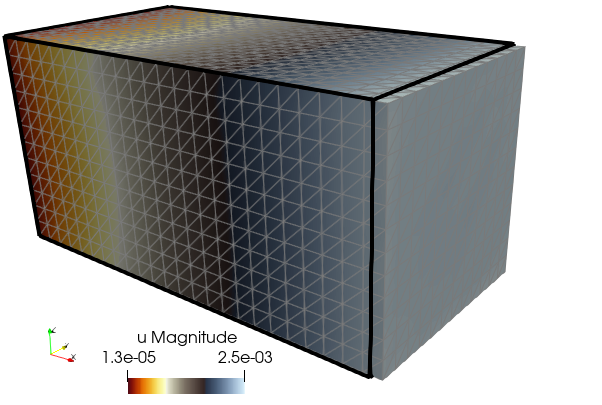}\\
\includegraphics[width=0.25\textwidth]{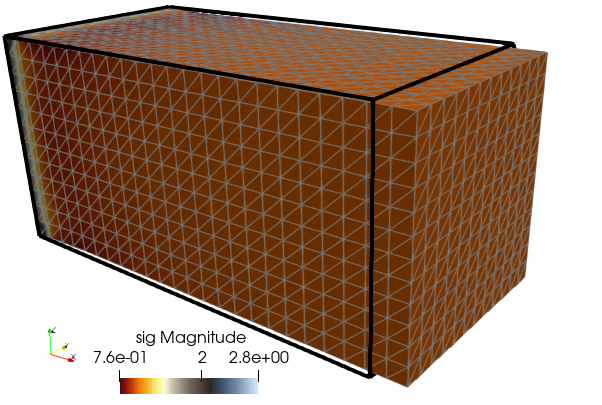}
\includegraphics[width=0.25\textwidth]{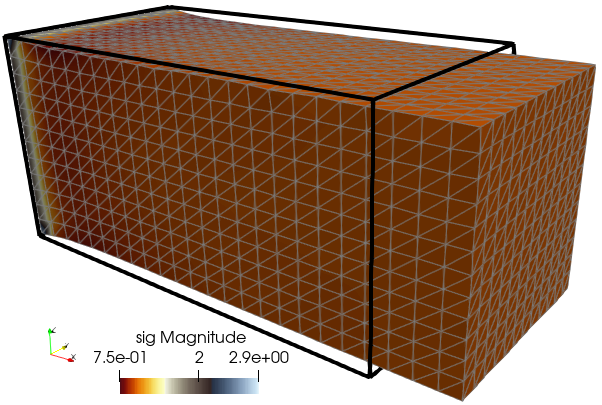}
\includegraphics[width=0.25\textwidth]{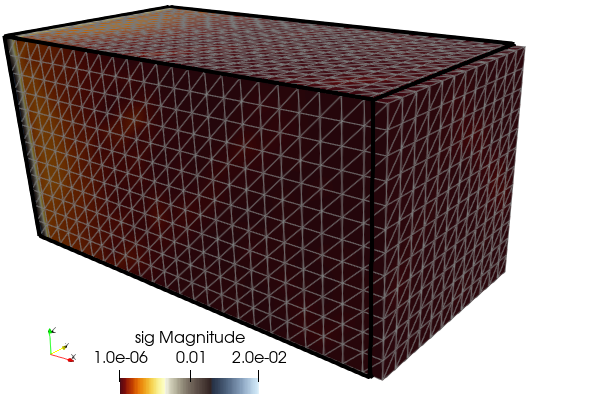}\\
\includegraphics[width=0.245\textwidth]{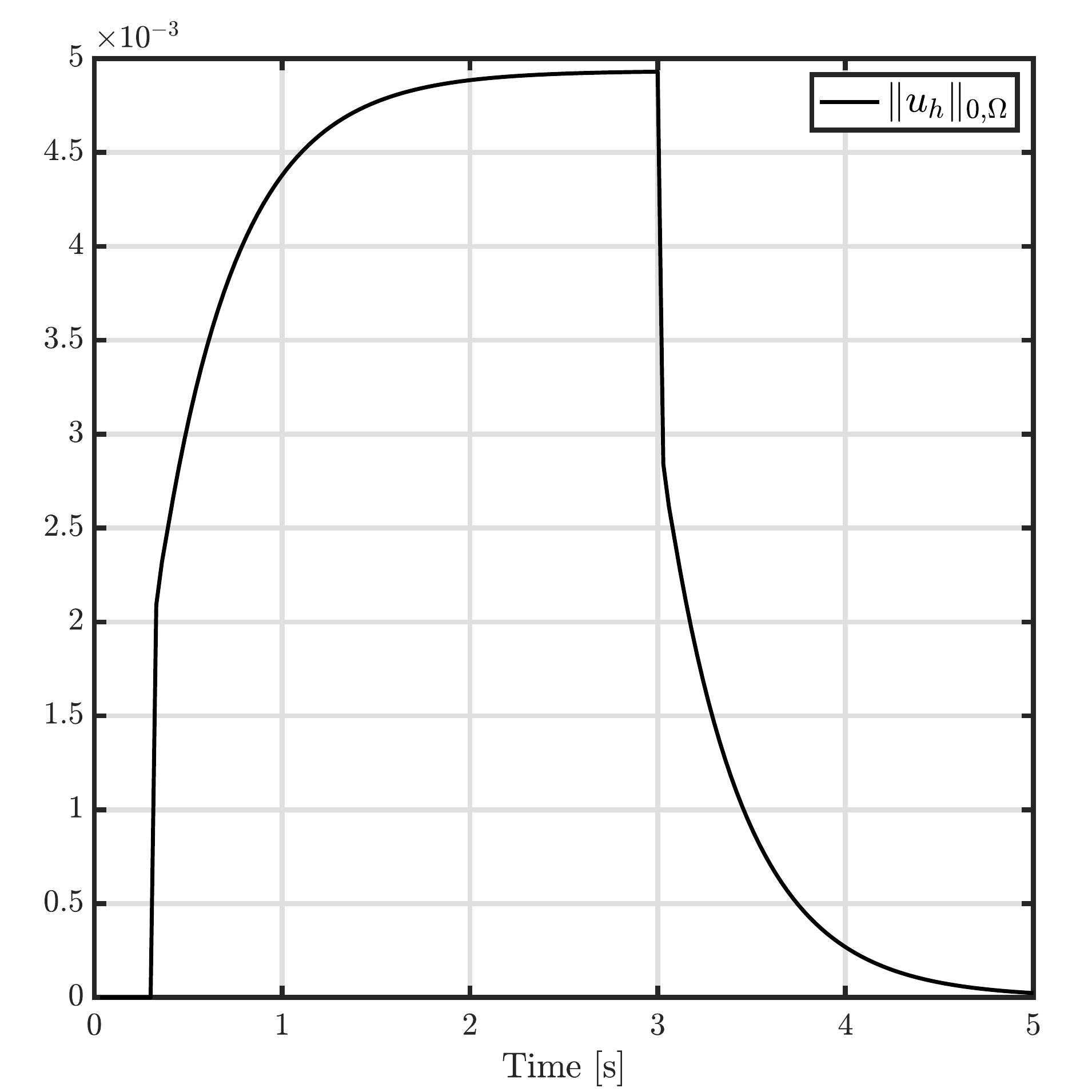}
\includegraphics[width=0.245\textwidth]{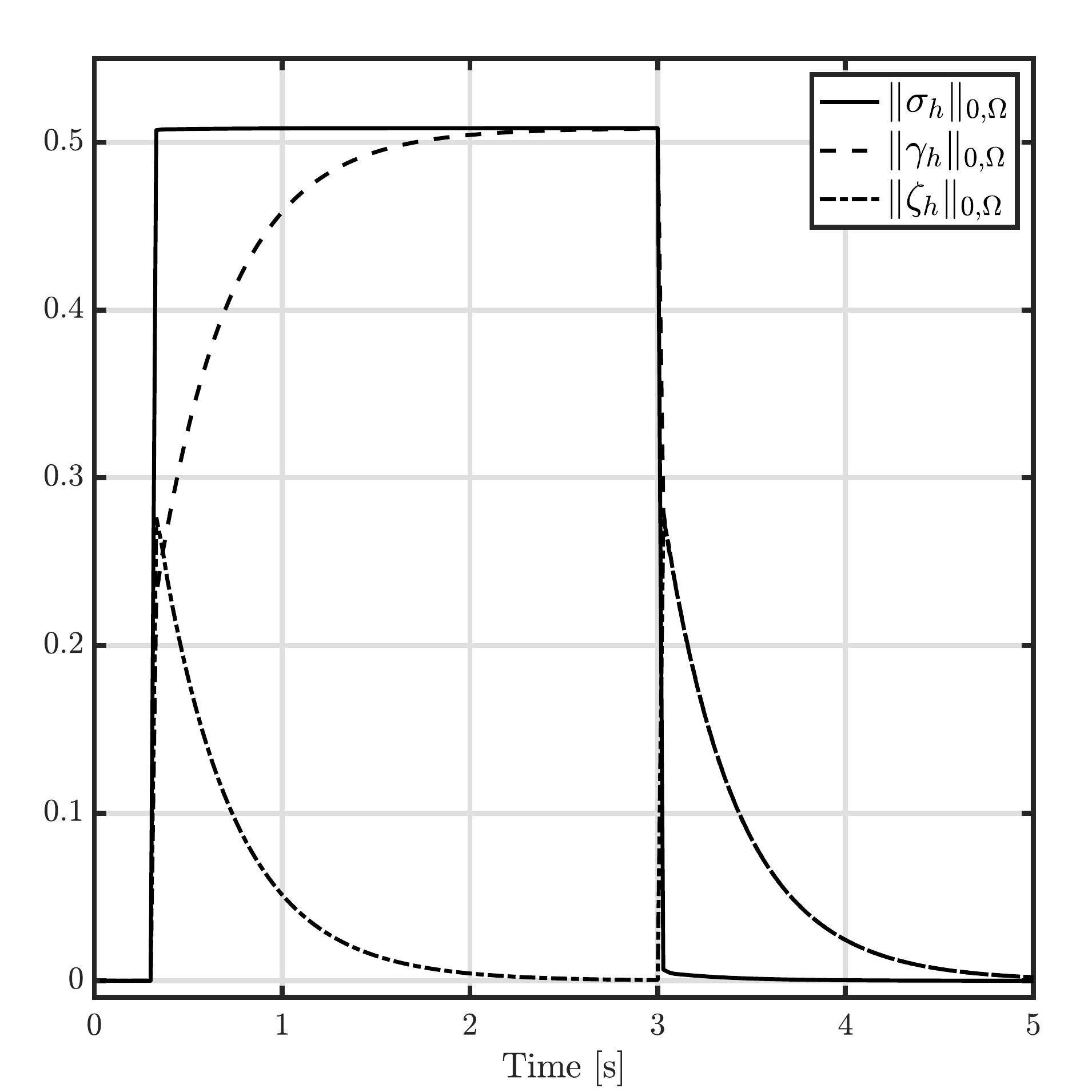}
\includegraphics[width=0.245\textwidth]{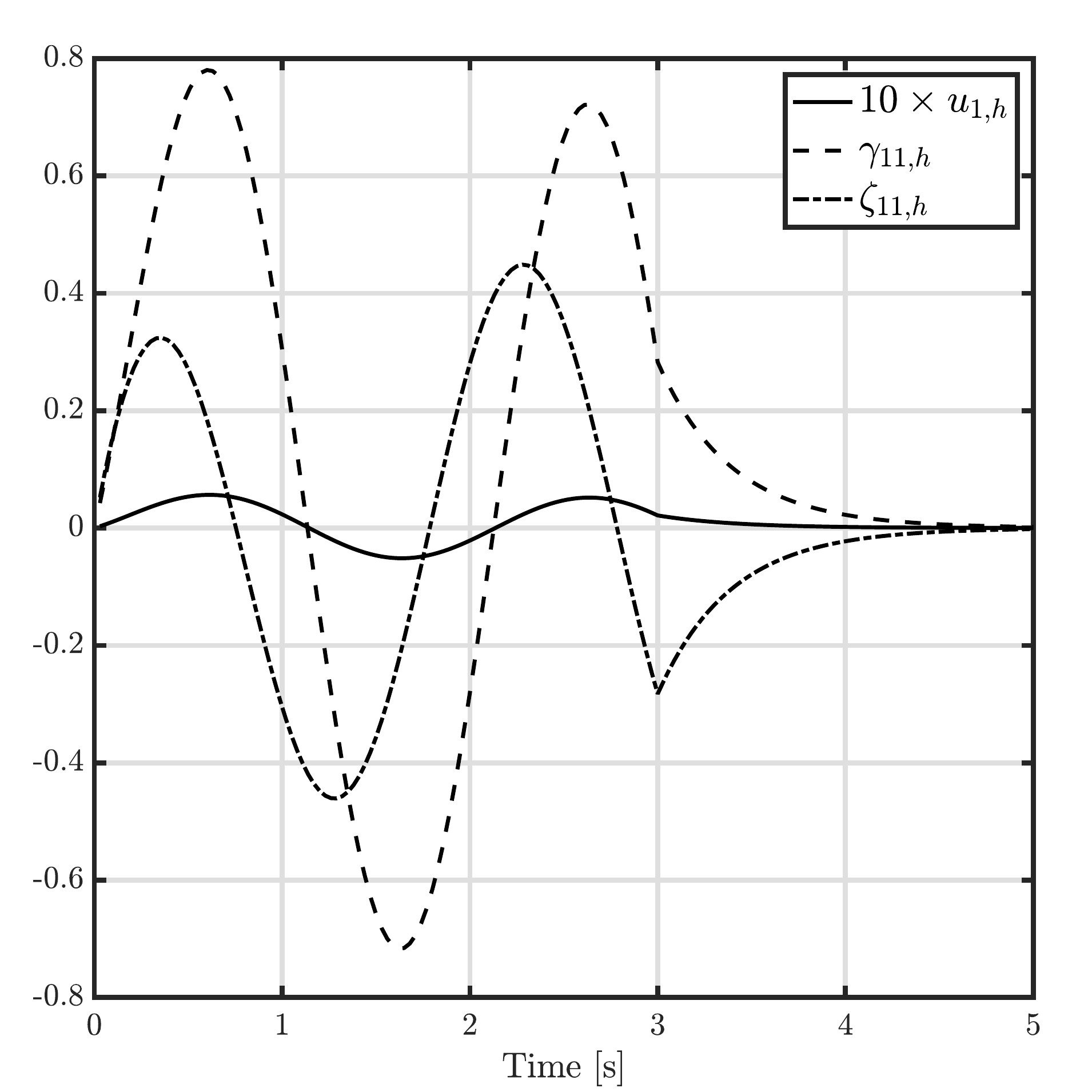}
\includegraphics[width=0.245\textwidth]{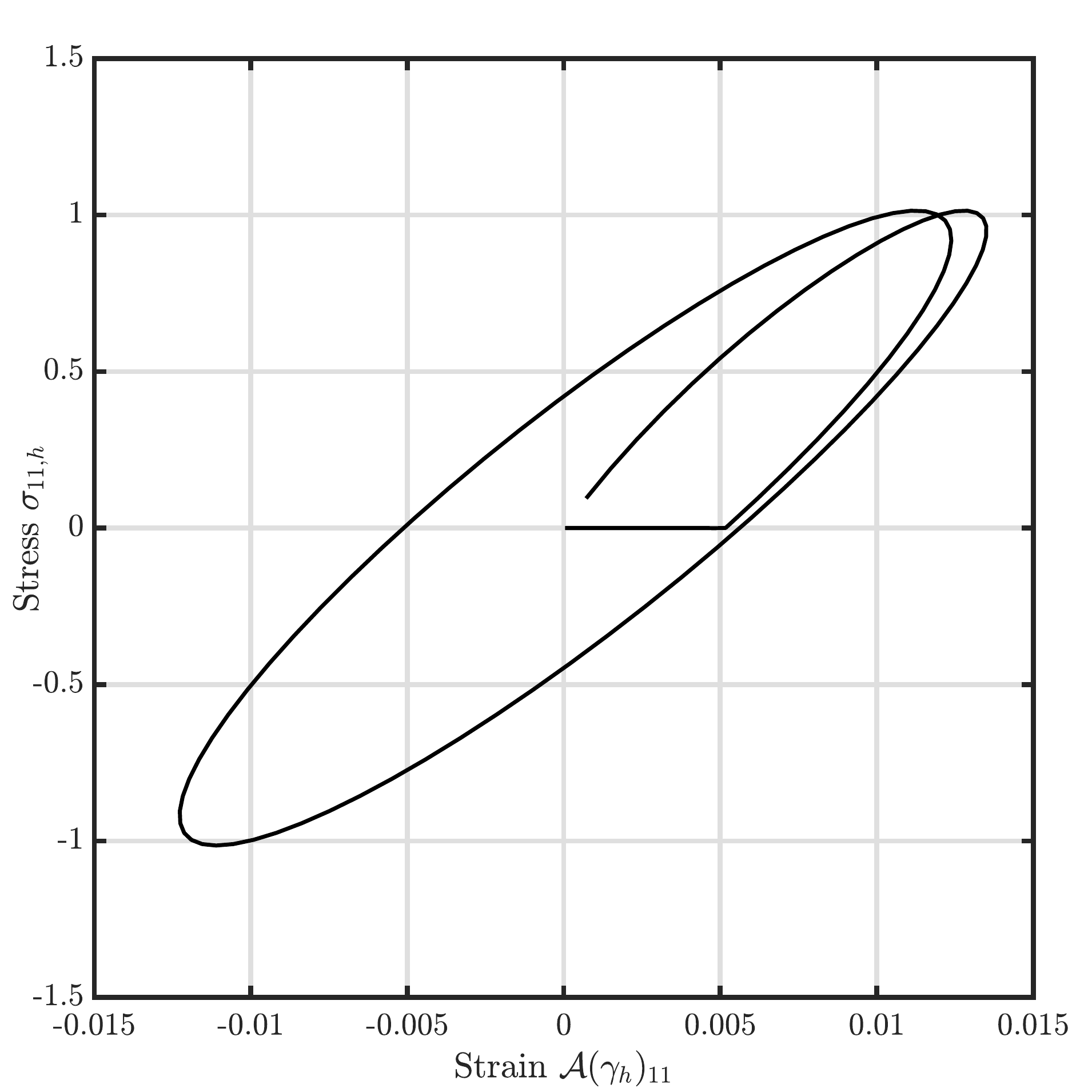}
\end{center}

\vspace{-5mm}
\caption{Example 3. Snapshots at 10x displacement magnification taken at times  0.33\,s, 3\,s, 3.6\,s (left, middle, right) of displacement (top) and distribution of total stress (middle row) for the creep of a viscoelastic slab. Solutions obtained with $k=1$. Bottom left: transients of $L^2-$norms of displacement, and elastic, viscous, and total  stresses  for the case of creep. Bottom right: transients at the point (0.5,0.25,0.25) of $x-$displacement (magnified) and $xx-$stresses; and stress-strain curve for the cyclic loading case.}\label{fig:creep}
\end{figure}

\medskip
\noindent\textbf{Example 3: Creep  and cyclic loading of a viscoelastic slab.} The model and the implementation are tested on a 3D scenario by computing numerical solutions of Zener's model on the domain $\Omega = (0,1)\times(0,\frac12)\times(0,\frac12)$\,m$^3$. The  following Lam\'e constants are employed 
 \[ 
 \mu_{\mathcal{C}} = 20\, \text{Pa}, \quad \lambda_{\mathcal{C}} = 100\, \text{Pa},\quad 
 \mu_{\mathcal{D}}= 50\, \text{Pa}, \quad \lambda_{\mathcal{D}} = 200\, \text{Pa}.
 \]
We run different tests associated with response to applied traction. Firstly creep (a constant normal stress is imposed over a short period and then released) and then with cyclic loading (the applied traction is periodic, generating an oscillatory deformation pattern). For the first case, an instantaneous traction (in the $x$ direction) of intensity 1\,Pa is applied on the sub-boundary at $x=1$ at $t = 0.3$\,s during   2.7\,s and then it is suddenly released. The boundary located at $x=0$ is maintained clamped, and the remaining parts of the boundary are considered stress-free. 
The test is run until $t = 5$\,s and the behaviour of the different stress components is plotted in the bottom-left panels of Figure~\ref{fig:creep}. The obtained profiles show the evolution of the $L^2$-norms of the approximate viscous and elastic stress as well as of postprocessed displacements. These results are qualitatively comparable to the expected behaviour shown in, e.g., \cite[Section 6.2.2 and Figure 6]{rognes} for 2D tests with the standard linear model and without inertia (that is, an instantaneous displacement increase at $t=0.3$\,s, and the decrease of viscous stress and increase of elastic stress needed to maintain a constant total stress, all of them eventually decaying with time). For the case of a slab subjected to cyclic loading, we apply the traction $(0.5\sin(\pi t)H(t\leq 3),0,0)^{\tt t}$ on the sub-boundary at $x=1$, and record in the bottom-right plots of Figure~\ref{fig:creep} the evolution of $x-$displacement and $xx-$stresses, and we also plot the stress--strain curve (where the strain tensor is accessed through the inverse constitutive equation $\cA\bgam = \beps(\bu)$) at the midpoint of the domain, exhibiting the typical viscoelastic behaviour. The remaining parameter values are $\rho = 10^{-3}$\,Kg/m$^3$, $\texttt{a}^* = 15$, $\omega = 1/6$\,s. We use a time step of $\Delta t = 0.03$\,s and a structured tetrahedral mesh with 20'736 elements, which, for $k=1$,  represents 995'328 DoF. 


\begin{figure}[t!]
\begin{center}
\raisebox{5mm}{\includegraphics[width = 0.24\textwidth]{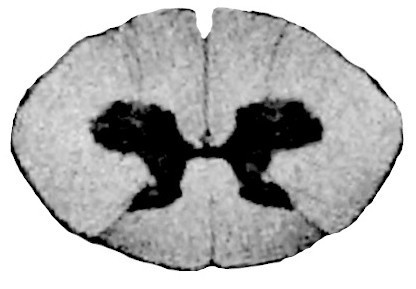}}
\includegraphics[width = 0.26\textwidth]{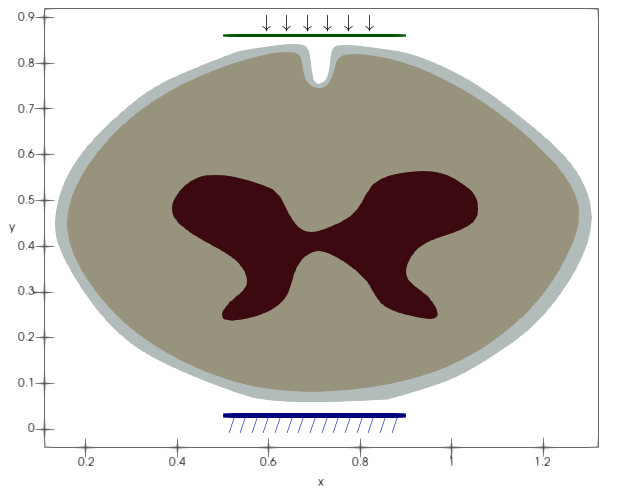}
\raisebox{4mm}{\includegraphics[width = 0.24\textwidth]{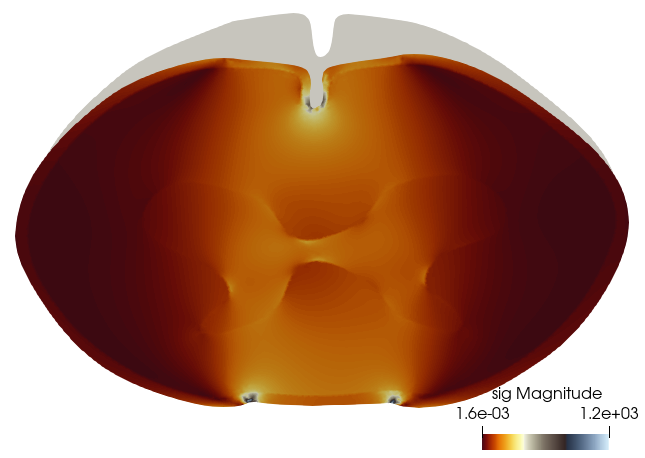}}
\raisebox{4mm}{\includegraphics[width = 0.24\textwidth]{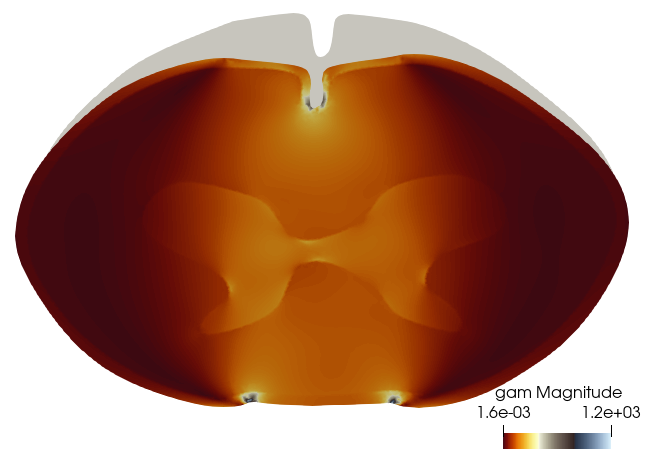}}\\
\includegraphics[width = 0.245\textwidth]{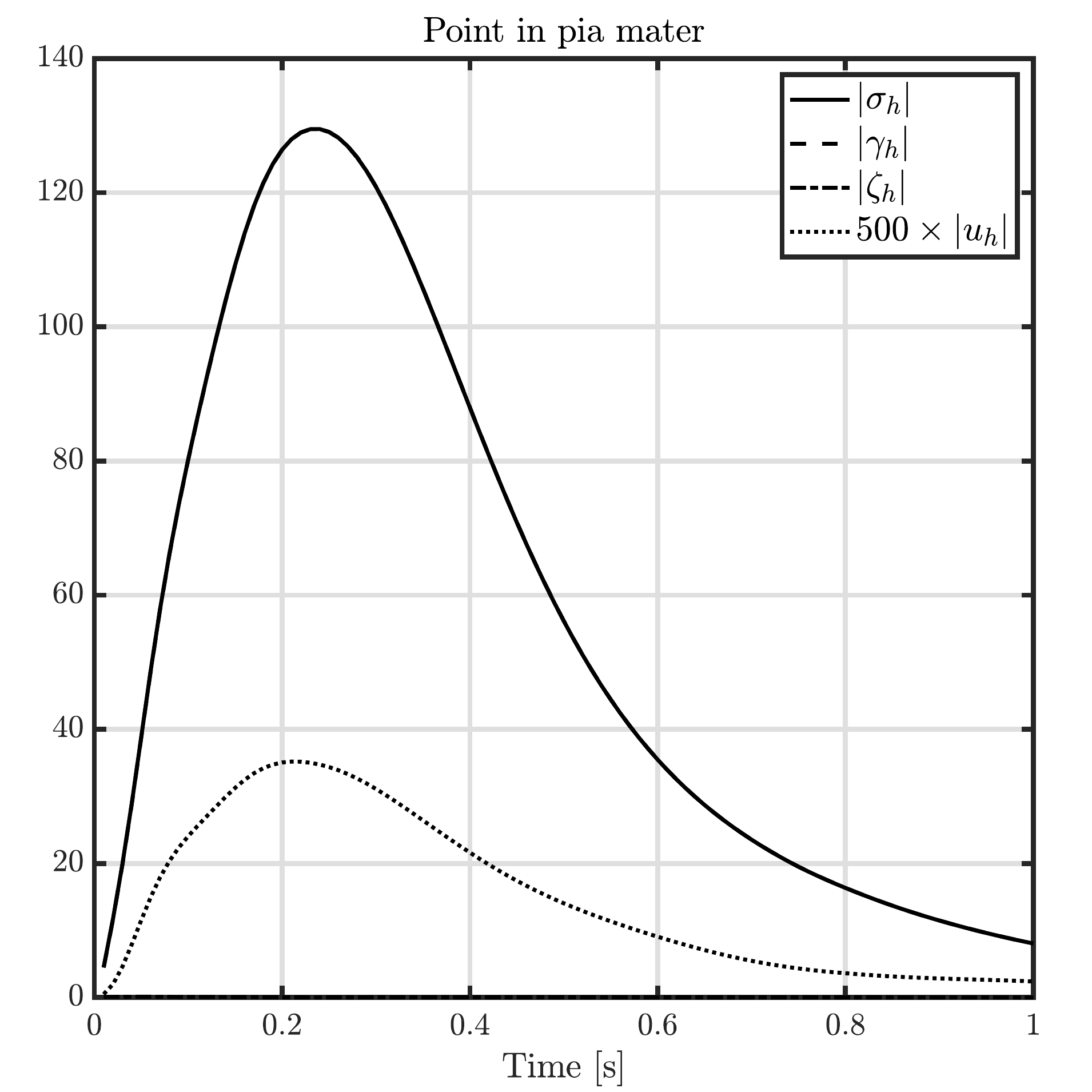}
\includegraphics[width = 0.245\textwidth]{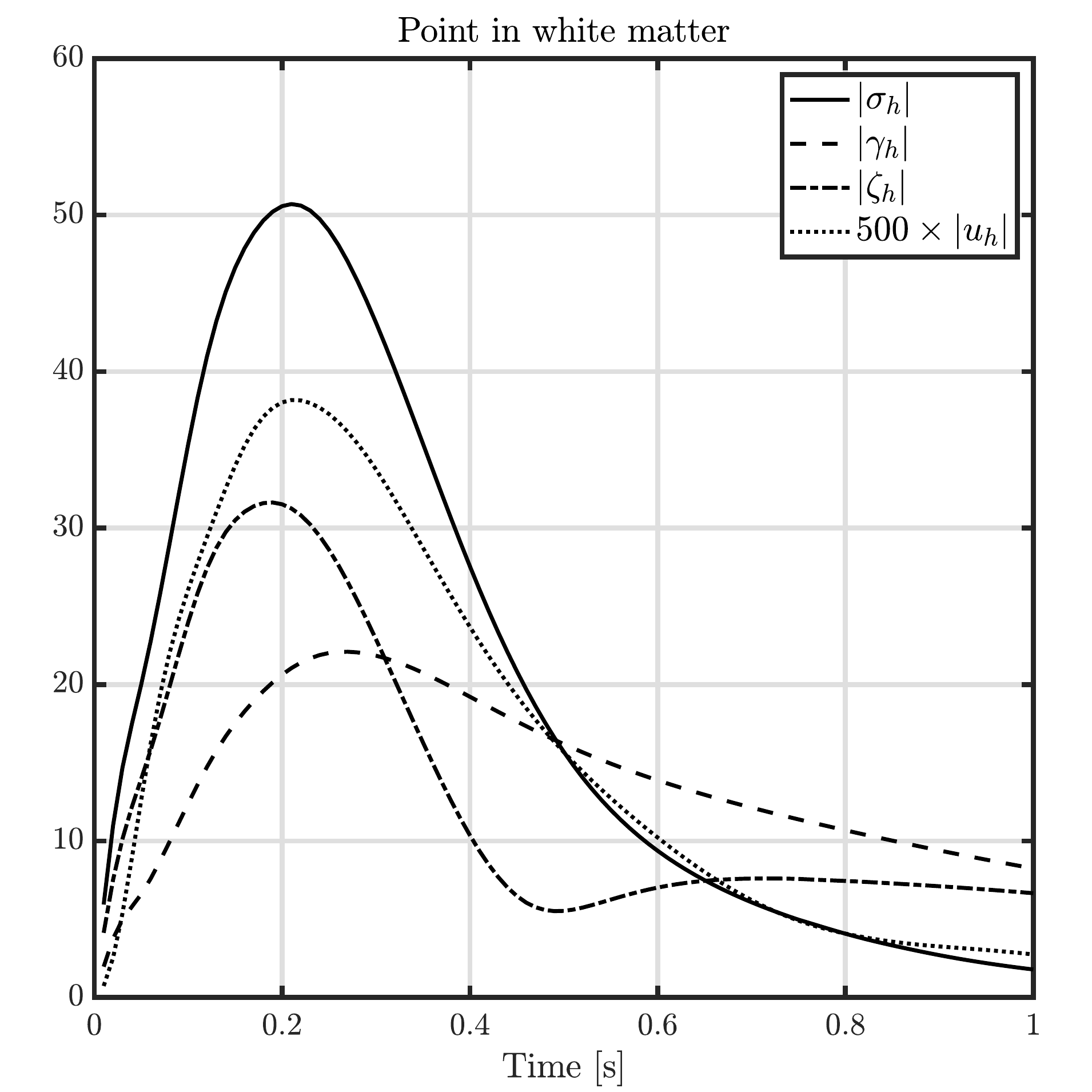}
\includegraphics[width = 0.245\textwidth]{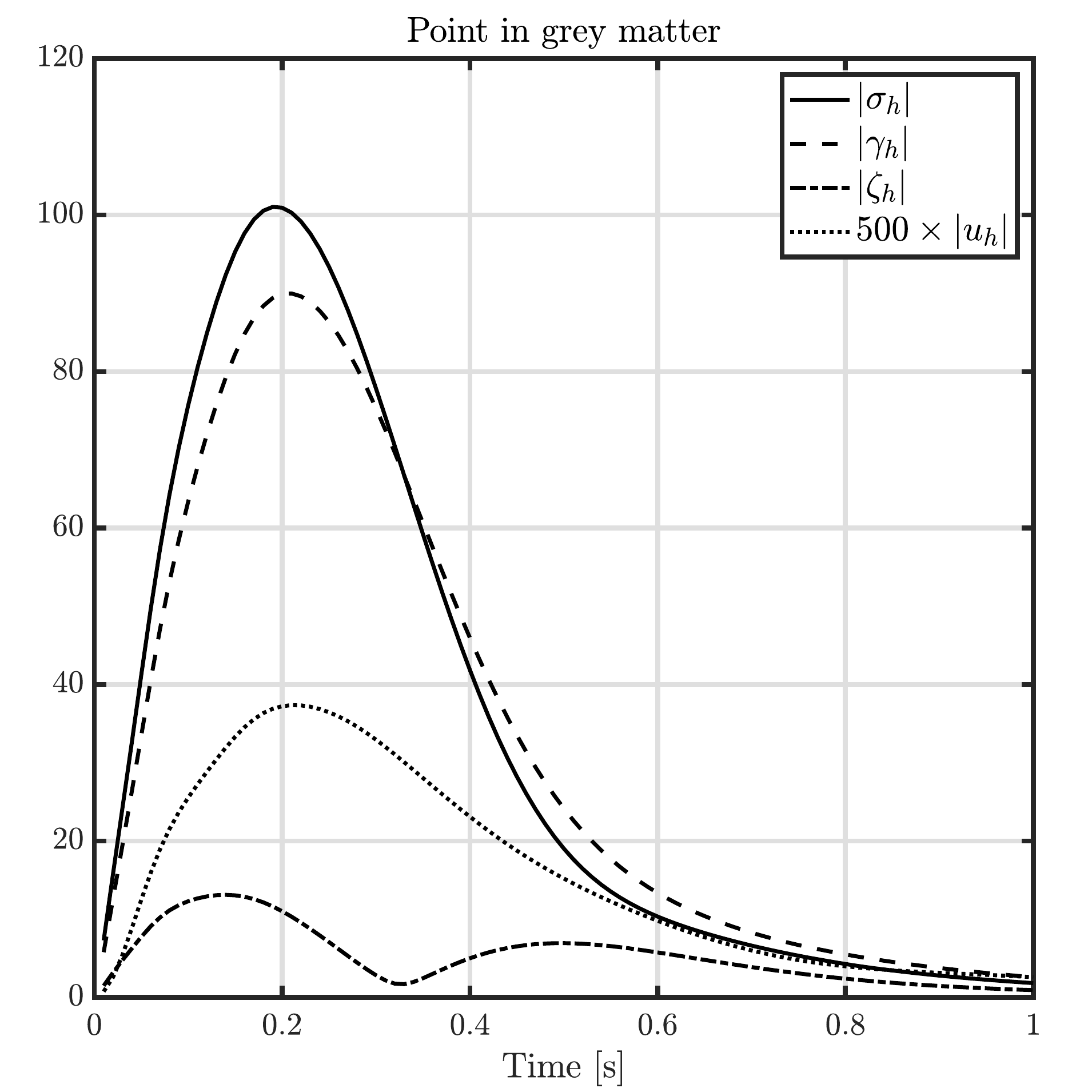}
\includegraphics[width = 0.245\textwidth]{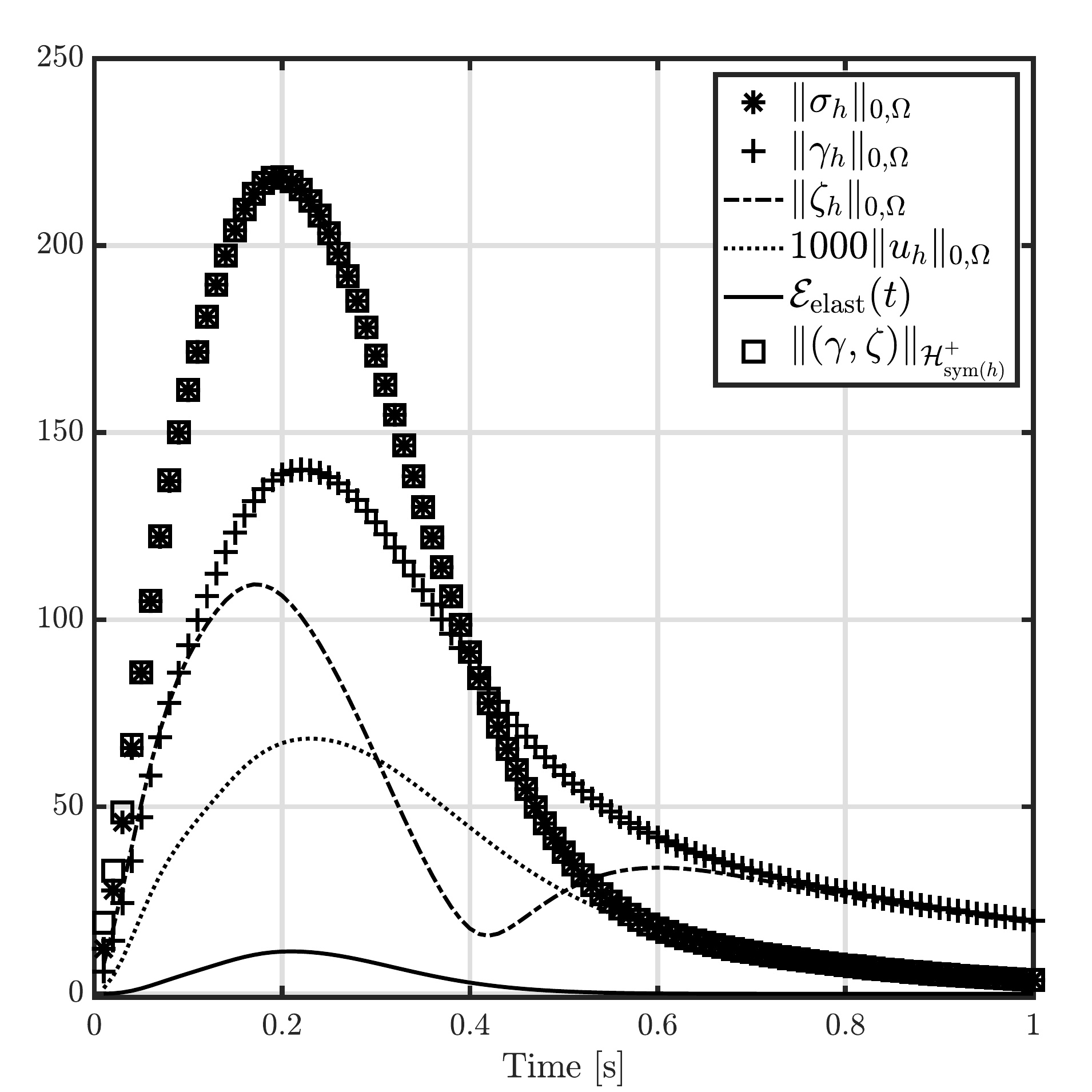}
\end{center}

\vspace{-5mm}
\caption{Example 4. Top left: Transversal cross-section of a cervical spinal cord with three material layers (from outer to inner: pia mater, white matter, grey matter), and schematic representation of front-to-back loading indicating regions where the outer layer is clamped (posterior, bottom), and where it undergoes an indentation   (anterior, top). Axes units are in cm. Top right: total and elastic stress after indentation during $T=0.35$\,s. 
Bottom: Transients of viscous and elastic stresses at  points in the pia mater, white matter, and grey matter. Bottom right: computed $L^2$-norms, energy $\mathcal{E}_{\mathrm{elast}}$, and stress energy norm.}\label{fig:spinal} 
\end{figure}
 
\medskip
\noindent\textbf{Example 4: Localisation of viscous stresses in a multilayered cross-section of spinal cord.} To conclude this section we consider a simple viscoelastic model for a segment of cervical spinal cord (consisting of white and grey matter), surrounded by the pia mater (represented as a thin layer of elastic material). The problem setup mimics indentation tests as in, e.g., \cite{rycman21,zhu20}, which in turn replicate problems arising due to degenerative factors. 
We only take a transversal cross-section of approximately 13\,mm in maximal diameter, and in this case the indentation region is simply a curved subset of the anterior part of the pia mater, having length 4\,mm.  The geometry and unstructured mesh have been generated from the  images in \cite{stover16} using the mesh manipulator GMSH \cite{gmsh}. In this region we will impose, as in the previous tests,  a traction $(\bgam+\bze)\bn = (0,-P)^{-\tt t}$, with $P$ a given time-dependent pressure profile with maximal amplitude 650\,Pa. The posterior part of the pia mater (a sub-boundary of length 4\,mm) will be considered as a rigid posterior support and therefore zero displacement boundary conditions will be prescribed. The remainder of the boundary (of the pia mater) is taken as stress-free (see the sketch in Figure~\ref{fig:spinal}, top left).

An advantage of the DG-based formulation advanced herein is that it permits us to readily consider discontinuous material parameters. 
For the three different layers of the domain we use the following values for   Young modulus and Poisson ratio (values from \cite{kylstad,stover16,zhu20}, see also \cite{klatt,rycman21}), 
\begin{gather*}
E^{\mathrm{pia}} = 2300\,\text{Pa}, \quad \nu^{\mathrm{pia}} = 0.3, \quad  E^{\mathrm{white}}_{\mathcal{C}} = 840\, \text{Pa}, \quad \nu^{\mathrm{white}}_{\mathcal{C}} = 0.479, \quad  
 E^{\mathrm{white}}_{\mathcal{D}}= 2030\, \text{Pa},\quad    \nu^{\mathrm{white}}_{\mathcal{D}} = 0.49,\\
 \quad E^{\mathrm{grey}}_{\mathcal{C}} = 1600\, \text{Pa}, \quad \nu^{\mathrm{grey}}_{\mathcal{C}} = 0.49 , \quad 
 E^{\mathrm{grey}}_{\mathcal{D}}= 2030\, \text{Pa},  \quad \nu^{\mathrm{grey}}_{\mathcal{D}} = 0.49. 
\end{gather*}
Note that in \cite{zhu20} the pia matter is considered elastic and the white and grey matter subdomains are considered hyperelastic, in \cite{rycman21} there is only pia mater and homogeneous spinal cord (all visco-hyperelastic), in \cite{kylstad}  the spinal cord is homogeneous and linear viscoelastic, whereas in \cite{stover16} a poroelastic model has been used for all layers. Here the elastic behaviour of the pia mater is modelled with a  much smaller value than in the rest of the domain, $\omega^{\mathrm{pia}} = 1/1000\,\text{s} < \omega^{\mathrm{white,grey}} =  1/6.7$\,s).  A fixed time step $\Delta t = 0.01$\,s is used and we run the simulation until $T=1$\,s.  The top-centre and top-right panels of Figure~\ref{fig:spinal} shows a sample of deformed configuration and distribution of total stress and elastic stress magnitude at time $t=0.35$\,s. 
For this problem we investigate numerically the   decay  of the elastic energy $\frac12 \int_\Omega \bsig:\beps(\bu) $, which, using the definition of viscous and elastic stress contributions, can be written as 
\[\mathcal{E}_{\mathrm{elast}}(t)=\frac12 \int_\Omega\mathcal{A}(\bgam): (\bgam+\bze).\] 
In addition, we also track the value of the energy norm defined in \eqref{norm:sym}.  These quantities are plotted in the middle and bottom rows of Figure~\ref{fig:spinal} together with transients of the principal stresses and displacements at three different locations in the domain layers.  The deformation vs load, as well as the stress distribution on the white and grey matter regions is qualitatively consistent with the behaviour reported in \cite{zhu20}. In the pia mater, as expected, only the elastic stress is visible.

\medskip 
\noindent\textbf{Acknowledgement.} We are thankful to Prof. Kent-Andr\'e Mardal for pointing out model parameters and data to use in Example 4. 



\end{document}